\documentclass[10pt,preprint]{article}
\usepackage[width=16cm,height=20cm]{geometry}
\usepackage{graphicx}
\usepackage{amssymb}
\usepackage{amsmath}
\usepackage{amsthm}
\usepackage{bm}
\usepackage{color}
\usepackage[bookmarks]{hyperref}
\usepackage{lineno}
\usepackage{todonotes}
\usepackage{aliases}
\usepackage{comment}
\usepackage[final]{showlabels}
\usepackage{tikz}
\usetikzlibrary{shapes,positioning,quotes,arrows,chains,calc}

\hypersetup{
    colorlinks=true, 
    linkcolor=blue, 
    urlcolor=red, 
    citecolor=magenta,
    linktoc=all 
}

\newcommand{\vel}{\bm{v}}
\newcommand{\p}{\mathcal{P}}
\newcommand{\R}{\mathcal{R}}
\newcommand{\Q}{\mathcal{Q}}
\newcommand{\E}{\mathcal{E}}

\newcommand{\con}{\bm{U}}

\newcommand{\Old}[1]{} 

\newcommand{\BNK}[2]{\textcolor{magenta}{{#1}}}

\newcommand{\prim}{\bm{Q}}
\def\lf{{\text{\sc l}}}
\def\rg{{\text{\sc r}}}
\def\En{E}
\def\Pres{\text{\sc p}}

\newcommand{\conl}{\con_{\lf}}
\newcommand{\conr}{\con_{\rg}}

\newcommand{\f}{\bm{F}}
\newcommand{\fx}{\bm{F}_1}
\newcommand{\fy}{\bm{F}_2}

\newcommand{\m}{\bm{m}}
\newcommand{\s}{\bm{S}}
\newcommand{\D}{\bm{D}}

\newcommand{\B}{\bm{B}}

\newcommand{\Diss}{\mathcal{D}}
\newcommand{\A}{\bm{A}}

\newcommand{\dt}{\Delta t}
\newcommand{\dx}{\Delta x}

\def\bt{\text{b}} 
\def\coefa{\text{\sc a}} 
\def\coefb{\text{\sc c}} 
\def\newInv{\beta}
\newcommand{\uad}{\mathcal{U}_{\textrm{ad}}}
\newcommand{\detP}{\textrm{det}(\p)}

\newcommand{\hz}{\hat{z}}
\newcommand{\tz}{\tilde{z}}

\newcommand{\Afun}{\color{red}\mathfrak a\color{black}}
\newcommand{\Bfun}{\color{red}\mathfrak b\color{black}}

\newtheorem{theorem}{Theorem}
\newtheorem{lemma}{Lemma}

\newtheorem{remark}{Remark}

\title{Exact solution for Riemann problems of the shear shallow water model}
\date{}

\begin{document}
\author{Boniface Nkonga\footnote{Universit\'e C\^ote d'Azur, INRIA, CNRS and LJAD, 06108 Nice Cedex 2, France. email: {\tt boniface.nkonga@unice.fr}} \ and Praveen Chandrashekar\footnote{Centre for Applicable Mathematics, Tata Institute of Fundamental Research, Bangalore--560065, India. email: {\tt praveen@math.tifrbng.res.in}}}


\maketitle
\begin{abstract}
The shear shallow water model is a higher order model for shallow flows which includes some shear effects that are neglected in the classical shallow models. The model is a non-conservative hyperbolic system which can admit shocks, rarefactions, shear and contact waves. The notion of weak solution is based on a path but the choice of the correct path is not known for this problem. In this paper, we construct exact solution for the Riemann problem assuming a linear path in the space of conserved variables, which is also used in approximate Riemann solvers. We compare the exact solutions with those obtained from a path conservative finite volume scheme on some representative test cases.
\end{abstract}
\section{Introduction}
In the present paper we investigate the solutions of Riemann problems for a non-linear, non-conservative hyperbolic system of equations arising in the modeling of shear shallow water (SSW) flows. In the framework of non-conservative hyperbolic systems, the notion of weak solutions and associated jump conditions need to be revisited. Indeed, in this context, we have to deal with non-classical multiplication of distributions that prevent unique derivation of jump conditions. The path-conservative approach is now a useful tool for numerical approximation of non-conservative hyperbolic systems. The main principle behind this approach is to define the weak solution by assuming some path between two states and derive generalized jump conditions.  The paper~\cite{Volpert1967} is the first to formulate a meaning to non-conservative products using Borel measures. In~\cite{DalMaso1995}, the notion of path is introduced and generalizes the results of~\cite{Volpert1967}. The first numerical applications resulting from these theoretical analyses are realized in~\cite{Toumi1992} in the context of Roe scheme for real gases and two-phase flow model, and was  generalized under the designation of ``path conservative methods" in~\cite{Pares2006}. Since then, the ``path conservative methods" have been widely applied for the numerical solution of non-conservative hyperbolic problems~\cite{Dumbser2009,Castro2012,Dumbser2016,CastroDiaz2019,Schneider2021}. Nevertheless, contrary to the Lax-Wendroff theorem~\cite{Lax1960} for conservative hyperbolic systems, there is no adequate mathematical theory that can ensure the numerical convergence for any non-conservative system. In the presence of discontinuities, numerical approximations may not converge to the specified entropic weak solution.  The equivalent equation of a path conservative scheme based on Lax-Friedrich scheme is examined in~\cite{Castro2008a}. Quoting from~\cite{Castro2008a}, {\em the difficulty comes from the fact that, unlike the conservative case, the vanishing viscosity limits depend on the regularization of the problem. Even if, for simplicity, we have only calculated the modified equations corresponding to the Lax–Friedrichs scheme, the same difficulty would be present for any other scheme involving a numerical viscous term: the numerical solutions approximate the vanishing viscosity limit of a modified equation whose regularization terms depend both on the chosen family of paths and on the specific form of its viscous terms.} Therefore, there is always a doubt about the ability of numerical strategies to produce relevant numerical solutions that can converge, by mesh refinement, towards a single limit solution. In~\cite{Abgrall2010}, the following question was pointed out: {\em once a path is specified and a consistent path-conservative scheme designed, does the numerical solution converge to the assumed path}. They also point out that in some contexts there is clearly a failure of convergence upon grid refinement.  In order to clarify the questions that arise in numerical simulations of non-conservative hyperbolic problems, it is necessary to construct exact solutions for a fixed path.  From there, we can use the same path in a numerical approach and study if we have a convergence of the numerical solution to the analytical solution.  Let us note that attempts to answer these questions exist in the literature for a model of elastodynamics described by a $2 \times 2$ non-hyperbolic system.  For this model, theoretical analyses and numerical investigations are proposed \cite{Colombeau1988,Cauret1989}, and we even have exact solutions for the Riemann problem \cite{Joseph2003}.

In this paper, we consider an example of a non-conservative hyperbolic system, the shear shallow water model, for which approximate Riemann solver based methods have been developed in the literature~\cite{Gavrilyuk2018,Bhole2019,Chandrashekar2020}  and for which we construct the exact Riemann solution in this work.  Riemann solvers are an important building block of modern numerical schemes for hyperbolic systems. Therefore, there can be some confidence when using this approach for more complex data setting~\cite{Bhole2019, Gavrilyuk2018, Chandrashekar2020}. In the coming sections we will first describe the equations for shear shallow water flows written in a specific non-conservative form. This formulation uses the set of quasi-conservative variables which is very similar to the 10-moment equations of gas dynamics~\cite{Levermore1998}, but the system is genuinely non-conservative. Then the path-conservative jump conditions are recalled and used to derive an exact solution of a Riemann problem. Finally, we discuss the convergence of the numerical solution obtained from a path conservative approximate Riemann solver~\cite{Chandrashekar2020} toward the designed exact solution for some representative test cases.
\section{The SSW model}
The system describing multi-dimensional shear shallow water flow was derived by Teshukov in 2007~\cite{Teshukov2007} by depth averaging the incompressible Euler equations. This system of equations describes the evolution of the fluid depth $h$, the depth averaged horizontal velocity $\vel$ and the Reynolds tensor $\p$, and can be written as~\cite{Gavrilyuk2018}
\begin{eqnarray}
\nonumber
\df{h}{t} + \nabla\cdot(h \vel) &=&0 \\
\label{eq:ssw1}
\df{(h\vel)}{t} + \nabla\cdot\left( h \vel \otimes \vel + \half g h^2 I + h \p \right) &=& -gh \nabla\bt - C_f |\vel| \vel \\
\nonumber
\df{\p}{t} + \vel\cdot\nabla\p + (\nabla\vel)\p + \p (\nabla\vel)^\top &=& \Diss
\end{eqnarray}
The tensor $\p$ is symmetric and positive definite; it measures the distortion of the instantaneous horizontal velocity with respect to the depth average velocity $\vel$. The system derived in~\cite{Teshukov2007} was non-dissipative ($C_f = 0$, $\Diss = 0$); in~\cite{Gavrilyuk2018}, the modeling of dissipation process was introduced for the evolution of the momentum and the Reynolds stress tensor. The dissipation model provides a closure to the averaging process and was designed such as to preserve the positive definite-ness of the tensor $\p$. Recently~\cite{Chandrashekar2020}, the dissipative model proposed in~\cite{Gavrilyuk2018} has been reformulated for the evolution of the energy tensor $\En$. In this context, the SSW model can be written in an almost conservative form. To do this, we define the symmetric tensors
\[
\R_{ij} := h \p_{ij}, \qquad \E_{ij} := \half \R_{ij} + \half h v_i v_j, \qquad  1 \le i,j \le 2
\]
Then, the set of equations for the SSW model~\eqref{eq:ssw1} can be written as follows
\begin{equation}
\label{eq:tenmom}
\df{\con}{t} + \df{\fx}{x_1} + \df{\fy}{x_2} + \B_1 \df{h}{x_1} + \B_2  \df{h}{x_2} =  \s
\end{equation}
where
\[
\con = \begin{bmatrix}
h \\
h v_1 \\
h v_2 \\
\E_{11} \\
\E_{12} \\
\E_{22} \end{bmatrix}, \quad
\fx = \begin{bmatrix}
h v_1 \\
\R_{11} + h v_1^2 + \half g h^2 \\
\R_{12} + h v_1 v_2 \\
(\E_{11} + \R_{11}) v_1 \\
\E_{12} v_1 + \half (\R_{11} v_2 + \R_{12} v_1) \\
\E_{22} v_1 + \R_{12} v_2 \end{bmatrix}, \quad
\fy = \begin{bmatrix}
h v_2 \\
\R_{12} + h v_1 v_2 \\
\R_{22} + h v_2^2 + \half g h^2 \\
\E_{11} v_2 + \R_{12} v_1 \\
\E_{12} v_2 + \half (\R_{12} v_2 + \R_{22} v_1) \\
(\E_{22} + \R_{22}) v_2 \end{bmatrix}
\]
\[
\B_1 = \begin{bmatrix}
0 \\
0 \\
0 \\
g h v_1 \\
\half g h v_2 \\
0 \end{bmatrix}, \qquad
\B_2 = \begin{bmatrix}
0 \\
0 \\
0 \\
0 \\
\half g h v_1 \\
g h v_2 \end{bmatrix}, \qquad
\s = \begin{bmatrix}
0 \\
-g h \df{\bt}{x_1} - C_f |\vel| v_1 \\
-g h \df{\bt}{x_2} - C_f |\vel| v_2 \\
- g h v_1 \df{\bt}{x_1} + \half h \Diss_{11} - C_f |\vel| v_1^2 \\
- \half g h v_2 \df{\bt}{x_1} - \half g h v_1 \df{\bt}{x_2} + \half h \Diss_{12}  - C_f |\vel| v_1 v_2 \\
- g h v_2 \df{\bt}{x_2} + \half h \Diss_{22}  - C_f |\vel| v_2^2
\end{bmatrix}
\]
In the present work, we assume that the bottom topography $\bt\equiv \bt\left( x_1, x_2\right)$ is a given  smooth function. The solution must satisfy some positivity constraints which leads to the following solution space for physically admissible solutions
\[
\uad = \{ \con \in \re^6 : h > 0, \quad \R > 0 \}
\]
where $\R > 0$ means that the symmetric tensor $\R$ must be positive definite. We next consider some properties of this model.
\subsection{Total energy equation}
The additional conservation laws satisfied by the SSW model have been investigated in~\cite{Gavrilyuk2018}. The first one is related to the energy and can be derived as follows. Multiply $h$ equation by $g(h+\bt)$ and add it to the $\E_{11}$ and $\E_{22}$ equations to obtain
\begin{equation} \label{eq:toteeq}
\begin{aligned}
\df{E}{t} + \df{}{x_1}\left[ \left(E + \R_{11} + \half g h^2 \right) v_1 + \R_{12} v_2 \right] +& \df{}{x_2}\left[ \left(E + \R_{22} + \half g h^2 \right) v_2 + \R_{12} v_1 \right]  \\
=&  - C_f  |\vel|^3 + \half h \trace(\Diss)
\end{aligned}
\end{equation}
where the {\em total energy} is defined as
\begin{equation}\label{eq:totE}
E = \E_{11} + \E_{22} + \half g h^2 + g h \bt = \half \trace(\R) + \half h |\vel|^2 + \half g h^2 + g h \bt
\end{equation}
The quantity $E=E(\con)$ is a convex function but it is not a strictly convex function since it has no dependence on $\E_{12}$, and so it cannot serve as an entropy function.
\subsection{Entropy equation}
We can define the specific entropy
\begin{equation}
s = \frac{\det \p}{h^2}
\label{eq:spent}
\end{equation}
which satisfies the equation (\cite{Gavrilyuk2018}, Eq. 31)
\[
\df{s}{t} + \vel \cdot \nabla s  = \frac{1}{h^2}[ \trace(\p) \trace(\Diss) - \trace(\p\Diss)]
\]
The above equation can be rewritten as an entropy balance law,
\[
\df{\eta}{t} + \nabla \cdot (\vel \eta) =  -\frac{1}{hs} [ \trace(\p) \trace(\Diss) - \trace(\p\Diss)]
\]
where
\[
\eta = \eta(\con) = -h \log s = - h \log \left( \frac{\det \p}{h^2} \right)
\]
is a convex function of $\con$~\cite{Levermore1996}. Smooth solutions in the absence of dissipation $\Diss$ satisfy the entropy conservation law. In general, when the solution is not smooth, we require an entropy inequality
\[
\df{\eta}{t} + \nabla \cdot (\vel \eta) \le 0
\]
to hold in the sense of distributions. For a scalar problem, the entropy condition serves to enforce uniqueness of weak solutions but this is an open problem for systems of conservation laws. However, it is important to satisfy the entropy condition since it is a fundamental property of all natural systems. The availability of such an entropy condition for the SSW model~\eqref{eq:tenmom} indicates that it can serve as a useful mathematical form for the construction of numerical schemes.
\begin{remark}
For a different but related PDE model for shear shallow flows, we refer the reader to~\cite{Busto2021} where a new matrix variable $\Q$ is introduced such that $\p = \Q \Q^\top$. Unlike $\p$, the matrix $\Q$ is not assumed to be symmetric which introduces an extra variable into the model. An equation for $\Q$ is derived under some simplifying assumptions on the rotation of Reynolds tensor by friction forces, whose evolution ensures positivity of $\p$. The total energy $E$ becomes a convex function in terms of the new set of variables $(h, h\vel, h\Q)$, leading to a thermodynamically consistent model. A numerical approach based on path conservative idea is developed, which under the assumption of exact integration of some quantities, leads to a first order semi-discrete scheme which is shown to conserve the total energy, is consistent with the entropy inequality and with the vanishing viscosity limit of the model. In order to ensure the conservation of total energy in the inviscid case for the fully discrete scheme, a scaling of the variable $\Q$ is performed after each time step.
\end{remark}
\subsection{Hyperbolicity}
We will consider the 1-D SSW model which can be written as
\begin{equation}
\label{eq:ssw1d}
\df{\con}{t} + \df{\f(\con)}{x} + \B(\m) \df{h}{x} = \s(\con)
\end{equation}
where $\f = \fx$, $\B = \B_1$, $\m = h \vel$, and the source term is given by
\[
\s = \begin{bmatrix}
0 \\
-g h \df{\bt}{x} - C_f |\vel| v_1 \\
 - C_f |\vel| v_2 \\
-\alpha |\vel|^3 \p_{11} - g h v_1 \df{\bt}{x} - C_f |\vel| v_1^2 \\
-\alpha |\vel|^3 \p_{12} - \half g h v_2 \df{\bt}{x}  - C_f |\vel| v_1 v_2 \\
-\alpha |\vel|^3 \p_{22}  - C_f |\vel| v_2^2
\end{bmatrix}
\]
Ignoring the source term in~\eqref{eq:ssw1d} for the moment as they do not contain derivatives of $\con$, let us write the non-conservative system~\eqref{eq:ssw1d} in quasi-linear form as
\begin{equation} \label{eq:ssw1db}
\df{\con}{t} + \A(\con) \df{\con}{x} = 0, \qquad \A = \f'(\con) +
\begin{bmatrix}
0 & 0 & 0 & 0 & 0 & 0 \\
0 & 0 & 0 & 0 & 0 & 0 \\
0 & 0 & 0 & 0 & 0 & 0 \\
g h v_1 & 0 & 0 & 0 & 0 & 0 \\
\half g h v_2 & 0 & 0 & 0 & 0 & 0 \\
0 & 0 & 0 & 0 & 0 & 0
\end{bmatrix}
\end{equation}
For simplicity of notation, we will sometimes write the velocity components as  $(u,v)= (v_1,v_2)$. The system of equations~\eqref{eq:ssw1db} is a hyperbolic system with eigenvalues of $\A$ being given by~\cite{Gavrilyuk2018,Berthon2002}
\[
\lambda_1 = u - \sqrt{gh + 3 \p_{11}}, \quad \lambda_2 = u - \sqrt{\p_{11}},    \quad \lambda_3 = \lambda_4 = u, \quad \lambda_5 = u+ \sqrt{\p_{11}}, \quad     \lambda_6 = u + \sqrt{g h + 3 \p_{11}}
\]
The first and last eigenvalues correspond to genuinely non-linear               characteristic fields in the sense of Lax~\cite{Godlewski1996}, while the       remaining eigenvalues correspond to linearly degenerate characteristic          fields~\cite{Gavrilyuk2018}. Hence $\lambda_1, \lambda_6$ are associated with   shock/rarefaction waves while the remaining eigenvalues give rise to shear/contact waves. To study the hyperbolicity, it is useful to transform the equations in terms of primitive variables
\[
\prim = [h, \ v_1, \ v_2, \ \p_{11}, \ \p_{12}, \ \p_{22}]
\]
as the independent variables. Define
\[
\coefa = \sqrt{gh + 3 \p_{11}}, \qquad \coefb = \sqrt{\p_{11}}
\]
Then the eigenvectors in terms of the primitive variables are give as follows.
\paragraph{1-wave: shock/rarefaction, $\lambda_1 = u-\coefa$}
\[
\bm{r}_1 = \left[ h(\coefa^2-\coefb^2), \ -\coefa(\coefa^2-\coefb^2), \ -2\coefa\p_{12}, \ 2\coefb^2(\coefa^2-\coefb^2), \ (\coefa^2+\coefb^2)\p_{12}, \ 4\p_{12}^2 \right]
\]
\paragraph{2-shear wave: $\lambda_2 = u-\coefb$}
\[
\bm{r}_2 = [ 0, \ 0, \ -\coefb, \ 0, \ \coefb^2, \ 2 \p_{12} ]^\top
\]

\paragraph{3,4-contact wave: $\lambda_3 = \lambda_4 = u$}
\[
\bm{r}_3 = [0, \ 0, \ 0, \ 0, \ 0, \ 1]^\top
\]
\[
\bm{r}_4 = [-h, \ 0, \ 0, \ gh + \p_{11}, \ \p_{12}, \ 0]
\]

\paragraph{5-shear wave: $\lambda_5 = u+\coefb$}
\[
\bm{r}_5 = [ 0, \ 0, \ \coefb, \ 0, \ \coefb^2, \ 2 \p_{12} ]^\top
\]

\paragraph{6-wave: shock/rarefaction, $\lambda_6 = u+\coefa$}
\[
\bm{r}_6 = \left[ h(\coefa^2-\coefb^2), \ \coefa(\coefa^2-\coefb^2), \ 2\coefa\p_{12}, \ 2\coefb^2(\coefa^2-\coefb^2), \ (\coefa^2+\coefb^2)\p_{12}, \ 4\p_{12}^2 \right]
\]
The waves and their ordering are illustrated in Figure~\ref{SSW_RP_Waves}.  Note that, when $\p_{11}$ goes to zero, we have $\coefb \rightarrow 0$ and the system is no more hyperbolic. Indeed, the eigenvectors $\bm{r}_2 $, $\bm{r}_3 $ and $\bm{r}_5$ become dependent. Moreover, even if for $\p_{11} =\p_{12}=\p_{22}=0$ the system \eqref{eq:ssw1db} can be formally reduced to a conservative formulation, this change in the nature of the model is accompanied here by an eigenvalue whose multiplicity becomes four but asymptotically associated to only two independent eigenvectors.
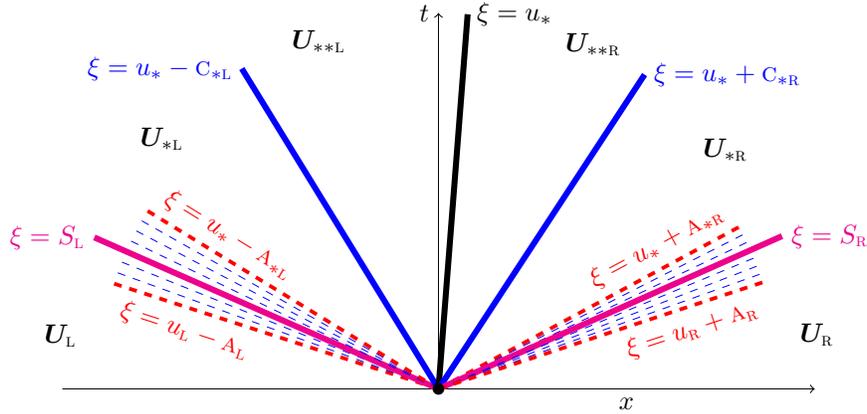
\begin{figure}
\begin{center}
 \begin{tikzpicture}[scale=4.0]
 \coordinate (X0) at (   0.00000000     ,   0.00000000     );
 \coordinate (Xm) at (  -1.25000000     ,   0.00000000     );
 \coordinate (Xp) at (   1.25000000     ,   0.00000000     );
 \coordinate (Xn) at (   0.00000000     ,   1.25000000     );
 \coordinate (Xt001) at (  -1.09377158     ,  0.355195343     );
 \coordinate (Xt002) at (  -1.08120501     ,  0.391785026     );
 \coordinate (Xt003) at (  -1.06741500     ,  0.427931309     );
 \coordinate (Xt004) at (  -1.14393139     ,  0.503905654     );
 \coordinate (Xt005) at (  -1.03622770     ,  0.498730540     );
 \coordinate (Xt006) at (  -1.01886570     ,  0.533303380     );
 \coordinate (Xt007) at (  -1.00035071     ,  0.567272663     );
 \coordinate (Xt008) at ( -0.980703592     ,  0.600600004     );
 \coordinate (Xt009) at ( -0.653674722     ,   1.06546199     );
 \coordinate (Xt010) at (   9.70319808E-02 ,   1.24622822     );
 \coordinate (Xt011) at (  0.685138464     ,   1.04550719     );
 \coordinate (Xt012) at (   1.01225865     ,  0.545740366     );
 \coordinate (Xt013) at (   1.02570510     ,  0.520027936     );
 \coordinate (Xt014) at (   1.03849852     ,  0.493984491     );
 \coordinate (Xt015) at (   1.14199018     ,  0.508289754     );
 \coordinate (Xt016) at (   1.06209445     ,  0.440970927     );
 \coordinate (Xt017) at (   1.07288182     ,  0.414034545     );
 \coordinate (Xt018) at (   1.08298612     ,  0.386834592     );
 \coordinate (Xt019) at (   1.09240103     ,  0.359388381     );
 \coordinate (Xw001) at (  -1.17188573     ,  0.177597672     );
 \coordinate (Xw002) at ( -0.817189157     ,  0.833030999     );
 \coordinate (Xw003) at ( -0.278321385     ,   1.15584517     );
 \coordinate (Xw004) at (  0.391085207     ,   1.14586771     );
 \coordinate (Xw005) at (  0.848698556     ,  0.795623779     );
 \coordinate (Xw006) at (   1.17120051     ,  0.179694191     );
 \draw[line width=0.1mm,->]  (Xm) -- (Xp)  node[near end,sloped,below] {$x$};
 \draw[line width=0.1mm,->]  (X0) -- (Xn);
 \draw (Xn) node[anchor=east]{$t$};
 \draw[red,line width=0.5mm, dashed]  (X0) -- (Xt001)  node[near end,sloped,below] {$\xi = u_{\lf} - \coefa_{\lf}$};
 \draw[blue,line width=0.1mm, dashed]  (X0) -- (Xt002);
 \draw[blue,line width=0.1mm, dashed]  (X0) -- (Xt003);
 \draw[magenta,line width=0.8mm]  (X0) -- (Xt004);
 \draw[magenta] (Xt004) node[anchor=east] {$\xi = S_{\lf}$};
 \draw[blue,line width=0.1mm, dashed]  (X0) -- (Xt005);
 \draw[blue,line width=0.1mm, dashed]  (X0) -- (Xt006);
 \draw[blue,line width=0.1mm, dashed]  (X0) -- (Xt007);
 \draw[red,line width=0.5mm, dashed]  (X0) -- (Xt008)  node[near end,sloped,above] {$\xi = u_{*} - \coefa_{*\lf}$};
 \draw[blue,line width=0.8mm]  (X0) -- (Xt009);
 \draw[blue] (Xt009) node[anchor=east]{$\xi =  u_{*} - \coefb_{*\lf}$};
 \draw[line width=0.8mm]  (X0) -- (Xt010);
 \draw (Xt010) node[anchor=west]{$\xi =  u_{*} $};
 \draw[blue,line width=0.8mm]  (X0) -- (Xt011);
 \draw[blue] (Xt011) node[anchor=west]{$\xi =  u_{*} + \coefb_{*\rg}$};
 \draw[red,line width=0.5mm, dashed]  (X0) -- (Xt012) node[near end,sloped,above] {$\xi = u_{*} + \coefa_{*\rg}$};
 \draw[blue,line width=0.1mm, dashed]  (X0) -- (Xt013);
 \draw[blue,line width=0.1mm, dashed]  (X0) -- (Xt014);
 \draw[magenta,line width=0.8mm]  (X0) -- (Xt015);
 \draw[magenta] (Xt015) node[anchor=west] {$\xi = S_{\rg}$};
 \draw[blue,line width=0.1mm, dashed]  (X0) -- (Xt016);
 \draw[blue,line width=0.1mm, dashed]  (X0) -- (Xt017);
 \draw[blue,line width=0.1mm, dashed]  (X0) -- (Xt018);
 \draw[red,line width=0.5mm, dashed]  (X0) -- (Xt019) node[near end,sloped,below] {$\xi = u_{\rg} + \coefa_{\rg}$};
 \draw (Xw001) node[anchor=east] {$\Large \con_{\lf}$};
 \draw (Xw002) node[anchor=east] {$\Large \con_{*\lf}$};
 \draw (Xw003) node[anchor=east] {$\Large \con_{**\lf}$};
 \draw (Xw004) node[anchor=west] {$\Large \con_{**\rg}$};
 \draw (Xw005) node[anchor=west] {$\Large \con_{*\rg}$};
 \draw (Xw006) node[anchor=west] {$\Large \con_{\rg}$};
 \filldraw [opacity=1] (X0) circle (0.5pt);
\end{tikzpicture}
\end{center}
   \caption{ Shear shallow water model: Wave pattern  for the 1-D
     Riemann problem. Plain lines are used for discontinuities and
     dashed lines for rarefaction waves. For the first and the last
     waves, we need to estimate  whether it is a shock or a
     rarefaction  wave. Waves speeds are defined with  the
     self-similar variable $\xi =x/t$.  }
 \label{SSW_RP_Waves}
 \end{figure}

\section{Concept of weak solution}
If we have discontinuous solutions for~\eqref{eq:ssw1db}, then we have to give a proper mathematical meaning to the spatial derivative term which is based on a weak formulation using integration by parts if $\A$ is the gradient of a flux function as in case of conservation laws. If $\A$ is not the gradient of a flux, then the non-conservative product is interpreted as a Borel measure~\cite{DalMaso1995}. This definition requires the choice of a smooth path $\Psi : [0,1] \times \uad \times \uad \to \uad$ connecting the two states $\conl, \conr$ across the jump discontinuity at $x=x_0$ such that
\[
\Psi(0; \conl, \conr) = \conl, \qquad \Psi(1; \conl, \conr) = \conr
\]
where $\uad$ is the set of admissible states. Then the non-conservative product is defined as the Borel measure~\cite{DalMaso1995,Gosse2001}
\[
\mu(x_0) = \left[ \int_0^1 \A(\Psi(\xi; \conl, \conr)) \dd{\Psi}{\xi}(\xi;\con_{\lf},\con_{\rg}) \ud \xi \right] \delta(x_0)
\]
where $\delta$ is the Dirac delta function. The quantity inside the square brackets will be referred to as the {\em fluctuation} and plays an important role in the construction of approximate Riemann solvers. This viewpoint is equivalent to the definition of non-conservative product proposed by Volpert~\cite{Volpert1967}. Using this notion, a theory of weak solutions can be developed based on which the Riemann problem has usual structure as for conservative systems, leading to shocks or rarefaction waves corresponding to genuinely non-linear characteristic fields and contact waves corresponding to linearly degenerate fields. Across a point of discontinuity moving with speed $S$, a weak solution has to  satisfy the generalized Rankine-Hugoniot jump condition
\[
\int_0^1 \left[ \A(\Psi(\xi; \conl, \conr)) - S I \right] \dd{\Psi}{\xi}(\xi;\con_{\lf},\con_{\rg}) \ud \xi = 0
\]
The choice of the correct path is a difficult question and has to be derived from a regularized model motivated from the physical background of the problem. In many applications, the choice of the correct path is not known and in practice, it is usual to consider the linear path in state space
\begin{equation}
\Psi(\xi; \conl, \conr) = \conl + \xi (\conr - \conl)
\label{eq:path}
\end{equation}
Then the jump condition for our model~\eqref{eq:ssw1d} becomes
\begin{equation}
\label{eq:rh}
\int_0^1 \A(\Psi(\xi; \conl, \conr))  \dd{\Psi}{\xi} \ud \xi = \f_R - \f_L + \B(\m_L, \m_R) (h_R - h_L) = S(\conr - \conl)
\end{equation}
where
\[
\B(\m_L, \m_R) = \B\left( \frac{\m_L + \m_R}{2} \right)
\]
The source term $\s$ does not make any contribution to the jump conditions since it does not contain derivative of $\con$.
\subsection{Rankine-Hugoniot  jump conditions}
In the following, we will assume that $\con_{\lf},\con_{\rg}$ are the left and right states in a Riemann problem. Let us define the average and jump operators by
\[
\avg{\cdot} = \frac{ (\cdot)_{\lf} + (\cdot)_{\rg}}{2}, \qquad \jump{ \cdot } = (\cdot)_{\rg} - (\cdot)_{\lf}
\]
Then the jump conditions~\eqref{eq:rh} across a discontinuity moving with speed $S$ lead to the following set of generalized Rankine-Hugoniot conditions.
\begin{subequations}
\begin{align}
\label{eq:rh1}
\jump{h u} &= S \jump{h} \\
\label{eq:rh2}
\jump{\R_{11} + h u^2 + \half g h^2} &= S \jump{h u} \\
\label{eq:rh3}
\jump{\R_{12} + h u v} &= S \jump{h v} \\
\label{eq:rh4}
\jump{\E_{11} u + \R_{11} u} + g \avg{h u} \jump{h} &= S \jump{\E_{11}} \\
\label{eq:rh5}
\jump{\E_{12} u + \half (\R_{11} v + \R_{12} u)} + \half g \avg{h v} \jump{h} &= S \jump{\E_{12}} \\
\label{eq:rh6}
\jump{\E_{22} u + \R_{12} v} &= S \jump{\E_{22}}
\end{align}
\end{subequations}
Moreover, the total energy equation~\eqref{eq:toteeq} also has an associated jump condition.
\begin{lemma}\label{lem:totE}
For the linear path~\eqref{eq:path}, the jump conditions~\eqref{eq:rh1}-\eqref{eq:rh6} are consistent with the jump conditions of the total energy equation~\eqref{eq:toteeq}.
\end{lemma}
\begin{proof}
We will show that the jump conditions~\eqref{eq:rh1}-\eqref{eq:rh6} imply that
\begin{equation}
\jump{ \left(E + \R_{11} + \half g h^2 \right) u + \R_{12} v } = S \jump{E}
\label{eq:jumptotE}
\end{equation}
which is the jump condition for the total energy equation. Adding the jump conditions from $\E_{11}, \E_{22}$ equations
\begin{equation}
\jump{(\E_{11} + \E_{22})u + \R_{11} u + \R_{12}v } + g \avg{h u} \jump{h} = S \jump{\E_{11} + \E_{22}}
\label{eq:jumpE1}
\end{equation}
Also
\begin{eqnarray}
\nonumber
\jump{(gh^2 + gh\bt)u}  &=& g \avg{hu} \jump{h} + g \avg{h} \jump{hu} + g \bt \jump{hu} \\
\nonumber
&=& g \avg{hu} \jump{h} + S g \avg{h} \jump{h} + S g \bt \jump{h}, \qquad \textrm{using~\eqref{eq:rh1}} \\
\label{eq:jumpE2}
&=& g \avg{hu} \jump{h} + S \jump{\half g h^2  + g h \bt }
\end{eqnarray}
Adding~\eqref{eq:jumpE1} and~\eqref{eq:jumpE2}, we obtain~\eqref{eq:jumptotE}.
\end{proof}
\section{ Properties and structures of single waves.}
We will  focus in this paper on the derivation of an exact solution for the Riemann problem associated to  the 1-D SSW model where the source term is set to zero. The spectral analysis of this genuinely non-conservative hyperbolic system was proposed in \cite{Gavrilyuk2018}. Within the {\em path conservative} framework, generalized  jump conditions for  this non-conservative system was derived in~\cite{Chandrashekar2020}.  These results  have been recalled in the previous section. In order to define the strategy that will allow us to obtain the  exact solution of the Riemann problem, we first need to characterize the properties of waves associated  with each eigenvalue. The  first and the sixth characteristic fields, respectively associated to $\lambda_1 = u -\coefa$ and $\lambda_6 = u + \coefa$,  are genuinely non-linear and can develop either shock (discontinuous) or rarefaction (continuous) waves.  The other characteristic fields are associated to linearly degenerate waves. We will name {\em contact wave}  the field associated to the eigenvalue $\lambda_3 =\lambda_4 =u$ and {\em shear waves} the field associated to $\lambda_2 = u -\coefb$ and $\lambda_5 = u + \coefb$. Asymptotically,  the contact and the shear waves will collapse to a single wave when $\coefb$ goes to zero,  which will be  the  case when the variable $\p_{11}$ goes to zero.  Since  the tensor $\p$ is always symmetric and positive definite, there is a strict ordering of the eigenvalues
\[ u -\coefa   <  u -\coefb   <  u   <  u +\coefb  <  u + \coefa
\]
However, when approximations are applied with small values of $\p_{11}$, we can face some numerical inconsistencies. Riemann invariants are constant across linearly degenerate waves and rarefaction waves, whereas for shock waves, generalized jump conditions should be satisfied. In the subsequent subsections, we will derive the Riemann invariants or the relations to be satisfied for each single wave connecting two different states: the state $\con_{\lf}$ on the left and $\con_{\rg}$ on the right of the wave.
\subsection{Rarefaction waves}
The eigenvalues $\lambda_1 = u -\coefa$ and $\lambda_6 = u + \coefa$ are genuinely non-linear and may give rise to rarefaction waves.
\begin{theorem}
For the 1-rarefaction wave, the Riemann invariants are given by
\begin{equation}\label{eq:rareinv1}
\frac{\p_{11}}{h^2}, \qquad u + \Afun(h,c), \qquad \frac{\detP}{h^2}, \qquad \frac{\p_{12}}{g h + 2 \p_{11}}, \qquad v + \frac{2\p_{12}}{g h + 2 \p_{11}} \Afun(h,c)
\end{equation}
and for the 6-rarefaction wave, the Riemann invariants are given by
\begin{equation}\label{eq:rareinv6}
\frac{\p_{11}}{h^2}, \qquad u - \Afun(h,c), \qquad \frac{\detP}{h^2}, \qquad \frac{\p_{12}}{g h + 2 \p_{11}}, \qquad v - \frac{2\p_{12}}{g h + 2 \p_{11}} \Afun(h,c)
\end{equation}
where
\begin{equation}
\Afun(h,c) = \sqrt{g h + 3 c h^2} + \frac{g}{\sqrt{3c}} \sinh^{-1}\sqrt{ \frac{3 c h}{g}}, \qquad c = \frac{\p_{11}}{h^2}
\label{eq:Afun}
\end{equation}
\end{theorem}
\begin{proof}
We find the Riemann invariants by analyzing the integrals curves of the eigenvector fields. The integral curve corresponding to $\bm{r}_1$ satisfy the following set of equations
\begin{equation}
\frac{\ud h}{h(\coefa^2-\coefb^2)} = \frac{\ud u}{-\coefa(\coefa^2 - \coefb^2)} = \frac{\ud v}{-2\coefa \p_{12}} = \frac{\ud \p_{11}}{2 \coefb^2(\coefa^2 - \coefb^2)} = \frac{\ud \p_{12}}{(\coefa^2 + \coefb^2)\p_{12}} = \frac{\ud \p_{22}}{4 \p_{12}^2}
\label{eq:rareintcurve}
\end{equation}
Using the first and fourth terms in \eqref{eq:rareintcurve}, we get
\[
\frac{\ud h}{h}  = \frac{\ud \p_{11}}{2 \p_{11}} \quad\implies\quad \frac{\p_{11}}{h^2} = \frac{\R_{11}}{h^3} = \textrm{constant} = c
\]
which is the first invariant. Using the first and second terms in \eqref{eq:rareintcurve}, we get
\[
\ud u = - \frac{\coefa}{h} \ud h = - \frac{ \sqrt{ g h + 3 c h^2}}{h} \ud h
\]
where we used the first invariant. Integrating this, we obtain the second invariant $u + \Afun(h,c)$, with
\[
\Afun(h,c)  = \int \frac{1}{h} \sqrt{ g h + 3 c h^2} \ud h \equiv \sqrt{g h + 3 c h^2} + \frac{g}{\sqrt{3c}} \sinh^{-1}\sqrt{ \frac{3 c h}{g}}
\]
By definition of the determinant, we have $\detP = \p_{11} \p_{22} - \p_{12}^2$. Therefore,
\[
\begin{array}{rclcl}
\ud[\detP] &=& \p_{22} \ud\p_{11} + \p_{11} \ud \p_{22} - 2 \p_{12} \ud \p_{12} & &\\
&=& \displaystyle
\p_{22} \frac{2 \coefb}{h} \ud h + \p_{11} \frac{4 \p_{12}^2}{h(\coefa^2-\coefb^2)} \ud h - 2 \p_{12} \frac{(\coefa^2+\coefb^2) \p_{12}}{h(\coefa^2-\coefb^2)} \ud h
&= &2 \dfrac{\detP}{h} \ud h
\end{array}
\]
Hence, the enstrophy $\frac{\detP}{h^2}$ is conserved across rarefaction waves, which is the third  Riemann invariant.
From the first and fifth terms in \eqref{eq:rareintcurve}, we get
\[
\ud \ln \p_{12} = \frac{1}{h} \frac{\coefa^2 + \coefb^2}{\coefa^2 - \coefb^2} \ud h = \frac{1}{h} \frac{gh + 4 \p_{11}}{g h + 2 \p_{11}} \ud h =  \frac{g + 4c h}{g h  + 2 c h^2} \ud h =  \ud \ln(g h + 2 c h^2)
\]
and we obtain the fourth invariant, for convenience denoted as $\newInv =  \frac{\p_{12}}{g h + 2 c h^2}$.
Finally, from the first and third terms in~\eqref{eq:rareintcurve}, we get
\[
\ud v = - \frac{2 \coefa \p_{12}}{h(\coefa^2-\coefb^2)} \ud h = -2 \newInv  \frac{\sqrt{gh + 3 c h^2}}{h} \ud h
\]
and integrating this we obtain the fifth invariant: $v - 2 \newInv \Afun(h,c)$.
The proof for the 6-rarefaction is similar except for some sign differences.
\end{proof}
The two states $\con_{\lf},\con_{\rg}$ can be connected by a
rarefaction wave provided they satisfy the Lax condition; for a 1-rarefaction, they must satisfy
\begin{equation} \label{eq:laxrare}
\lambda_1(\con_{\lf}) < \lambda_1(\con_{\rg})
\end{equation}
and a similar condition must be satisfied in case of a 6-rarefaction wave.
\begin{lemma}
The set of admissible left and right states $\con_{\lf},\con_{\rg}$ that can be connected by
\begin{enumerate}
\item a 1-rarefaction must satisfy $h_{\rg} < h_{\lf}$
\item a 6-rarefaction must satisfy $h_{\lf} < h_{\rg}$
\end{enumerate}
\end{lemma}
\begin{proof}
(1) Using the Riemann invariants $c = \frac{\p_{11}}{h^2}$ and $u + \Afun(h, c) $
the difference of the velocity  can be written as
\begin{eqnarray*}
u_{\lf} - u_{\rg} &=& \Afun(h_{\rg}, c) - \Afun(h_{\lf},c) \\
&=& \sqrt{g h_{\rg} + 3 \p_{11}^{\rg}} - \sqrt{g h_{\lf} + 3 \p_{11}^{\lf}} + \frac{g}{\sqrt{3c}} \left[ \sinh^{-1}\sqrt{ \frac{3 c h_{\rg}}{g}} - \sinh^{-1}\sqrt{ \frac{3 c h_{\lf}}{g}} \right] 
\end{eqnarray*}
For a 1-rarefaction wave, the characteristic speeds must satisfy the condition~\eqref{eq:laxrare}, which leads to
\[
u_{\lf} - \sqrt{g h_{\lf} + 3 \p_{11}^{\lf}} < u_{\rg} - \sqrt{g h_{\rg} + 3 \p_{11}^{\rg}} \quad\implies\quad u_{\lf} - u_{\rg} < \sqrt{g h_{\lf} + 3 \p_{11}^{\lf}} - \sqrt{g h_{\rg} + 3 \p_{11}^{\rg}}
\]
Combining the above two relations, we get
\[
2 \sqrt{g h_{\rg} + 3 c h_{\rg}^2}  + \frac{g}{\sqrt{3c}} \sinh^{-1}\sqrt{ \frac{3 c h_{\rg}}{g}} \le  2 \sqrt{g h_{\lf} + 3 c h_{\lf}^2}  + \frac{g}{\sqrt{3c}} \sinh^{-1}\sqrt{ \frac{3 c h_{\lf}}{g}}
\]
 On the other hand, the function
$f(h,c) = 2 \sqrt{g h + 3 c h^2}  + \frac{g}{\sqrt{3c}} \sinh^{-1}\sqrt{ \frac{3 c h}{g}}$ is a increasing function of $h$ for $c$ fixed and $g>0$ a given constant.
Therefore, the conditions on the characteristic speeds 
is satisfied if and only if $h_{\rg} < h_{\lf}$. The proof is similar for the case of a 6-rarefaction wave.
\end{proof}
\subsection{Internal structure of 1-rarefaction}
The solution inside the rarefaction is self-similar and depends only on the ratio $x/t$. The slope of the characteristics is
\[
\xi = \frac{x}{t} = u - \coefa \quad \Longrightarrow \quad
-\frac{\coefa}{h} \ud h = \ud u = \ud\xi + \ud \coefa
\]
where equation  \eqref{eq:rareintcurve}  has been used. From this we obtain the relation
\[
\ud\xi = - \frac{3 g + 12 c h}{2\sqrt{gh + 3 c h^2}} \ud h
\qquad\text{ where } \qquad  c = \frac{\p_{11}^{\lf}}{h_{\lf}^2}
\]
We can integrate this ODE with the initial condition:
$
h(\xi_{\lf}) = h_{\lf}$ and  $\xi_{\lf} = u_{\lf} - \sqrt{g h_{\lf} + 3 \p_{11}^{\lf}}$.\\
We then obtain
\[
\xi - \xi_{\lf} = -\left (\rule{0mm}{4mm} \Bfun(h) - \Bfun(h_{\lf}) \right )
\qquad\text{ with } \qquad
  \Bfun(h) =  2\sqrt{g h + 3 c h^2} + \frac{g}{\sqrt{3c}} \sinh^{-1}\sqrt{ \frac{3 c h}{g}}
\]
This equation implicitly defines the function $h(\xi)$ in the internal structure of the 1-rarefaction. Once $h = h(\xi)$ is obtained, we can use the Riemann invariants to compute all the other variables inside the rarefaction wave leading to the complete solution $\con(\xi)$  for $u_{\lf} - \coefa_{\lf} \leq \xi \leq  u_{\rg} - \coefa_{\rg}$.
\subsection{Shear waves}
\label{sec:jumps}
The eigenvalues $\lambda_2, \lambda_5$ and the associated eigenvectors give rise to shear waves. Across a shear wave, the water depth $h$ and normal velocity $u$ are continuous while the transverse velocity $v$ may have a jump as shown by the Riemann invariants.
\begin{theorem}
For the 2-shear wave, the Riemann invariants are given by
\[
h, \qquad u, \qquad \p_{11}, \qquad v \sqrt{\p_{11}} + \p_{12}, \qquad \detP
\]
while for the 5-shear wave, they are given by
\[
h, \qquad u, \qquad \p_{11}, \qquad v \sqrt{\p_{11}} - \p_{12}, \qquad \detP
\]
\end{theorem}
\begin{proof}
The integral curve corresponding to the eigenvector $\bm{r}_2$ satisfies the equations
\[
\frac{\ud h}{0} = \frac{\ud u}{0} = \frac{\ud v}{-\coefb} = \frac{\ud \p_{11}}{0} = \frac{\ud \p_{12}}{\coefb^2} = \frac{\ud \p_{22}}{2 \p_{12}}
\]
We immediately see that $h, u, \p_{11}$ are invariants. From the third and fifth terms, we obtain
\[
\coefb \ud v + \ud \p_{12} = 0 \qquad\implies\qquad v \sqrt{\p_{11}} + \p_{12} = \textrm{constant}
\]
Finally from the fifth and sixth terms, we obtain
\[
- 2 \p_{12} \ud \p_{12} + \coefb^2 \ud \p_{22} = 0 \qquad\implies\qquad \detP = \textrm{constant}
\]
The proof for the 5-shear wave is similar.
\end{proof}
\begin{remark}
Note that, as $h$ and $\p_{11}$ are Riemann invariants of shear waves, the total pressure $\Pres = \frac{gh^2}{2} + \R_{11}$ is also invariant across shear waves. Moreover, across the 2-shear wave, the eigenvalue $\lambda_2 = u - \sqrt{\p_{11}}$ is an invariant, and across the 5-shear wave, the eigenvalue $\lambda_5 = u + \sqrt{\p_{11}}$ is an invariant. It can be checked that any two states $\con_\lf, \con_\rg$ which satisfy the Riemann invariants will satisfy all the jump conditions for the shear waves, with the speed of the discontinuity being $\lambda_2$ or $\lambda_5$.
\end{remark}
\subsection{Contact waves }
\label{sec:jumpc}
The eigenvalues $\lambda_3, \lambda_4$ and the corresponding eigenvectors give rise to contact waves. Across such a wave the velocity is continuous but the water depth may possibly have a jump discontinuity.
\begin{theorem}
For the contact wave, the Riemann invariants are given by
\[
u, \qquad  v, \qquad  \R_{12} \quad \text{ and } \quad
 \Pres = \frac{gh^2}{2} + \R_{11}
 \]
 where $\Pres$ is defined as the total pressure.
\end{theorem}
\begin{proof}
The contact wave is associated to the eigenvalue  $u$ with a multiplicity of two with two linearly independent eigenvectors. As the multiplicity is two, we cannot expect more than four Riemann invariants. Indeed, the invariants should satisfy  the following equations due to the two eigenvectors $\bm{r}_3, \bm{r}_4$,
\[
\frac{\ud h}{0} = \frac{\ud u}{0} = \frac{\ud v}{0} = \frac{\ud \p_{11}}{0} = \frac{\ud \p_{12}}{0} = \frac{\ud \p_{22}}{1}
\]
and
\[
\frac{\ud h}{-h} = \frac{\ud u}{0} = \frac{\ud v}{0} = \frac{\ud \p_{11}}{gh+\p_{11}} = \frac{\ud \p_{12}}{\p_{12}} = \frac{\ud \p_{22}}{0}
\]
As a consequence, the Riemann invariants for contact waves are defined
by the following equalities
\[
\frac{\ud h}{-h} = \frac{\ud u}{0} = \frac{\ud v}{0} = \frac{\ud \p_{11}}{gh+\p_{11}} = \frac{\ud \p_{12}}{\p_{12}}
\]
From the second and the third terms of these equalities we obtain the invariants $u$ and $v$. Combining the first and the fifth terms we find that $\R_{12}$ is the third invariant. Finally, the first and the fourth terms give
\[
gh \ud h  + \p_{11} \ud h +  h \ud\p_{11}  = 0 \qquad\Longrightarrow\qquad
\ud\left (  g\frac{h^2}{2} + h\p_{11}
\right ) = 0
\]
so that $ \Pres = \frac{gh^2}{2} + \R_{11}$ is the fourth invariant.
\end{proof}
\begin{remark}
By definition, the gradients of these Riemann invariants, with respect to
the primitive variable $\prim$, are orthogonal to the plane spanned by
eigenvectors $\bm{r}_3$ and $\bm{r}_4$,
\[
\left \{
\begin{array}{rl}
\bm{r}_3  \cdot \dfrac{ \partial u}{\partial \prim} & = 0 \\[3mm]
\bm{r}_4 \cdot \dfrac{ \partial u}{\partial \prim} & =  0
\end{array}
\right .
, \quad
\left \{
\begin{array}{rl}
\bm{r}_3  \cdot \dfrac{ \partial v}{\partial \prim}  & =  0 \\[3mm]
\bm{r}_4 \cdot \dfrac{ \partial v}{\partial \prim} & =  0
\end{array}
\right .
, \quad
\left \{
\begin{array}{rl}
\bm{r}_3 \cdot \dfrac{ \partial \R_{12}} {\partial \prim} & =  0 \\[3mm]
\bm{r}_4 \cdot \dfrac{ \partial \R_{12}}{\partial \prim}& = 0
\end{array}
\right .
\text{ and }
\left \{
\begin{array}{rl}
\bm{r}_3 \cdot \dfrac{ \partial \Pres}{\partial \prim}  & =  0 \\[3mm]
\bm{r}_4 \cdot \dfrac{ \partial \Pres}{\partial \prim} & = 0
\end{array}
\right .
\]
which can be verified. Moreover, we see that the eigenvalues $\lambda_3 = \lambda_4 = u$ is an invariant.
\end{remark}
\begin{remark}
Let us examine the jump conditions for the contact wave. The speed of the contact wave is equal to the common fluid velocity $u_\lf = u_\rg = u = \lambda_3 = \lambda_4$ which are linearly degenerate. Then the jump conditions lead to the following set of conditions
\begin{equation}
\jump{\R_{11} + \half g h^2} = 0, \qquad \jump{\R_{12}} = 0, \qquad             \jump{\R_{11} v} + g \avg{hv} \jump{h} = 0, \qquad \jump{\R_{12} v} = 0
\label{eq:jumpmid}
\end{equation}
From the first condition, the total pressure $\Pres$ is constant across this wave. From the first and third conditions, we obtain
\[
\left( \avg{\R_{11}} + \frac{1}{4} g \jump{h}^2 \right) \jump{v} = 0
\]
Since we $\R_{11}$ must be strictly positive, the first factor cannot
be zero and   hence we require that $\jump{v} = 0$, so that both
velocity components are      continuous across the contact wave. The
second condition of~\eqref{eq:jumpmid}   shows that $\R_{12}$ is also
continuous across the middle wave. These results are consistent with
the Riemann invariants derived in the previous theorem.
\end{remark}
\subsection{Shock wave, Hugoniot curve and entropy condition}
The states $\con_{\lf}, \con_{\rg}$ can be connected by a shock wave
only if they satisfy the Lax condition, i.e., the characteristics must
intersect into the shock wave . For the 1-shock, this condition is
given by
\begin{equation} \label{eq:laxshock}
\lambda_1(\con_{\lf}) > S > \lambda_1(\con_{\rg})
\end{equation}
where $S$ is the shock speed, with a similar condition for the 6-shock wave. Before using this condition, we derive the Hugoniot relation between the two states which follows from the generalized jump conditions after eliminating the velocity.
\begin{theorem}
The set of states $\con_{\lf}, \con_{\rg}$ which can be connected by a shock lie on the Hugoniot curve given by
\begin{equation}
\frac{3}{2} \jump{\tau \R_{11}} - \avg{\tau} \jump{\R_{11}} + \frac{g \jump{\tau}^3}{4 \tau_{\lf}^2 \tau_{\rg}^2}  = 0
\quad \text{ with } \quad  \tau = \frac{1}{h}
\label{eq:hugoniot}
\end{equation}
\end{theorem}
\begin{proof}
Let us change to a coordinate frame in which the shock is stationary. The jump conditions for the continuity, $x$-momentum and $x$ component of energy equation are
\[
\jump{hu}=0, \qquad \jump{\R_{11} + h u^2 + \half g h^2} = 0, \qquad \jump{(\E_{11} + \R_{11})u} + g \avg{hu} \jump{h} = 0
\]
Let $m = h_{\lf} u_{\lf} = h_{\rg} u_{\rg}$, then the second and third conditions can be written as
\[
\jump{\R_{11}} + m \jump{u} + g \avg{h} \jump{h} = 0, \qquad \frac{3}{2} \jump{\p_{11}} + \avg{u} \jump{u} + g \jump{h} = 0
\]
Let $\tau = 1/h$; then
\[
\avg{h} = \avg{1/\tau} = \frac{1}{\tau_{\lf} \tau_{\rg}} \avg{\tau}, \qquad \jump{h} = \jump{1/\tau} = - \frac{1}{\tau_{\lf} \tau_{\rg}} \jump{\tau}
\]
The two jump conditions become
\[
\jump{\R_{11}} + m \jump{u} - \frac{g}{\tau_{\lf}^2 \tau_{\rg}^2} \avg{\tau} \jump{\tau} = 0, \qquad \frac{3}{2} \jump{\tau \R_{11}} + \avg{u} \jump{u} - \frac{g}{\tau_{\lf} \tau_{\rg}} \jump{\tau} = 0
\]
Using the first equation, we eliminate $\jump{u}$ from the second equation
\[
\frac{3}{2} \jump{\tau \R_{11}} + \frac{\avg{u}}{m} \left( \frac{g}{\tau_{\lf}^2 \tau_{\rg}^2} \avg{\tau} \jump{\tau} - \jump{\R_{11}} \right) - \frac{g}{\tau_{\lf} \tau_{\rg}} \jump{\tau} = 0
\]
But since $\avg{u}/m = \avg{\tau}$, we get
\[
\frac{3}{2} \jump{\tau \R_{11}} - \avg{\tau} \jump{\R_{11}} + \frac{g}{\tau_{\lf} \tau_{\rg}} \jump{\tau} \left( \frac{\avg{\tau}^2}{\tau_{\lf} \tau_{\rg}} - 1 \right) = 0
\]
which upon simplification of the last term yields the Hugoniot curve~\eqref{eq:hugoniot}.

\end{proof}
We now find some constraints on the two states imposed by the Lax entropy condition if they have to be connected by a shock wave.
\begin{theorem} Any given admissible left and right states $\con_{\lf}, \con_{\rg}$ can be connected by a
\begin{itemize}
\item  1-shock wave if $h_{\rg} \in (h_{\lf}, 2 h_{\lf})$.
\item  6-shock wave if $h_{\lf} \in (h_{\rg}, 2 h_{\rg})$.
\end{itemize}
\end{theorem}
\begin{proof}
Given the left state $(\tau_{\lf}, \R_{11}^{\lf})$ the Hugoniot curve gives the set of right states $(\tau_{\rg}, \R_{11}^{\rg})$ that can be connected to it by a shock. Using the Hugoniot curve, we can  obtain the stress component at the right state as
\[
\R_{11}^{\rg} = \frac{1}{2 \tau_{\rg} - \tau_{\lf}} \left[ (2 \tau_{\lf} - \tau_{\rg}) \R_{11}^{\lf} - \frac{g \jump{\tau}^3}{2\tau_{\lf}^2 \tau_{\rg}^2} \right]
= R_{11}(\tau_{\rg}; \tau_{\lf}, \R_{11}^{\lf})
\]
where
\begin{equation}
\R_{11}(\tau; \tau_{\lf}, \R_{11}^{\lf}) = \frac{1}{2 \tau - \tau_{\lf}} \left[ (2 \tau_{\lf} - \tau) \R_{11}^{\lf} - \frac{g (\tau - \tau_{\lf})^3}{2\tau_{\lf}^2 \tau^2} \right]
\label{eq:r11shock}
\end{equation}
When $\tau_{\rg} = \half \tau_{\lf}$ we have $\R_{11}^{\rg} = \infty$ and moreover $\R_{11}^{\rg} < 0$ for $\tau_{\rg} < \half \tau_{\lf}$. Hence from positivity requirement, the admissible range of values for $\tau_{\rg}$ is such that $ \tau_{\rg} > \half \tau_{\lf}$.

The Lax entropy condition says that characteristics must enter into the shock curve which means that, if $S$ is the shock speed, we have
\[
u_{\lf} - \sqrt{g h_{\lf} + 3 \p_{11}^{\lf}} \quad > \quad S \quad > \quad u_{\rg} - \sqrt{g h_{\rg} + 3 \p_{11}^{\rg}}
\]
from which we obtain two Lax inequalities
\[
u_{\lf} - S > \sqrt{g h_{\lf} + 3 \p_{11}^{\lf}} > 0 \qquad  \text{ and } \qquad u_{\rg} - S < \sqrt{g h_{\rg} + 3 \p_{11}^{\rg}}
\]
The first Lax inequality shows that the left state is the pre-shock state, since the velocity relative to the shock is positive. Using the jump condition of the continuity equation, $h_{\lf}(u_{\lf}-S) = h_{\rg}(u_{\rg} - S)$, we get
\[
u_{\rg} - S > \frac{h_{\lf}}{h_{\rg}} \sqrt{g h_{\lf} + 3 \p_{11}^{\lf}}
\]
Combining this with the second Lax inequality, we get
\[
g h_{\lf}^3 + 3 h_{\lf} \R_{11}^{\lf} < g h_{\rg}^3 + 3 h_{\rg} \R_{11}^{\rg} \qquad \Longrightarrow \qquad
\R_{11}^{\rg} > \R_{11}^\star(\tau_{\rg}; \tau_{\lf}, \R_{11}^{\lf})
\]
where the function $\R_{11}^\star(\tau; \tau_{\lf}, \R_{11}^{\lf}) $ is defined by
\[
\R_{11}^\star(\tau; \tau_{\lf}, \R_{11}^{\lf}) =  \frac{\tau}{3} \left[ \frac{g}{\tau_{\lf}^3} - \frac{g}{\tau^3} + 3 \frac{\R_{11}^{\lf}}{\tau_{\lf}} \right]
\]
The entropy condition (second Lax inequality) requires that
\[
\R_{11}^{\rg} = \R_{11}(\tau_{\rg}; \tau_{\lf}, \R_{11}^{\lf}) > \R_{11}^s(\tau_{\rg}; \tau_{\lf}, \R_{11}^{\lf})
\]
Now\footnote{We are not interested in the case $\tau \le \half\tau_{\lf}$.}
\[
\dd{}{\tau}\R_{11}(\tau; \tau_{\lf}, \R_{11}^{\lf}) = - \frac{g (\tau - \tau_{\lf})^2 (4 \tau - \tau_{\lf}) + 6 \tau_{\lf}^4 \R_{11}^{\lf}}{2 \tau_{\lf}^3 (2 \tau - \tau_{\lf})^2} < 0, \qquad \tau > \half \tau_{\lf}
\]
and
\[
\dd{}{\tau}\R_{11}^\star(\tau; \tau_{\lf}, \R_{11}^{\lf}) = \frac{g}{3}\left( \frac{2}{\tau^3} + \frac{1}{\tau_{\lf}^3} \right) + \frac{\R_{11}^{\lf}}{\tau_{\lf}} > 0, \qquad \tau > 0
\]
This shows that for $\tau > \half \tau_{\lf}$, $\R_{11}(\tau; \tau_{\lf}, \R_{11}^{\lf})$ is a decreasing function and $\R_{11}^\star(\tau; \tau_{\lf}, \R_{11}^{\lf})$ is an increasing function; moreover $\R_{11}(\half\tau_{\lf}; \tau_{\lf}, \R_{11}^{\lf}) = \infty > \R_{11}^\star(\half\tau_{\lf}; \tau_{\lf}, \R_{11}^{\lf})$ and $\R_{11}(\tau_{\lf}; \tau_{\lf}, \R_{11}^{\lf}) = \R_{11}^\star(\tau_{\lf}; \tau_{\lf}, \R_{11}^{\lf})$. Hence
\[
\R_{11}(\tau; \tau_{\lf}, \R_{11}^{\lf}) > \R_{11}^\star(\tau; \tau_{\lf}, \R_{11}^{\lf}), \qquad \textrm{if and only if } \tau \in \left(\half \tau_{\lf}, \tau_{\lf}\right)
\]
The admissible range of values for $\tau_{\rg}$ is $\left(\half \tau_{\lf}, \tau_{\lf} \right)$ and hence $h_{\rg} \in (h_{\lf}, 2 h_{\lf})$. Across a shock wave, the water depth $h$ can at most increase by a factor of less than two.

The proof for the 6-shock case follows similarly.
\end{proof}
\begin{lemma} \ \\
(1) If the left and right states $\con_{\lf},\con_{\rg}$ are connected by a 1-shock, then:
$
u_{\rg} < u_{\lf} \quad \text{ and } \quad \Pres_{\rg} > \Pres_{\lf}
$ \\
(2) If the left and right states $\con_{\lf},\con_{\rg}$ are connected by a 6-shock, then:
$
u_{\lf} > u_{\rg} \quad \text{ and } \quad \Pres_{\lf} > \Pres_{\rg}
$\\
(3) Moreover, in either case, we have
\begin{equation}
u_{\lf} - u_{\rg} = \sqrt{ \frac{(h_{\rg}-h_{\lf})(\Pres_{\rg} - \Pres_{\lf})}{h_{\rg} h_{\lf}} }
\label{eq:shockulur}
\end{equation}
\end{lemma}
\begin{proof} \ \\
(1) The jump condition of the continuity equation, $h_{\lf}(u_{\lf}-S) = h_{\rg}(u_{\rg} - S)$, when applied to the 1-wave, gives
\[
u_{\rg} = \underbrace{\frac{h_{\lf}}{h_{\rg}}}_{\in (\half,1)} \underbrace{(u_{\lf} - S)}_{>0} + S \le u_{\lf} - S + S = u_{\lf}
\]
Thus the post-shock velocity $u_{\rg}$ is smaller than the pre-shock velocity $u_{\lf}$.
The total pressure is  defined as $\Pres = \R_{11} + \half g h^2$. Then, using the Hugoniot curve, we have
\[
\Pres_{\rg} -  \Pres_{\lf} = \frac{3(h_{\rg} - h_{\lf})}{2 h_{\lf} - h_{\rg}} \R_{11}^{\lf} + \half g \frac{h_{\lf} (h_{\rg}^2-4 h_{\rg} h_{\lf} + 3 h_{\lf}^2)}{h_{\rg} - 2 h_{\lf}}
\]
and since $h_{\rg} \in (h_{\lf}, 2 h_{\lf})$, both terms on the right of the above equation are positive, so that $\Pres_{\rg} > \Pres_{\lf}$. \\
\noindent
(2) In the context of a 6-shock, we have
\[
u_{\lf} = \underbrace{\frac{h_{\rg}}{h_{\lf}}}_{\in (\half,1)} \underbrace{(u_{\rg} - S)}_{<0} + S \ge u_{\rg} - S + S = u_{\rg}
\]
In this context, ``$\rg$" is pre-shock state and ``$\lf$" is post-shock state.
Similarly as for the 1-wave, we obtain that   $u_{\lf} < u_{\rg}$ and $\Pres_{\rg} > \Pres_{\lf}$.\\
(3) Dividing the jump conditions for continuity  and $x$-momentum equations, we get
\[
(\Pres_{\rg} + h_{\rg} u_{\rg}^2 - \Pres_{\lf} - h_{\lf} u_{\lf}^2)(h_{\rg} - h_{\lf}) = h_{\rg}^2 u_{\rg}^2 + h_{\lf}^2 u_{\lf}^2 - 2 h_{\lf} h_{\rg} u_{\lf} u_{\rg}
\]
Simplifying we obtain a quadratic equation
\[
u_{\rg}^2 - 2 u_{\lf} u_{\rg} + u_{\lf}^2 - \frac{(h_{\rg}-h_{\lf})(\Pres_{\rg} - \Pres_{\lf})}{h_{\rg} h_{\lf}} = 0
\]
whose solution is
\[
u_{\rg} = u_{\lf} \pm \sqrt{ \frac{(h_{\rg}-h_{\lf})(\Pres_{\rg} - \Pres_{\lf})}{h_{\rg} h_{\lf}} }
\]
If we pick the minus sign, then we satisfy the conditions in part (1) and (2) of the lemma which yields~\eqref{eq:shockulur}.
\end{proof}
\section{Exact solution of 1-D Riemann problem}
The Riemann problem is an initial value problem where the initial data is discontinuous at a single point.  The Riemann problem is to find $\con(t,x)$ solution of the SSW system~\eqref{eq:ssw1d}, with the following initial data
\begin{equation} \label{eq:rp1d}
  \con(t=0,x) =  \left \{
    \begin{array}{rcl}
         \con_{\lf} & \text { if } & x<0\\[3mm]
         \con_{\rg} & \text { if } & x>0
    \end{array}
    \right .
\end{equation}
\begin{lemma}
The solution of the Riemann problem with states $\con_{\lf},
\con_{\rg}$ gives rise to four intermediate states denoted by
$\con_{*\lf}, \con_{**\lf}, \con_{**\rg}, \con_{*\rg}$ which satisfy
the following ten relations, see Figure~\ref{SSW_RPm_Waves}.
\[
u_{*\lf} = u_{**\lf} = u_{**\rg} = u_{*\rg}, \qquad
\Pres_{*\lf} =\Pres_{**\lf} = \Pres_{**\rg} = \Pres_{*\rg}
\]
\[ h_{*\lf} = h_{**\lf}, \quad h_{**\rg} = h_{*\rg},
\quad v_{**\lf} = v_{**\rg}, \quad \R_{12}^{**\lf} = \R_{12}^{**\rg}.
\]
Using the definition of the total
pressure, a consequence the previous relations is that $\R_{11}^{*\lf}
= \R_{11}^{**\lf}$ and  $\R_{11}^{**\rg} = \R_{11}^{*\rg}$.
\end{lemma}
 \begin{figure}[tbp]
 \begin{center}
  \begin{tikzpicture}[scale=3.80]
 \coordinate (X0) at (   0.00000000     ,   0.00000000     );
 \coordinate (Xm) at (  -1.40129846E-45 ,   0.00000000     );
 \coordinate (Xp) at (   1.40129846E-45 ,   0.00000000     );
 \coordinate (Xn) at (   0.00000000     ,   1.40129846E-45 );
 \coordinate (Xd001) at (  -1.61687970     ,   0.00000000     );
 \coordinate (Xu001) at (  -1.61687970     ,   1.45000005     );
 \coordinate (Xd002) at (  -1.56445384     ,   0.00000000     );
 \coordinate (Xu002) at (  -1.56445384     ,   1.45000005     );
 \coordinate (Xd003) at (  -1.51500678     ,   0.00000000     );
 \coordinate (Xu003) at (  -1.51500678     ,   1.45000005     );
 \coordinate (Xd004) at (  -1.46828258     ,   0.00000000     );
 \coordinate (Xu004) at (  -1.46828258     ,   1.45000005     );
 \coordinate (Xd005) at (  -1.42405391     ,   0.00000000     );
 \coordinate (Xu005) at (  -1.42405391     ,   1.45000005     );
 \coordinate (Xd006) at (  -1.38211858     ,   0.00000000     );
 \coordinate (Xu006) at (  -1.38211858     ,   1.45000005     );
 \coordinate (Xd007) at (  -1.34229493     ,   0.00000000     );
 \coordinate (Xu007) at (  -1.34229493     ,   1.45000005     );
 \coordinate (Xd008) at (  -1.30442047     ,   0.00000000     );
 \coordinate (Xu008) at (  -1.30442047     ,   1.45000005     );
 \coordinate (Xd009) at (  -1.26834857     ,   0.00000000     );
 \coordinate (Xu009) at (  -1.26834857     ,   1.45000005     );
 \coordinate (Xd010) at (  -1.23394716     ,   0.00000000     );
 \coordinate (Xu010) at (  -1.23394716     ,   1.45000005     );
 \coordinate (Xd011) at (  -1.20109665     ,   0.00000000     );
 \coordinate (Xu011) at (  -1.20109665     ,   1.45000005     );
 \coordinate (Xd012) at (  -1.16968858     ,   0.00000000     );
 \coordinate (Xu012) at (  -1.16968858     ,   1.45000005     );
 \coordinate (Xd013) at (  -1.13962424     ,   0.00000000     );
 \coordinate (Xu013) at (  -1.13962424     ,   1.45000005     );
 \coordinate (Xd014) at (  -1.11081374     ,   0.00000000     );
 \coordinate (Xu014) at (  -1.11081374     ,   1.45000005     );
 \coordinate (Xd015) at (  -1.08317518     ,   0.00000000     );
 \coordinate (Xu015) at (  -1.08317518     ,   1.45000005     );
 \coordinate (Xd016) at (  -1.05663335     ,   0.00000000     );
 \coordinate (Xu016) at (  -1.05663335     ,   1.45000005     );
 \coordinate (Xd017) at (  -1.03111947     ,   0.00000000     );
 \coordinate (Xu017) at (  -1.03111947     ,   1.45000005     );
 \coordinate (Xd018) at (  -1.00657034     ,   0.00000000     );
 \coordinate (Xu018) at (  -1.00657034     ,   1.45000005     );
 \coordinate (Xd019) at ( -0.982927799     ,   0.00000000     );
 \coordinate (Xu019) at ( -0.982927799     ,   1.45000005     );
 \coordinate (Xd020) at ( -0.960138083     ,   0.00000000     );
 \coordinate (Xu020) at ( -0.960138083     ,   1.45000005     );
 \coordinate (Xd021) at ( -0.938151836     ,   0.00000000     );
 \coordinate (Xu021) at ( -0.938151836     ,   1.45000005     );
 \coordinate (Xd022) at ( -0.916923344     ,   0.00000000     );
 \coordinate (Xu022) at ( -0.916923344     ,   1.45000005     );
 \coordinate (Xd023) at ( -0.896410108     ,   0.00000000     );
 \coordinate (Xu023) at ( -0.896410108     ,   1.45000005     );
 \coordinate (Xd024) at ( -0.876572669     ,   0.00000000     );
 \coordinate (Xu024) at ( -0.876572669     ,   1.45000005     );
 \coordinate (Xd025) at ( -0.857374907     ,   0.00000000     );
 \coordinate (Xu025) at ( -0.857374907     ,   1.45000005     );
 \coordinate (Xd026) at ( -0.322138071     ,   0.00000000     );
 \coordinate (Xu026) at ( -0.322138071     ,   1.45000005     );
 \coordinate (Xd027) at (   4.08823341E-02 ,   0.00000000     );
 \coordinate (Xu027) at (   4.08823341E-02 ,   1.66750002     );
 \coordinate (Xd028) at (  0.344088137     ,   0.00000000     );
 \coordinate (Xu028) at (  0.344088137     ,   1.45000005     );
 \coordinate (Xd029) at (  0.933582962     ,   0.00000000     );
 \coordinate (Xu029) at (  0.933582962     ,   1.45000005     );
 \coordinate (Xd030) at (  0.949876368     ,   0.00000000     );
 \coordinate (Xu030) at (  0.949876368     ,   1.45000005     );
 \coordinate (Xd031) at (  0.966609538     ,   0.00000000     );
 \coordinate (Xu031) at (  0.966609538     ,   1.45000005     );
 \coordinate (Xd032) at (  0.983802259     ,   0.00000000     );
 \coordinate (Xu032) at (  0.983802259     ,   1.45000005     );
 \coordinate (Xd033) at (   1.00147581     ,   0.00000000     );
 \coordinate (Xu033) at (   1.00147581     ,   1.45000005     );
 \coordinate (Xd034) at (   1.01965249     ,   0.00000000     );
 \coordinate (Xu034) at (   1.01965249     ,   1.45000005     );
 \coordinate (Xd035) at (   1.03835630     ,   0.00000000     );
 \coordinate (Xu035) at (   1.03835630     ,   1.45000005     );
 \coordinate (Xd036) at (   1.05761254     ,   0.00000000     );
 \coordinate (Xu036) at (   1.05761254     ,   1.45000005     );
 \coordinate (Xd037) at (   1.07744837     ,   0.00000000     );
 \coordinate (Xu037) at (   1.07744837     ,   1.45000005     );
 \coordinate (Xd038) at (   1.09789252     ,   0.00000000     );
 \coordinate (Xu038) at (   1.09789252     ,   1.45000005     );
 \coordinate (Xd039) at (   1.11897564     ,   0.00000000     );
 \coordinate (Xu039) at (   1.11897564     ,   1.45000005     );
 \coordinate (Xd040) at (   1.14073050     ,   0.00000000     );
 \coordinate (Xu040) at (   1.14073050     ,   1.45000005     );
 \coordinate (Xd041) at (   1.16319227     ,   0.00000000     );
 \coordinate (Xu041) at (   1.16319227     ,   1.45000005     );
 \coordinate (Xd042) at (   1.18639874     ,   0.00000000     );
 \coordinate (Xu042) at (   1.18639874     ,   1.45000005     );
 \coordinate (Xd043) at (   1.21038997     ,   0.00000000     );
 \coordinate (Xu043) at (   1.21038997     ,   1.45000005     );
 \coordinate (Xd044) at (   1.23520935     ,   0.00000000     );
 \coordinate (Xu044) at (   1.23520935     ,   1.45000005     );
 \coordinate (Xd045) at (   1.26090300     ,   0.00000000     );
 \coordinate (Xu045) at (   1.26090300     ,   1.45000005     );
 \coordinate (Xd046) at (   1.28752148     ,   0.00000000     );
 \coordinate (Xu046) at (   1.28752148     ,   1.45000005     );
 \coordinate (Xd047) at (   1.31511819     ,   0.00000000     );
 \coordinate (Xu047) at (   1.31511819     ,   1.45000005     );
 \coordinate (Xd048) at (   1.34375131     ,   0.00000000     );
 \coordinate (Xu048) at (   1.34375131     ,   1.45000005     );
 \coordinate (Xd049) at (   1.37348390     ,   0.00000000     );
 \coordinate (Xu049) at (   1.37348390     ,   1.45000005     );
 \coordinate (Xd050) at (   1.40438342     ,   0.00000000     );
 \coordinate (Xu050) at (   1.40438342     ,   1.45000005     );
 \coordinate (Xd051) at (   1.43652391     ,   0.00000000     );
 \coordinate (Xu051) at (   1.43652391     ,   1.45000005     );
 \coordinate (Xd052) at (   1.46998525     ,   0.00000000     );
 \coordinate (Xu052) at (   1.46998525     ,   1.45000005     );
 \coordinate (Xd053) at (   1.50485456     ,   0.00000000     );
 \coordinate (Xu053) at (   1.50485456     ,   1.45000005     );
 \coordinate (Xw001) at (  -1.73006141     ,   0.00000000     );
 \coordinate (Xw002) at ( -0.589756489     ,   0.00000000     );
 \coordinate (Xw003) at ( -0.140627861     ,   0.00000000     );
 \coordinate (Xw004) at (  0.192485243     ,   0.00000000     );
 \coordinate (Xw005) at (  0.638835549     ,   0.00000000     );
 \coordinate (Xw006) at (   1.61019444     ,   0.00000000     );
 \draw[red,line width=0.8mm,dashed]  (Xd001) -- (Xu001) node[near end,sloped,above] {\pmb{$\xi = u_{\lf} - \coefa_{\lf}$}};
 \draw[blue,line width=0.2mm,dashed]  (Xd002) -- (Xu002);
 \draw[blue,line width=0.2mm,dashed]  (Xd003) -- (Xu003);
 \draw[blue,line width=0.2mm,dashed]  (Xd004) -- (Xu004);
 \draw[blue,line width=0.2mm,dashed]  (Xd005) -- (Xu005);
 \draw[blue,line width=0.2mm,dashed]  (Xd006) -- (Xu006);
 \draw[blue,line width=0.2mm,dashed]  (Xd007) -- (Xu007);
 \draw[blue,line width=0.2mm,dashed]  (Xd008) -- (Xu008);
 \draw[blue,line width=0.2mm,dashed]  (Xd009) -- (Xu009);
 \draw[blue,line width=0.2mm,dashed]  (Xd010) -- (Xu010);
 \draw[blue,line width=0.2mm,dashed]  (Xd011) -- (Xu011);
 \draw[magenta,line width=0.8mm]  (Xd012) -- (Xu012) node[near end,sloped,above] {};
 \draw (Xu012) node[anchor=south] {    $ \begin{array}{c} \\ \text{ Rarefaction } \\ \text{ or Shock }\\\pmb{\xi = S_{\lf}}\end{array}   $};
 \draw[blue,line width=0.2mm,dashed]  (Xd013) -- (Xu013);
 \draw[blue,line width=0.2mm,dashed]  (Xd014) -- (Xu014);
 \draw[blue,line width=0.2mm,dashed]  (Xd015) -- (Xu015);
 \draw[blue,line width=0.2mm,dashed]  (Xd016) -- (Xu016);
 \draw[blue,line width=0.2mm,dashed]  (Xd017) -- (Xu017);
 \draw[blue,line width=0.2mm,dashed]  (Xd018) -- (Xu018);
 \draw[blue,line width=0.2mm,dashed]  (Xd019) -- (Xu019);
 \draw[blue,line width=0.2mm,dashed]  (Xd020) -- (Xu020);
 \draw[blue,line width=0.2mm,dashed]  (Xd021) -- (Xu021);
 \draw[blue,line width=0.2mm,dashed]  (Xd022) -- (Xu022);
 \draw[blue,line width=0.2mm,dashed]  (Xd023) -- (Xu023);
 \draw[blue,line width=0.2mm,dashed]  (Xd024) -- (Xu024);
 \draw[red,line width=0.8mm, dashed]  (Xd025) -- (Xu025) node[near end,sloped,below] {\pmb{$\xi = u_{*} - \coefa_{*\lf}$}};
 \draw[blue,line width=0.8mm]  (Xd026) -- (Xu026) node[near end,sloped,above] {};
 \draw (Xu026) node[anchor=south] {   $ \begin{array}{c} \\ \text{ Shear }\\  \pmb{\xi = u_{*} - \coefb_{*\lf}}\end{array} $};
 \draw[line width=0.8mm]  (Xd027) -- (Xu027) node[near end,sloped,above] {};
 \draw (Xu027) node[anchor=south] {$     \begin{array}{c} \text{Contact}\\\pmb{\xi = u_{*}} \end{array}$} ;
 \draw[blue,line width=0.8mm]  (Xd028) -- (Xu028) node[near end,sloped,below] {};
 \draw (Xu028) node[anchor=south] {  $ \begin{array}{c} \\ \text{ Shear }\\  \pmb{\xi = u_{*} +\coefb_{*\rg}}\end{array} $};
 \draw[red,line width=0.8mm,dashed]  (Xd029) -- (Xu029) node[near end,sloped,above] {\pmb{$\xi = u_{*} + \coefa_{*\rg}$}};
 \draw[blue,line width=0.2mm,dashed]  (Xd030) -- (Xu030);
 \draw[blue,line width=0.2mm,dashed]  (Xd031) -- (Xu031);
 \draw[blue,line width=0.2mm,dashed]  (Xd032) -- (Xu032);
 \draw[blue,line width=0.2mm,dashed]  (Xd033) -- (Xu033);
 \draw[blue,line width=0.2mm,dashed]  (Xd034) -- (Xu034);
 \draw[blue,line width=0.2mm,dashed]  (Xd035) -- (Xu035);
 \draw[blue,line width=0.2mm,dashed]  (Xd036) -- (Xu036);
 \draw[blue,line width=0.2mm,dashed]  (Xd037) -- (Xu037);
 \draw[blue,line width=0.2mm,dashed]  (Xd038) -- (Xu038);
 \draw[blue,line width=0.2mm,dashed]  (Xd039) -- (Xu039);
 \draw[magenta,line width=0.8mm]  (Xd040) -- (Xu040) node[near end,sloped,below] {};
 \draw (Xu040) node[anchor=south] {    $ \begin{array}{c}  \text{ Rarefaction } \\ \text{ or Shock }\\ \pmb{\xi = S_{\rg}} \end{array}   $};
 \draw[blue,line width=0.2mm,dashed]  (Xd041) -- (Xu041);
 \draw[blue,line width=0.2mm,dashed]  (Xd042) -- (Xu042);
 \draw[blue,line width=0.2mm,dashed]  (Xd043) -- (Xu043);
 \draw[blue,line width=0.2mm,dashed]  (Xd044) -- (Xu044);
 \draw[blue,line width=0.2mm,dashed]  (Xd045) -- (Xu045);
 \draw[blue,line width=0.2mm,dashed]  (Xd046) -- (Xu046);
 \draw[blue,line width=0.2mm,dashed]  (Xd047) -- (Xu047);
 \draw[blue,line width=0.2mm,dashed]  (Xd048) -- (Xu048);
 \draw[blue,line width=0.2mm,dashed]  (Xd049) -- (Xu049);
 \draw[blue,line width=0.2mm,dashed]  (Xd050) -- (Xu050);
 \draw[blue,line width=0.2mm,dashed]  (Xd051) -- (Xu051);
 \draw[blue,line width=0.2mm,dashed]  (Xd052) -- (Xu052);
 \draw[red,line width=0.8mm, dashed]  (Xd053) -- (Xu053) node[near end,sloped,below] {\pmb{$\xi = u_{\rg} + \coefa_{\rg}$}};
 \draw (Xw001) node[anchor=south east] {$\Large   \begin{array}{c} h_{\lf}\\[2mm] u_{\lf}\\[2mm] v_{\lf}\\[2mm] \Pres_{\lf}\\[2mm]   \R_{12}^{\lf} \\[2mm] \R_{22}^{\lf}  \end{array} $};
 \draw (Xw002) node[anchor=south west] {$\!\!\!\!\!\!\Large   \begin{array}{c} \\ h_{*\lf}\\[2mm] u_{*}\\[2mm] v_{*\lf}\\[2mm] \Pres_{*}\\[2mm]   \R_{12}^{*\lf} \\[2mm] \R_{22}^{*\lf}  \end{array} $\!};
 \draw (Xw003) node[anchor=south] {$\Large   \begin{array}{c} h_{*\lf}\\[2mm] u_{*}\\[2mm] v_{**}\\[2mm] \Pres_{*}\\[2mm]   \R_{12}^{**} \\[2mm] \R_{22}^{**\lf}  \end{array} $};
 \draw (Xw004) node[anchor=south] {$\Large   \begin{array}{c} h_{*\rg}\\[2mm] u_{*}\\[2mm] v_{**}\\[2mm] \Pres_{*}\\[2mm]   \R_{12}^{**} \\[2mm] \R_{22}^{**\rg}  \end{array} $};
 \draw (Xw005) node[anchor=south east] {$\Large   \begin{array}{c} h_{*\rg}\\[2mm] u_{*}\\[2mm] v_{*\rg}\\[2mm] \Pres_{*}\\[2mm]   \R_{12}^{*\rg} \\[2mm] \R_{22}^{*\rg}  \end{array} \!\!\!\!\!\!$};
 \draw (Xw006) node[anchor=south west] {$\Large   \begin{array}{c} h_{\rg}\\[2mm] u_{\rg}\\[2mm] v_{\rg}\\[2mm] \Pres_{\rg}\\[2mm]   \R_{12}^{\rg} \\[2mm] \R_{22}^{\rg}  \end{array} $};
   \end{tikzpicture}
   \caption{Shear Shallow Water (SSW) model: Wave structure of the 1-D Riemann problem. Usefull set  of variables in the intermediate states for the computation of  the analytical solution}
\label{SSW_RPm_Waves}
\end{center}
\end{figure}
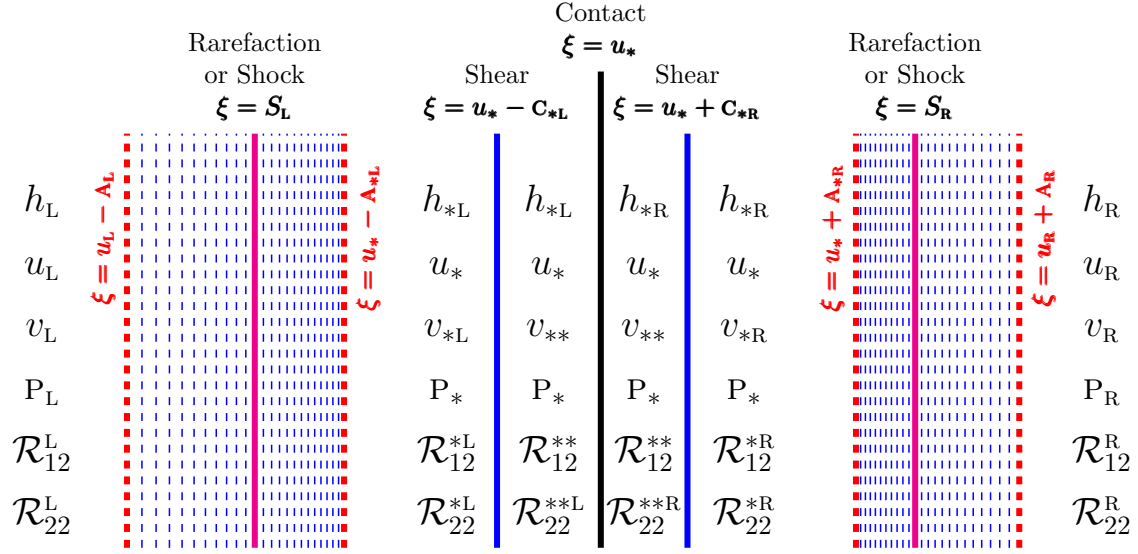

The solution is obtained by using the constancy of total pressure ($\Pres_*$) and normal velocity ($u_*$) inside the Riemann fan. If the 1-wave is a rarefaction, then $h_{*\lf} \le h_{\lf}$ while if it is a shock, then $h_{*\lf} \in (h_{\lf}, 2h_{\lf})$. Similarly, if the 6-wave is a rarefaction then $h_{*\rg} < h_{\rg}$, while if it is a shock, then $h_{*\rg} \in (h_{\rg}, 2h_{\rg})$. The total pressure in the first intermediate state can be written as
\begin{equation}
\Pres_{*\lf} =  
\left \{
\begin{array}{lr}
\left( \frac{h_{*\lf}}{{h_{\lf}}}\right)^3 \R_{11}^{\lf} + \half g h_{*\lf}^2 &\text{ for a 1-rarefaction : }  h_{*\lf} \le h_{\lf} \\[3mm]
\frac{1}{2 h_{\lf} - h_{*\lf}} \left[ (2 h_{*\lf} - h_{\lf}) \R_{11}^{\lf} - \frac{g (h_{\lf} - h_{*\lf})^3}{2} \right]+ \half g h_{*\lf}^2&
\text{for a 1-shock : }  h_{*\lf} > h_{\lf}
\end{array} \right .
\label{eq:totpsl}
\end{equation}
The velocity is given by
\[
u_{*\lf} = \left \{
\begin{array}{lr}
u_{\lf}  -  \left[ \rule{0mm}{4mm} \Afun(h_{*\lf},c_{\lf}) -\Afun(h_{\lf}, c_{\lf}) \right] & \text{ for a 1-rarefaction : }  h_{*\lf} \le h_{\lf} \\[3mm]
u_{\lf} - \sqrt{ \frac{(h_{*\lf}-h_{\lf})(P_{*\lf} - P_{\lf})}{h_{*\lf} h_{\lf}}}& \text{ for a 1-shock : }  h_{*\lf} > h_{\lf}
\end{array} \right .
\]
For the 6-wave, and given right state, we have
\[
\Pres_{*\rg} =  \left \{
\begin{array}{lr}
\left( \frac{h_{*\lf}}{{h_{\rg}}}\right)^3 \R_{11}^{\rg} + \half g h_{*\rg}^2 &\text{ for a 6-rarefaction : }  h_{*\rg} \le h_{\rg} \\[3mm]
\frac{1}{2 h_{\rg} - h_{*\rg}} \left[ (2 h_{*\rg} - h_{\rg}) \R_{11}^{\rg} - \frac{g (h_{\rg} - h_{*\rg})^3}{2} \right] + \half g h_{*\rg}^2&
\text{for a  6-shock : }  h_{*\rg} > h_{\rg}
\end{array} \right .
\]
and
\[
u_{*\rg} = \left \{
\begin{array}{lr}
u_{\rg}  +\left[ \rule{0mm}{4mm}  \Afun(h_{*\rg},c_{\rg})  - \Afun(h_{\rg}, c_{\rg})  \right] & \text{ for a 6-rarefaction : }  h_{*\rg} \le h_{\rg} \\[3mm]
u_{\rg} + \sqrt{ \frac{(h_{*\rg}-h_{\rg})(P_{*\rg} - P_{\rg})}{h_{*\rg} h_{\rg}}}&
 \text{for a  6-shock : }  h_{*\rg} > h_{\rg}
\end{array} \right .\]
We now want to determine $z_{\lf} = \frac{h_{*\lf}}{h_{\lf}}$ and $z_{\rg} = \frac{h_{*\rg}}{h_{\rg}}$ such that the total pressure and the velocity obtained from the 1-wave matches with those obtained from the 6-wave:
\[
\Pres_{*\lf}  - \Pres_{*\rg}  = 0 \qquad \text{ and } \qquad u_{*\lf} - u_{*\rg} = 0
\]
We define the functions $f(z; h, \R_{11}) $ for the total pressure and $g_\pm(z; h, u, \R_{11}) $ for the velocity as
\[
f(z; h, \R_{11}) = \begin{cases}
z^3\R_{11} + \half g z^2 h^2 & 0 < z \le 1 \\[10pt]
\frac{2 z - 1}{2 - z} \R_{11} + \half g h^2 \frac{(z - 1)^3}{2 - z} + \half g z^2 h^2 & 1 < z < 2
\end{cases}
\]
and
\[
g_\pm(z; h, u, \R_{11}) = \begin{cases}
u \pm [\Afun(z h, c) - \Afun(h,c)] & 0 < z \le 1 \\[10pt]
u  \pm \sqrt{ \frac{(z-1)[f(z;h,\R_{11}) - \R_{11} - \half g h^2 ]}{z h} } & 1 < z < 2
\end{cases}
\]
where $c = \R_{11}/h^3$. The problem can now be stated as:
\begin{equation}
  \text{ find $\quad z_{\lf}, z_{\rg} \in (0,2)\quad $ such that} \qquad
  \left \{
  \begin{array}{rcl}
    F(z_{\lf},z_{\rg}) &=& 0\\
    G(z_{\lf},z_{\rg}) &=& 0
  \end{array}
  \right .
  \label{eq:HugoniotRP}
\end{equation}
where
\begin{equation}\label{eq:FG}
\begin{aligned}
F(z_1,z_2) ~=~& f(z_1; h_{\lf}, \R_{11}^{\lf}) - f(z_2; h_{\rg}, \R_{11}^{\rg})  \\
G(z_1,z_2) ~=~& g_-(z_1; h_{\lf}, u_{\lf}, \R_{11}^{\lf}) - g_+(z_2; h_{\rg}, u_{\rg}, \R_{11}^{\rg})
\end{aligned}
\end{equation}
If the solution is such that $z_{\lf} \in (0,1)$ then the 1-wave is a rarefaction, and otherwise if $z_{\lf} \in (1,2)$, then it is a 1-shock.
 Similar interpretation applies to the 6-wave. The roots can be obtained by a Newton method as described in Appendix \ref{sec:root}.

We can numerically investigate the above functions $F,G$ by plotting contours of their level sets. For a given Riemann data of dam break problem from Section~\ref{sec:dambreak}, we plot contours of $F,G$ and also plot their zero contour lines. The solution is at the intersection of the zero contour lines of the two functions. In the Figure~\ref{fig:dambreak}, the bold solid lines are the zero level curves of $F,G$ and we see that they intersect at a unique point, which is approximately
\[
z_{\lf} = 0.731428410320821, \qquad z_{\rg} = 1.4177231168358784
\]
Hence the 1-wave is a 1-rarefaction and the 6-wave is a 6-shock.  We observe that the level curves of $F,G$ have a monotonic behaviour which implies that they intersect at a unique point and we now prove this behaviour in the general case.

\begin{figure}
\begin{center}
\includegraphics[width=0.60\textwidth]{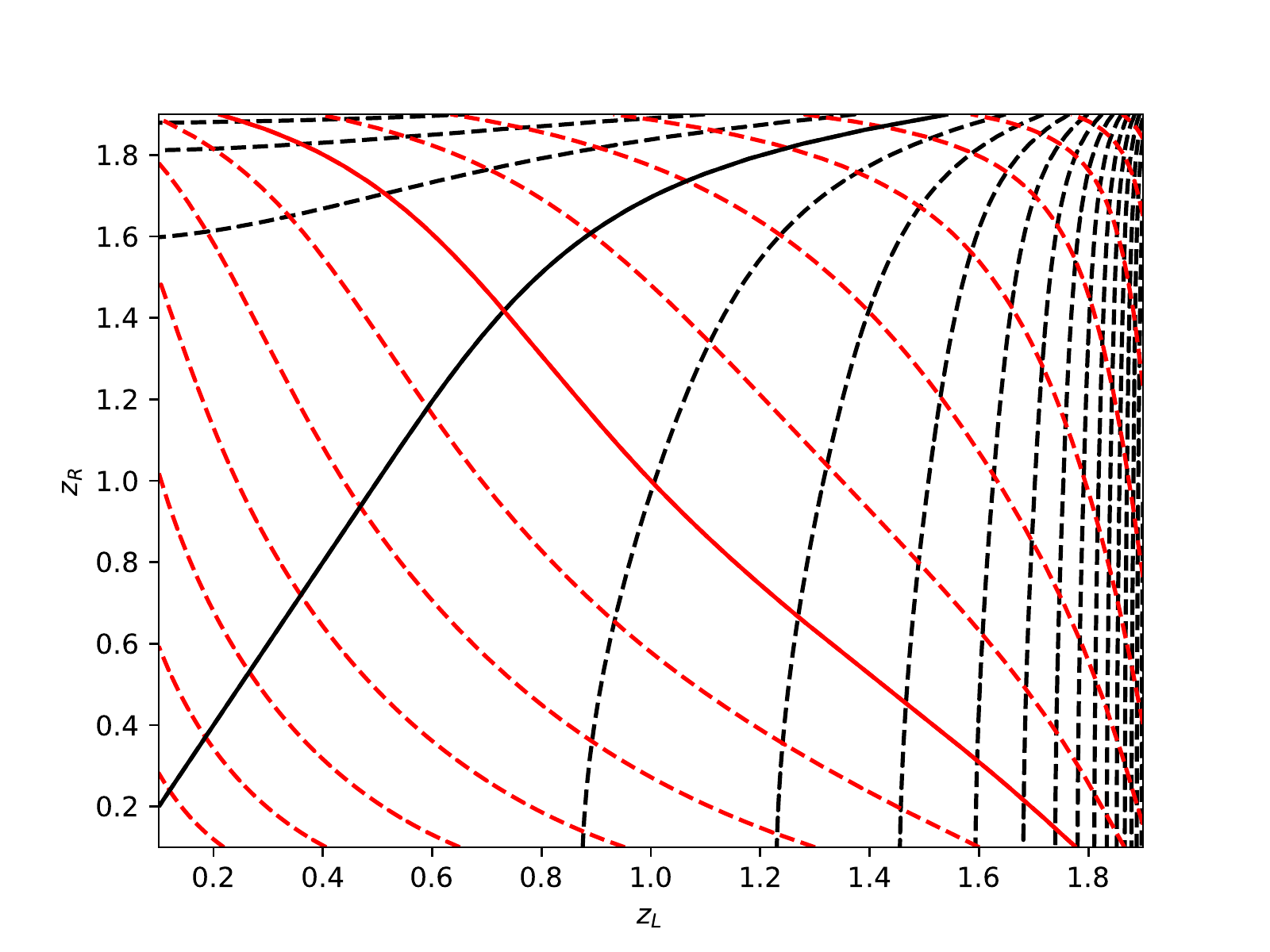}
\caption{Contours of $F$ (black) and $G$ (red) for dam break
  problem. Solid lines are  where the functions are zero. Intersection of the solid lines gives the desired $z_{\lf}$ and $z_{\rg}$.}
\label{fig:dambreak}
\end{center}
\end{figure}
\Old{ \clr{green}{
\paragraph{Delete this}
\[
g_\pm(z; h, u, \R_{11}) = \begin{cases}
u \pm \int_{h}^{zh} \frac{1}{\xi} \sqrt{ g \xi + 3 \xi^2 \R_{11} / h^3} \ud \xi & 0 < z \le 1 \\[10pt]
u  \pm \sqrt{ \frac{(z-1)[f(z;h,\R_{11}) - \R_{11} - \half g h^2 ]}{z h} } & 1 < z < 2
\end{cases}
\]
The minus sign corresponds to 1-rarefaction and the plus sign to 6-rarefaction. We can write the integrals by a change of variable as
\[
g_\pm(z; h, u, \R_{11}) = \begin{cases}
u \pm \int_{1}^{z} \frac{1}{\xi} \sqrt{ g h \xi + 3 \xi^2 \R_{11} / h} \ud \xi & 0 < z \le 1 \\[10pt]
u  \pm \sqrt{ \frac{(z-1)[f(z;h,\R_{11}) - \R_{11} - \half g h^2 ]}{z h} } & 1 < z < 2
\end{cases}
\]
}
}
\begin{theorem}
Assume that the two states in the Riemann problem are positive
($h_\lf, h_\rg > 0$). If
\begin{equation}
u_{\rg} - u_{\lf} < \Afun(h_{\lf},c_{\lf}) + \Afun(h_{\rg},c_{\rg})
\label{eq:riemcond}
\end{equation}
where $\Afun(h,c)$ is given by~\eqref{eq:Afun}, then  there exists a unique solution $(z_{\lf},z_{\rg}) \in (0,2) \times (0,2)$ such that
\begin{equation}\label{eq:FGeq}
F(z_{\lf},z_{\rg}) = 0\qquad \text{ and } \qquad G(z_{\lf},z_{\rg}) = 0
\end{equation}
where $F,G$ are given by~\eqref{eq:FG}. In this context, the Riemann
problem has a unique, positive solution.
\end{theorem}
\begin{proof} We want to show that the set of equations~\eqref{eq:FGeq} has a unique solution $(z_{\lf},z_{\rg}) \in (0,2) \times (0,2)$. Now
\[
\df{F}{z_{\lf}} = \df{}{z_{\lf}}f(z_{\lf};h_{\lf},\R_{11}^{\lf}) = \begin{cases}
3 z_{\lf}^2 \R_{11}^{\lf} + g z_{\lf} h_{\lf}^2, & 0 < z_{\lf} \le 1 \\
\frac{3}{(2 - z_{\lf})^2} \R_{11}^{\lf} + \half g h_{\lf}^2 \frac{z_{\lf}^2 - 4 z_{\lf} +5}{(2 - z_{\lf})^2}, & 1 \le z_{\lf} < 2
\end{cases}
\]
Hence $f(z; h, \R_{11})$ is an increasing function of $z \in (0,2)$ with $f(0; h, \R_{11})=0$ and $f(2; h, \R_{11}) = \infty$. Thus given any $z_\lf \in (0,2)$, the  equation $F(z_{\lf},z_{\rg}) = 0$ has a  unique solution $z_\rg \in (0,2)$ Now, since
\[
\df{F}{z_{\lf}} > 0, \qquad 
\df{F}{z_{\rg}} < 0, \qquad z_\lf, z_{\rg}\in (0,2)
\]
then by implicit function theorem, we have a  continuously differentiable function $z_{\rg} = \hz_{\rg}(z_{\lf})$, $z_{\lf} \in (0,2)$ such that $F(z_{\lf}, \hz_{\rg}(z_{\lf}))=0$. Moreover $F(0,0)=0$ so that $\hz_{\rg}(0)=0$. Now
\[
\dd{\hz_{\rg}}{z_{\lf}} = - \frac{\df{F}{z_{\lf}}}{\df{F}{z_{\rg}}} > 0, \qquad z_{\lf} \in (0,2)
\]
so that $\hz_{\rg}(z_{\lf})$ is an increasing function. Now
\[
\dd{\hz_{\rg}}{z_{\lf}}(0) = \lim_{z_{\lf} \to 0}\frac{3 z_{\lf}^2 \R_{11}^{\lf} + g z_{\lf} h_{\lf}^2}{3 \hz_{\rg}(z_{\lf})^2 \R_{11}^{\rg} + g \hz_{\rg}(z_{\lf}) h_{\rg}^2}
\]
which is of $0/0$ form. Applying L'Hopital rule, we get
\[
\dd{\hz_{\rg}}{z_{\lf}}(0) = \lim_{z_{\lf} \to 0}\frac{6 z_{\lf} \R_{11}^{\lf} + g h_{\lf}^2}{6 \hz_{\rg}(z_{\lf}) \dd{\hz_{\rg}}{z_{\lf}}(z_{\lf}) \R_{11}^{\rg} + g \dd{\hz_{\rg}}{z_{\lf}}(z_{\lf}) h_{\rg}^2} = \frac{h_{\lf}^2}{\dd{\hz_{\rg}}{z_{\lf}}(0) h_{\rg}^2} \quad\implies\quad \dd{\hz_{\rg}}{z_{\lf}}(0)  = \frac{h_{\lf}}{h_{\rg}} > 0
\]
As $z_{\lf} \to 2$, the first term of $F$ in~\eqref{eq:FG} which depends on $z_{\lf}$ goes to $\infty$ and this requires that $z_{\rg} \to 2$ also, i.e., $\hz_{\rg}(z_{\lf}) \to 2$. Moreover, $\hz_{\rg}(z_{\lf}) \ne 2$ for $z_{\lf} \in (0,2)$ since the second term in $F$ goes to $\infty$ as $z_{\rg} \to 2$. Hence the curve $(z_{\lf}, \hz_{\rg}(z_{\lf}))$ starts at $(0,0)$ and approaches $(2,2)$ in a monotonic way.

Now consider the function $G$ for which
\[
\df{G}{z_{\lf}} = \df{}{z_{\lf}}g_-(z_{\lf};h_{\lf},u_{\lf},\R_{11}^{\lf}) = \begin{cases}
- \frac{1}{z_{\lf}} \sqrt{g h_{\lf} z_{\lf} + 3 z_{\lf}^2 \R_{11}^{\lf}/h_{\lf}} & 0 < z_{\lf} \le 1\\[10pt]
- \frac{6 \R_{11}^{\lf} + \half g h_{\lf}^2 (z_{\lf}^3 - 3 z_{\lf} + 6)}{2[z_{\lf} (2 - z_{\lf})]^{3/2} h_{\lf}^{1/2} [3 \R_{11}^{\lf} + \half g h_{\lf}^2 (3-z_{\lf})]^{1/2} } & 1 \le z_{\lf} < 2
\end{cases}
\]
with a similar expression for $\df{G}{z_\rg}$. Hence
\[
\df{G}{z_{\lf}} < 0, \qquad \df{G}{z_{\rg}} < 0, \qquad z_\lf, z_{\rg} \in (0,2)
\]
Thus $g_-(z;h,u,\R_{11})$ is a decreasing function and $g_+(z;h,u,\R_{11})$ is an increasing function in $(0,2)$, see Figure~\ref{fig:rpsol}, and moreover
\[
g_-(0;h_\lf,u_\lf,\R_{11}^\lf) = u_\lf + \Afun(h_\lf,c_\lf), \qquad 
g_-(2;h_\lf,u_\lf,\R_{11}^\lf) = -\infty
\]
\[
g_+(0;h_\rg,u_\rg,\R_{11}^\rg) = u_\rg - \Afun(h_\rg,c_\rg), \qquad 
g_+(2;h_\rg,u_\rg,\R_{11}^\rg) = +\infty
\]
\begin{figure}
\begin{center}
\includegraphics[width=0.5\textwidth]{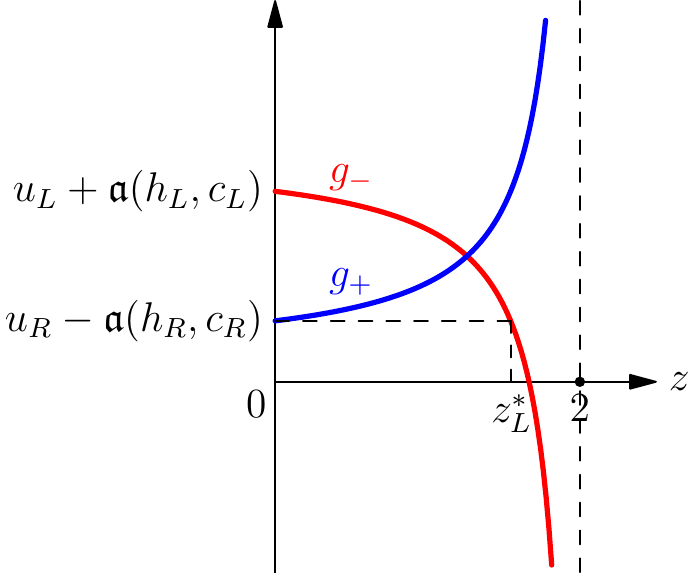}
\caption{Illustration of the functions $g_\pm$ under the condition~\eqref{eq:riemcond}}
\label{fig:rpsol}
\end{center}
\end{figure}
Under the assumption~\eqref{eq:riemcond}, we have $g_-(0;h_\lf,u_\lf,\R_{11}^\lf) > g_+(0;h_\rg,u_\rg,\R_{11}^\rg)$, and the equation $G(z_\lf, z_\rg)=0$ has a unique solution $z_\rg = \tilde z_\rg(z_\lf) \in [0,2)$ for all $z_\lf \in [0,z_\lf^*]$ with $\tilde z_\rg(z_\lf^*) = 0$ where $z_\lf^*$ satisfies $g_-(z_\lf^*; h_\lf, u_\lf, \R_{11}^\lf) = g_+(0;h_\rg,u_\rg,\R_{11}^\rg) = u_\rg - \Afun(h_\rg, c_\rg)$. By implicit function theorem, there is a  continuously differentiable function $z_{\rg} = \tz_{\rg}(z_{\lf})$, $z_{\lf} \in [0,z_\lf^*]$ such that $G(z_{\lf}, \tz_{\rg}(z_{\lf}))=0$. Now
\[
\dd{\tz_{\rg}}{z_{\lf}} = - \frac{\df{G}{z_{\lf}}}{\df{G}{z_{\rg}}} < 0, \qquad z_{\lf} \in [0,z_\lf^*]
\]
so that $\tilde z_{\rg}(z_{\lf})$ is a decreasing function for $z_{\lf} \in [0,z_\lf^*]$. 

We have shown that $\hat z_\rg : [0,2) \to [0,2)$ is increasing function with $\hat z_\rg(0) = 0$, $\lim_{z_\lf \to 2} \hat z_\rg(z_\lf) = 2$, and $\tilde z_\rg : [0,z_\lf^*] \to [0,2)$ is decreasing function with $\tilde z_\rg(0) \in (0,2)$, $\tilde z_\rg(z_\lf^*) = 0$, so they intersect at a unique point in $z_\lf \in (0,2)$ which is the desired solution.
\end{proof}

\subsection{Shock speed and jump conditions}
\label{sec:shockjump}
Suppose that the 1-wave is a 1-shock; then $h_{*\lf} = z_{\lf} h_{\lf}$ and using~\eqref{eq:r11shock}
\begin{equation}
\R_{11}^{*\lf} = \R_{11}(z_{\lf} h_{\lf}; h_{\lf}, \R_{11}^{\lf}) = \frac{2 z_{\lf} - 1}{2 - z_{\lf}} \R_{11}^{\lf} + \half g h_{\lf}^2 \frac{(z_{\lf} - 1)^3}{2 - z_{\lf}}
\label{eq:r11sl}
\end{equation}
while~\eqref{eq:shockulur} yields
\begin{equation}
u_* = u_{*\lf} = u_{\lf} - \sqrt{ \frac{(h_{*\lf}-h_{\lf})(\Pres_{*\lf} - \Pres_{\lf})}{h_{*\lf} h_{\lf}}}
\label{eq:ustar}
\end{equation}
The 1-shock speed can be  computed from the jump condition \eqref{eq:rh1}
\begin{eqnarray}
S_{\lf} &=& \frac{\jump{h u}}{\jump{h}} = \avg{u} + \avg{h} \frac{\jump{u}}{\jump{h}}
\label{eq:shockspeed}
= \frac{u_{\lf} + u_*}{2} + \left ( \frac{z_{\lf} + 1}{z_{\lf} - 1}\right )\frac{u_* - u_{\lf}}{2}
\end{eqnarray}
The jump conditions \eqref{eq:rh1}, \eqref{eq:rh2}, \eqref{eq:rh4}
have already been satisfied since they were used to determine the
Hugoniot curve. We can find $v_{*\lf}, \R_{12}^{*\lf}$ from
\eqref{eq:rh3}, \eqref{eq:rh5} which is a linear system of equations
\begin{equation}
  \left \{
\begin{array}{rcrcl}
h_{*\lf} (u_* - S_{\lf}) v_{*\lf} &+& \R_{12}^{*\lf} &=& a_1 \\[3mm]
\left( \half \R_{11}^{*\lf} + \half h_{*\lf} u_* (u_* - S_{\lf}) + \frac{1}{4} g h_{*\lf} (h_{*\lf} - h_{\lf})\right) v_{*\lf} &+& \left(u_* - \half S_{\lf}\right) \R_{12}^{*\lf} &=& a_2
\end{array}
\right .
\end{equation}
where
\begin{eqnarray*}
a_1 &=& h_{\lf} (u_{\lf} - S_{\lf})v_{\lf} + \R_{12}^{\lf}  \\
a_2 &=& (u_{\lf} - S_{\lf}) \E_{12}^{\lf} + \half (\R_{11}^{\lf} v_{\lf} + \R_{12}^{\lf} u_{\lf}) - \frac{1}{4} g h_{\lf} v_{\lf} (h_{*\lf} - h_{\lf})
\end{eqnarray*}
The determinant of the $2 \times 2$ matrix is
\[
Det = - \half \R_{11}^{*\lf} + \half h_{*\lf} (u_* - S_{\lf})^2 - \frac{1}{4} g h_{*\lf} (h_{*\lf} - h_{\lf})
\]
But using \eqref{eq:shockspeed} and \eqref{eq:ustar}
\[
u_* - S_{\lf} = - \frac{u_* - u_{\lf}}{z_{\lf}-1}  = \sqrt{\frac{(\Pres_{*\lf} - \Pres_{\lf})}{(z_{\lf} - 1) h_{*\lf}}}
\]
and hence,  using \eqref{eq:r11sl}, we get
\begin{eqnarray*}
  Det
  &=&
  \frac{2 - z_{\lf}}{2(z_{\lf}-1)} \R_{11}^{*\lf} - \frac{\R_{11}^{\lf} }{2(z_{\lf} - 1)} + \frac{1}{4} g h_{\lf}^2 (1 + 2 z_{\lf} - z_{\lf}^2)
=  R_{11}^{\lf} + \half g h_{\lf}^2
= \Pres_{\lf}
 > 0
\end{eqnarray*}
and hence the $2 \times 2$ system has a unique solution. Once $v_{*\lf}, \R_{12}^{*\lf}$ have been determined, we can compute $\R_{22}^{*\lf}$ from \eqref{eq:rh6}
\[
\E_{22}^{*\lf} = \frac{1}{u_* - S_{\lf}}\left[ (u_{\lf} - S_{\lf}) \E_{22}^{\lf} - (\R_{12}^{*\lf} v_{*\lf} - \R_{12}^{\lf} v_{\lf})\right], \qquad \R_{22}^{*\lf} = 2 \E_{22}^{*\lf} - h_{*\lf} v_{*\lf}^2
\]
We have thus satisfied all the jump conditions and completely determined the $\con_{*\lf}$ state. The jump conditions for a 6-shock can be satisfied in a similar way to determine the $\con_{*\rg}$ state.
\paragraph{Shock speed.}
If the 1-wave is a shock, then $1 < z_L < 2$ and from~\eqref{eq:totpsl}
\begin{equation}
\Pres_{*\lf} - \Pres_{\lf} = \frac{z_{\lf} - 1}{2 - z_{\lf}} \left[ 3 \R_{11}^{\lf} + \half g h_{\lf}^2 (3 - z_{\lf})\right]
\label{eq:dtotpsl}
\end{equation}
and the shock speed is given by
\begin{eqnarray}
\nonumber
S_{\lf} &=& \frac{z_{\lf} u_* - u_{\lf}}{z_{\lf}-1} \qquad \textrm{from \eqref{eq:shockspeed}}\\
&=&  u_{\lf} - \sqrt{ \frac{z_{\lf}}{2 - z_{\lf}} \left[ 3 \p_{11}^{\lf} + \half g h_{\lf} (3 - z_{\lf})\right] }, \qquad \textrm{from \eqref{eq:ustar} and \eqref{eq:dtotpsl}}
\label{eq:1shockspeed}
\end{eqnarray}
Similarly, the speed of the 6-shock is given by
\[
S_{\rg} = u_{\rg} + \sqrt{ \frac{z_{\rg}}{2 - z_{\rg}} \left[ 3 \p_{11}^{\rg} + \half g h_{\rg} (3 - z_{\rg})\right] }
\]
\begin{remark}
In HLL-type solvers, it is necessary to have estimates of the slowest and fastest speeds arising in the solution of the Riemann problem. If the 1-wave is a shock, then we would like a lower bound $\tilde{S}_{\lf}$ on this speed
\[
S_{\lf} \ge \tilde{S}_{\lf} := u_{\lf} - \sup_{z \in (1,2)} \sqrt{ \frac{z}{2 - z} \left[ 3 \p_{11}^{\lf} + \half g h_{\lf} (3 - z)\right] }
\]
But the supremum is $\infty$ and we do not get a useful lower bound.
\end{remark}
\subsection{Resumed computation of the intermediate states.}
For a Riemann problem, the left ($\con_{\lf}$)  and the right
($\con_{\rg}$) states are input data.
\begin{itemize}
\item
For given $\con_{\lf}$  and $\con_{\rg}$, the system
\eqref{eq:HugoniotRP} is solved and  $z_{\lf},z_{\rg}$ are obtained. Therefore,
\[
\begin{array}{rcl}
h_{*\lf} &= &z_{\lf} h_{\lf}, \\
h_{*\rg} &= &z_{\rg} h_{\rg},
\end{array}
\qquad
\begin{array}{rcll}
u_{*}  &=& g_-(z_{\lf}; h_{\lf}, u_{\lf}, \R_{11}^{\lf})  &= g_+(z_{\rg}; h_{\rg}, u_{\rg}, \R_{11}^{\rg}),
\\
\Pres_{*}  &= & f(z_{\lf}; h_{\lf}, \R_{11}^{\lf}) &= f(z_{\rg}; h_{\rg}, \R_{11}^{\rg}).
\end{array}
\]
The variables $h$ and $\Pres$ are now defined for all
  intermediate states. Using the definition of the total
  pressure $\Pres= \frac{gh^2}{2} + \R_{11}$, we can get $\R_{11}^{*\lf}$
  and  $\R_{11}^{*\rg}$ .

  \item When $z_{\lf} \le 1$  the 1-wave is a rarefaction. The associated Riemann invariants are used to
    compute $v_{*\lf}$, $\R_{12}^{*\lf}$  and $\R_{12}^{*\lf}$. The internal
    structure of the rarefaction is obtained by integration of
    equations for the 1-wave integral curve. Similarly,
    when $z_{\rg} \le 1$, Riemann invariants for 6-rarefaction are  used to
    compute $v_{*\rg}$, $\R_{12}^{*\rg}$  and $\R_{12}^{*\rg}$.

  \item When $z_{\lf} > 1$  the 1-wave is a shock. Then, generalized jump conditions  are used to
    compute $S_{\lf}$,  $v_{*\lf}$, $\R_{12}^{*\lf}$  and
    $\R_{22}^{*\lf}$.
    Similarly, when $z_{\rg} > 1$, the generalized jump conditions  are used to
    compute $S_{\rg}$,  $v_{*\rg}$, $\R_{12}^{*\rg}$  and $\R_{22}^{*\rg}$.

  \item At this step, $\con_{*\lf}$ and $\con_{*\rg}$ are defined. Using
    the appropriate Riemann invariants of the 2-wave, we get $v_{**}$,
    $\R_{12}^{**}$  and $\R_{22}^{*\lf}$. The invariants for the 5-wave
    give $\R_{22}^{*\rg}$.
\end{itemize}
    The computation of  intermediate states is then completed.
\subsection{Single shock solution}\label{sec:singshock}
Given the left state $(h_\lf, u_\lf, v_\lf, \p_{11}^\lf, \p_{12}^\lf, \p_{22}^\lf)$, let us find a right state that is connected by a 1-shock. We will take a value of $z = h_\rg/h_\lf \in (1,2)$. Then $h_\rg = z h_\lf$ and from the Hugoniot curve, we obtain
\[
\R_{11}^\rg = (2 h_\rg - h_\lf) \R_{11}^\lf/(2 h_\lf - h_\rg) - \half g (h_\lf - h_\rg)^3/(2 h_\lf - h_\rg)
\]
Then the velocity and shock speed are given by~\eqref{eq:ustar},~\eqref{eq:1shockspeed}
\[
u_\rg = u_\lf - \sqrt{\frac{(h_\rg - h_\lf)(\Pres_\rg - \Pres_\lf)}{h_\lf h_\rg}}, \qquad
S = u_\lf - \sqrt{\frac{z}{2-z}\left(3 \p_{11}^\lf + \half g h_\lf (3 - z) \right)}
\]
where $\Pres_\lf, \Pres_\rg$ are the total pressures. The remaining quantities can be computed using the procedure in Section~\ref{sec:shockjump}.
\subsection{Vacuum states}
A vacuum state refers to a zero value of water depth $h$ and is also called a dry state. For classical shallow water model, Riemann problems with vacuum states can be solved with rarefaction waves~\cite{Toro2001b,Pares2019}. The velocity in the vacuum state is allowed to be non-zero which is not physically meaningful since there is no fluid in this state, but we seek a mathematically correct solution. For the SSW model, let us consider a left non-vacuum state ($h_\lf > 0$, $\R^\lf > 0$) and a right vacuum state. In the vacuum state we also assume that the Reynolds tensor $\p^\rg=0$ and hence also $\R^\rg=0$. Let us first try to connect the states by a simple jump discontinuity moving at speed $S$. The jump condition of the $h$ equation yields $0 - h_{\lf} u_{\lf} = S (0 - h_{\lf})$ so that the discontinuity speed is $S = u_{\lf}$. From the jump condition of the $x$ momentum equation we get $0 - (\R_{11}^\lf + h_\lf u_{\lf}^2 + \half g h_\lf^2) = S(0 - h_\lf u_\lf) = - h_\lf u_\lf^2$ so that $\R_{11}^\lf + \half g h_\lf^2 = 0$, which implies that there is no solution.

We now try to connect the two states by a 1-rarefaction wave and make use of the invariants shown in equation~\eqref{eq:rareinv1}. The second invariant yields $u_\rg = u_\lf + \Afun(h_\lf,c_\lf)$ and the sixth invariant yields $v_\rg = v_\lf + \frac{2\p_{12}^\lf}{g h_\lf + 2 \p_{11}^\lf} \Afun(h_\lf,c_\lf)$. Similarly, if the left state is a vacuum state and the right state is a non-vacuum state, they can be connected by a 6-rarefaction wave.

\begin{figure}
\centering
\begin{tabular}{cc}
\includegraphics[width=0.3\textwidth]{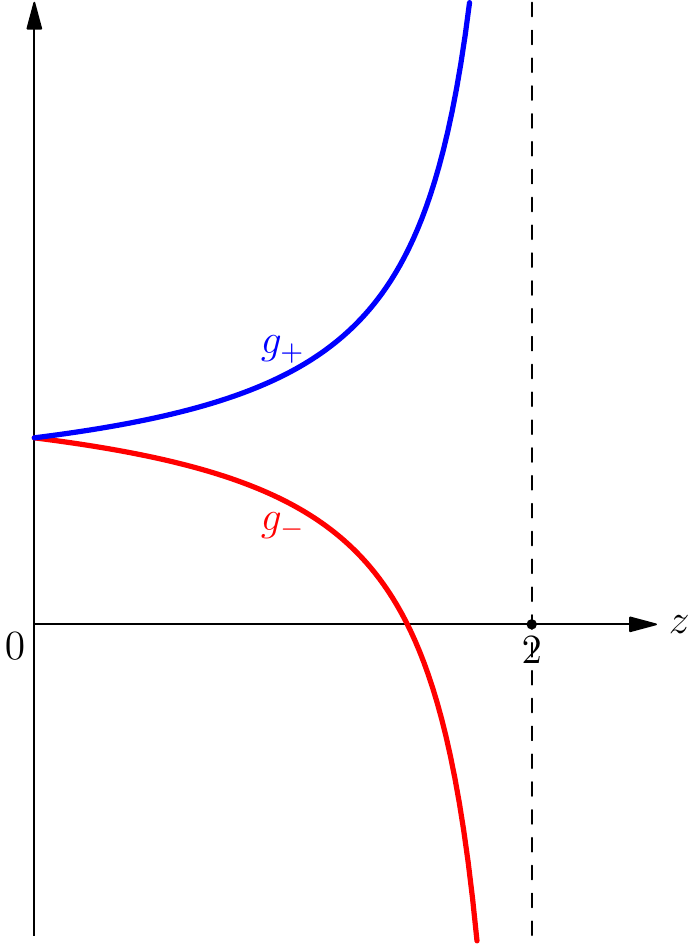} &
\includegraphics[width=0.3\textwidth]{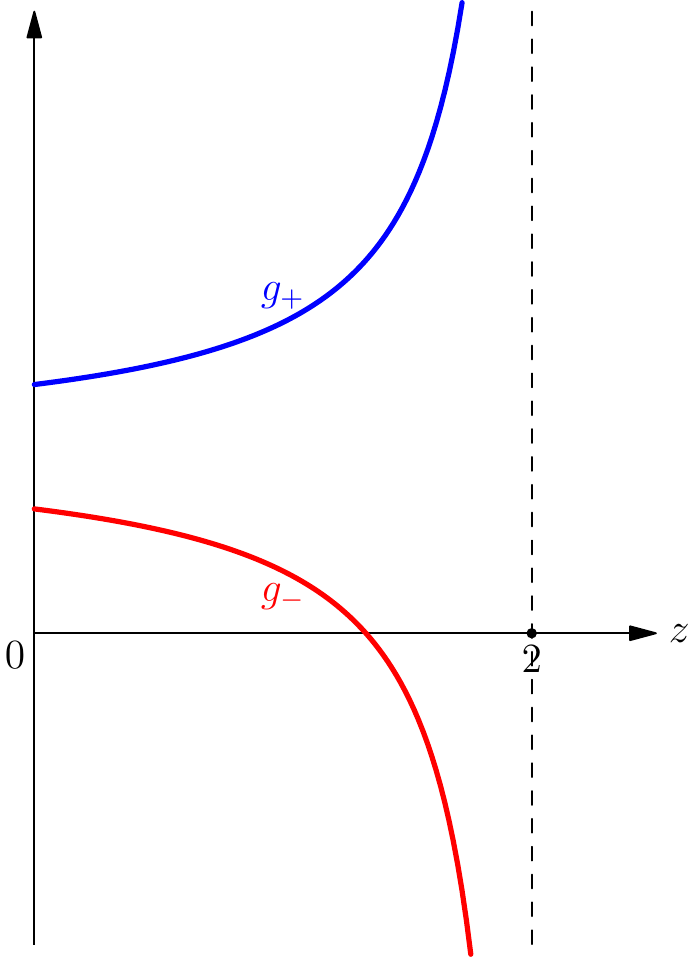} \\
(a) & (b) 
\end{tabular}
\caption{Illustration of the functions $g_\pm$ leading to intermediate vacuum state in the Riemann problem. (a) $u_{\rg} - u_{\lf} = \Afun(h_{\lf},c_{\lf}) + \Afun(h_{\rg},c_{\rg})$, (b) $u_{\rg} - u_{\lf} > \Afun(h_{\lf},c_{\lf}) + \Afun(h_{\rg},c_{\rg})$}
\label{fig:gmgp2}
\end{figure}
If an intermediate state is a vacuum state, say $h_{*\lf}=0$ then necessarily all the intermediate states in Figure~\ref{SSW_RPm_Waves} must be vacuum states, i.e., $h_{*\rg}=0$, since shear/contact waves cannot connect a vacuum state to a non-vacuum state. The constancy of $u$ in the intermediate states means that
\[
u_{*\lf} = u_\lf + \Afun(h_\lf, c_\lf) = u_\rg - \Afun(h_\rg, c_\rg) = u_{*\rg}
\]
i.e., we have equality in~\ref{eq:riemcond}. The functions $g_\pm$ in this case are shown in Figure~\ref{fig:gmgp2}a which shows that the solution of $G(z_\lf, z_\rg)=0$ is $z_\lf = z_\rg = 0$. On the other hand if $u_{\rg} - u_{\lf} > \Afun(h_{\lf},c_{\lf}) + \Afun(h_{\rg},c_{\rg})$ the functions $g_\pm$ are shown in Figure~\ref{fig:gmgp2}b and there is no solution to $G(z_\lf,z_\rg)=0$. But we can still construct a solution with a 1-rarefaction and 6-rarefaction with an intermediate vacuum state, but it will not be possible to find a proper solution that satisfies all the structure of the intermediate states as shown in Figure~\ref{SSW_RPm_Waves}, since $u_{*\lf} \ne u_{*\rg}$, see Figure~\ref{SSW_Vacuum_Waves}. However the momentum is constant and zero in the intermediate state which may be considered as a solution that satisfies all the jump conditions, but the velocity in the intermediate states is not well defined. In this sense, the solution of the Riemann problem can be extended to include vacuum states. We note that the solutions described in the next Theorem are admissible weak solutions, since they are continuous in $h$, $hv$, $E_{ij}$, variables and they reduce to smooth solutions in the intermediate regions (the 2 rarefaction waves and the vacuum states). We summarise the solution with vacuum states in the following theorem.
 \begin{figure}[tbp]
 \begin{center}
  \begin{tikzpicture}[scale=3.50]
 \coordinate (X0) at (   0.00000000     ,   0.00000000     );
 \coordinate (Xm) at (  -1.40129846E-45 ,   0.00000000     );
 \coordinate (Xp) at (   1.40129846E-45 ,   0.00000000     );
 \coordinate (Xn) at (   0.00000000     ,   1.40129846E-45 );
 \coordinate (Xd001) at (  -1.61687970     ,   0.00000000     );
 \coordinate (Xu001) at (  -1.61687970     ,   1.45000005     );
 \coordinate (Xd002) at (  -1.56445384     ,   0.00000000     );
 \coordinate (Xu002) at (  -1.56445384     ,   1.45000005     );
 \coordinate (Xd003) at (  -1.51500678     ,   0.00000000     );
 \coordinate (Xu003) at (  -1.51500678     ,   1.45000005     );
 \coordinate (Xd004) at (  -1.46828258     ,   0.00000000     );
 \coordinate (Xu004) at (  -1.46828258     ,   1.45000005     );
 \coordinate (Xd005) at (  -1.42405391     ,   0.00000000     );
 \coordinate (Xu005) at (  -1.42405391     ,   1.45000005     );
 \coordinate (Xd006) at (  -1.38211858     ,   0.00000000     );
 \coordinate (Xu006) at (  -1.38211858     ,   1.45000005     );
 \coordinate (Xd007) at (  -1.34229493     ,   0.00000000     );
 \coordinate (Xu007) at (  -1.34229493     ,   1.45000005     );
 \coordinate (Xd008) at (  -1.30442047     ,   0.00000000     );
 \coordinate (Xu008) at (  -1.30442047     ,   1.45000005     );
 \coordinate (Xd009) at (  -1.26834857     ,   0.00000000     );
 \coordinate (Xu009) at (  -1.26834857     ,   1.45000005     );
 \coordinate (Xd010) at (  -1.23394716     ,   0.00000000     );
 \coordinate (Xu010) at (  -1.23394716     ,   1.45000005     );
 \coordinate (Xd011) at (  -1.20109665     ,   0.00000000     );
 \coordinate (Xu011) at (  -1.20109665     ,   1.45000005     );
 \coordinate (Xd012) at (  -1.16968858     ,   0.00000000     );
 \coordinate (Xu012) at (  -1.16968858     ,   1.45000005     );
 \coordinate (Xd013) at (  -1.13962424     ,   0.00000000     );
 \coordinate (Xu013) at (  -1.13962424     ,   1.45000005     );
 \coordinate (Xd014) at (  -1.11081374     ,   0.00000000     );
 \coordinate (Xu014) at (  -1.11081374     ,   1.45000005     );
 \coordinate (Xd015) at (  -1.08317518     ,   0.00000000     );
 \coordinate (Xu015) at (  -1.08317518     ,   1.45000005     );
 \coordinate (Xd016) at (  -1.05663335     ,   0.00000000     );
 \coordinate (Xu016) at (  -1.05663335     ,   1.45000005     );
 \coordinate (Xd017) at (  -1.03111947     ,   0.00000000     );
 \coordinate (Xu017) at (  -1.03111947     ,   1.45000005     );
 \coordinate (Xd018) at (  -1.00657034     ,   0.00000000     );
 \coordinate (Xu018) at (  -1.00657034     ,   1.45000005     );
 \coordinate (Xd019) at ( -0.982927799     ,   0.00000000     );
 \coordinate (Xu019) at ( -0.982927799     ,   1.45000005     );
 \coordinate (Xd020) at ( -0.960138083     ,   0.00000000     );
 \coordinate (Xu020) at ( -0.960138083     ,   1.45000005     );
 \coordinate (Xd021) at ( -0.938151836     ,   0.00000000     );
 \coordinate (Xu021) at ( -0.938151836     ,   1.45000005     );
 \coordinate (Xd022) at ( -0.916923344     ,   0.00000000     );
 \coordinate (Xu022) at ( -0.916923344     ,   1.45000005     );
 \coordinate (Xd023) at ( -0.896410108     ,   0.00000000     );
 \coordinate (Xu023) at ( -0.896410108     ,   1.45000005     );
 \coordinate (Xd024) at ( -0.876572669     ,   0.00000000     );
 \coordinate (Xu024) at ( -0.876572669     ,   1.45000005     );
 \coordinate (Xd025) at ( -0.857374907     ,   0.00000000     );
 \coordinate (Xu025) at ( -0.857374907     ,   1.45000005     );
 \coordinate (Xd026) at ( -0.322138071     ,   0.00000000     );
 \coordinate (Xu026) at ( -0.322138071     ,   1.45000005     );
 \coordinate (Xd027) at (   4.08823341E-02 ,   0.00000000     );
 \coordinate (Xu027) at (   4.08823341E-02 ,   1.45000002     );
 \coordinate (Xd028) at (  0.344088137     ,   0.00000000     );
 \coordinate (Xu028) at (  0.344088137     ,   1.45000005     );
 \coordinate (Xd029) at (  0.933582962     ,   0.00000000     );
 \coordinate (Xu029) at (  0.933582962     ,   1.45000005     );
 \coordinate (Xd030) at (  0.949876368     ,   0.00000000     );
 \coordinate (Xu030) at (  0.949876368     ,   1.45000005     );
 \coordinate (Xd031) at (  0.966609538     ,   0.00000000     );
 \coordinate (Xu031) at (  0.966609538     ,   1.45000005     );
 \coordinate (Xd032) at (  0.983802259     ,   0.00000000     );
 \coordinate (Xu032) at (  0.983802259     ,   1.45000005     );
 \coordinate (Xd033) at (   1.00147581     ,   0.00000000     );
 \coordinate (Xu033) at (   1.00147581     ,   1.45000005     );
 \coordinate (Xd034) at (   1.01965249     ,   0.00000000     );
 \coordinate (Xu034) at (   1.01965249     ,   1.45000005     );
 \coordinate (Xd035) at (   1.03835630     ,   0.00000000     );
 \coordinate (Xu035) at (   1.03835630     ,   1.45000005     );
 \coordinate (Xd036) at (   1.05761254     ,   0.00000000     );
 \coordinate (Xu036) at (   1.05761254     ,   1.45000005     );
 \coordinate (Xd037) at (   1.07744837     ,   0.00000000     );
 \coordinate (Xu037) at (   1.07744837     ,   1.45000005     );
 \coordinate (Xd038) at (   1.09789252     ,   0.00000000     );
 \coordinate (Xu038) at (   1.09789252     ,   1.45000005     );
 \coordinate (Xd039) at (   1.11897564     ,   0.00000000     );
 \coordinate (Xu039) at (   1.11897564     ,   1.45000005     );
 \coordinate (Xd040) at (   1.14073050     ,   0.00000000     );
 \coordinate (Xu040) at (   1.14073050     ,   1.45000005     );
 \coordinate (Xd041) at (   1.16319227     ,   0.00000000     );
 \coordinate (Xu041) at (   1.16319227     ,   1.45000005     );
 \coordinate (Xd042) at (   1.18639874     ,   0.00000000     );
 \coordinate (Xu042) at (   1.18639874     ,   1.45000005     );
 \coordinate (Xd043) at (   1.21038997     ,   0.00000000     );
 \coordinate (Xu043) at (   1.21038997     ,   1.45000005     );
 \coordinate (Xd044) at (   1.23520935     ,   0.00000000     );
 \coordinate (Xu044) at (   1.23520935     ,   1.45000005     );
 \coordinate (Xd045) at (   1.26090300     ,   0.00000000     );
 \coordinate (Xu045) at (   1.26090300     ,   1.45000005     );
 \coordinate (Xd046) at (   1.28752148     ,   0.00000000     );
 \coordinate (Xu046) at (   1.28752148     ,   1.45000005     );
 \coordinate (Xd047) at (   1.31511819     ,   0.00000000     );
 \coordinate (Xu047) at (   1.31511819     ,   1.45000005     );
 \coordinate (Xd048) at (   1.34375131     ,   0.00000000     );
 \coordinate (Xu048) at (   1.34375131     ,   1.45000005     );
 \coordinate (Xd049) at (   1.37348390     ,   0.00000000     );
 \coordinate (Xu049) at (   1.37348390     ,   1.45000005     );
 \coordinate (Xd050) at (   1.40438342     ,   0.00000000     );
 \coordinate (Xu050) at (   1.40438342     ,   1.45000005     );
 \coordinate (Xd051) at (   1.43652391     ,   0.00000000     );
 \coordinate (Xu051) at (   1.43652391     ,   1.45000005     );
 \coordinate (Xd052) at (   1.46998525     ,   0.00000000     );
 \coordinate (Xu052) at (   1.46998525     ,   1.45000005     );
 \coordinate (Xd053) at (   1.50485456     ,   0.00000000     );
 \coordinate (Xu053) at (   1.50485456     ,   1.45000005     );
 \coordinate (Xw001) at (  -1.73006141     ,   0.00000000     );
 \coordinate (Xw002) at ( -0.589756489     ,   0.00000000     );
 \coordinate (Xw003) at ( -0.140627861     ,   0.00000000     );
 \coordinate (Xw004) at (  0.192485243     ,   0.00000000     );
 \coordinate (Xw005) at (  0.638835549     ,   0.00000000     );
 \coordinate (Xw006) at (   1.61019444     ,   0.00000000     );
 \draw[red,line width=0.8mm,dashed]  (Xd001) -- (Xu001) node[near end,sloped,above] {\pmb{$\xi = u_{\lf} - \coefa_{\lf}$}};
 \draw[blue,line width=0.2mm,dashed]  (Xd002) -- (Xu002);
 \draw[blue,line width=0.2mm,dashed]  (Xd003) -- (Xu003);
 \draw[blue,line width=0.2mm,dashed]  (Xd004) -- (Xu004);
 \draw[blue,line width=0.2mm,dashed]  (Xd005) -- (Xu005);
 \draw[blue,line width=0.2mm,dashed]  (Xd006) -- (Xu006);
 \draw[blue,line width=0.2mm,dashed]  (Xd007) -- (Xu007);
 \draw[blue,line width=0.2mm,dashed]  (Xd008) -- (Xu008);
 \draw[blue,line width=0.2mm,dashed]  (Xd009) -- (Xu009);
 \draw[blue,line width=0.2mm,dashed]  (Xd010) -- (Xu010);
 \draw[blue,line width=0.2mm,dashed]  (Xd011) -- (Xu011);
 \draw (Xu012) node[anchor=south] {    $\!\!\!\!\! \begin{array}{c} \\ \text{  1-Rarefaction } \end{array}   $};
 \draw[blue,line width=0.2mm,dashed]  (Xd013) -- (Xu013);
 \draw[blue,line width=0.2mm,dashed]  (Xd014) -- (Xu014);
 \draw[blue,line width=0.2mm,dashed]  (Xd015) -- (Xu015);
 \draw[blue,line width=0.2mm,dashed]  (Xd016) -- (Xu016);
 \draw[blue,line width=0.2mm,dashed]  (Xd017) -- (Xu017);
 \draw[blue,line width=0.2mm,dashed]  (Xd018) -- (Xu018);
 \draw[blue,line width=0.2mm,dashed]  (Xd019) -- (Xu019);
 \draw[blue,line width=0.2mm,dashed]  (Xd020) -- (Xu020);
 \draw[blue,line width=0.2mm,dashed]  (Xd021) -- (Xu021);
 \draw[blue,line width=0.2mm,dashed]  (Xd022) -- (Xu022);
 \draw[blue,line width=0.2mm,dashed]  (Xd023) -- (Xu023);
 \draw[blue,line width=0.2mm,dashed]  (Xd024) -- (Xu024);
 \draw[red,line width=0.8mm, dashed]  (Xd025) -- (Xu025) node[near end,sloped,below] {\pmb{$\xi = u_{*\lf} $}};
 \draw (Xu027) node[anchor=south] {Vacuum};
 \draw[red,line width=0.8mm,dashed]  (Xd029) -- (Xu029) node[near end,sloped,above] {\pmb{$\xi = u_{*\rg} $}};
 \draw[blue,line width=0.2mm,dashed]  (Xd030) -- (Xu030);
 \draw[blue,line width=0.2mm,dashed]  (Xd031) -- (Xu031);
 \draw[blue,line width=0.2mm,dashed]  (Xd032) -- (Xu032);
 \draw[blue,line width=0.2mm,dashed]  (Xd033) -- (Xu033);
 \draw[blue,line width=0.2mm,dashed]  (Xd034) -- (Xu034);
 \draw[blue,line width=0.2mm,dashed]  (Xd035) -- (Xu035);
 \draw[blue,line width=0.2mm,dashed]  (Xd036) -- (Xu036);
 \draw[blue,line width=0.2mm,dashed]  (Xd037) -- (Xu037);
 \draw[blue,line width=0.2mm,dashed]  (Xd038) -- (Xu038);
 \draw[blue,line width=0.2mm,dashed]  (Xd039) -- (Xu039);
 \draw (Xu040) node[anchor=south] {    $\qquad \begin{array}{c}  \text{ 6-Rarefaction }  \end{array}   $};
 \draw[blue,line width=0.2mm,dashed]  (Xd041) -- (Xu041);
 \draw[blue,line width=0.2mm,dashed]  (Xd042) -- (Xu042);
 \draw[blue,line width=0.2mm,dashed]  (Xd043) -- (Xu043);
 \draw[blue,line width=0.2mm,dashed]  (Xd044) -- (Xu044);
 \draw[blue,line width=0.2mm,dashed]  (Xd045) -- (Xu045);
 \draw[blue,line width=0.2mm,dashed]  (Xd046) -- (Xu046);
 \draw[blue,line width=0.2mm,dashed]  (Xd047) -- (Xu047);
 \draw[blue,line width=0.2mm,dashed]  (Xd048) -- (Xu048);
 \draw[blue,line width=0.2mm,dashed]  (Xd049) -- (Xu049);
 \draw[blue,line width=0.2mm,dashed]  (Xd050) -- (Xu050);
 \draw[blue,line width=0.2mm,dashed]  (Xd051) -- (Xu051);
 \draw[blue,line width=0.2mm,dashed]  (Xd052) -- (Xu052);
 \draw[red,line width=0.8mm, dashed]  (Xd053) -- (Xu053) node[near end,sloped,below] {\pmb{$\xi = u_{\rg} + \coefa_{\rg}$}};
 \draw (Xw001) node[anchor=south east] {$\Large   \begin{array}{c} h_{\lf}\\[2mm] u_{\lf}\\[2mm] v_{\lf}\\[2mm] \Pres_{\lf}\\[2mm]   \R_{12}^{\lf} \\[2mm] \R_{22}^{\lf}  \end{array} $};
 \draw (Xw002) node[anchor=south west] {$\!\!\!\!\!\!\Large   \begin{array}{c} \\ h_{*\lf}=0\\[2mm] u_{*\lf}\\[2mm] v_{*\lf}\\[2mm] \Pres_{*}=0\\[2mm]   \R_{12}^{*\lf}=0 \\[2mm] \R_{22}^{*\lf}=0   \end{array} $\!};
 \draw (Xw005) node[anchor=south east] {$\Large   \begin{array}{c} h_{*\rg}=0\\[2mm] u_{*\rg}\\[2mm] v_{*\rg}\\[2mm] \Pres_{*}=0\\[2mm]   \R_{12}^{*\rg}=0 \\[2mm] \R_{22}^{*\rg}=0  \end{array} \!\!\!\!\!\!$};
 \draw (Xw006) node[anchor=south west] {$\Large   \begin{array}{c} h_{\rg}\\[2mm] u_{\rg}\\[2mm] v_{\rg}\\[2mm] \Pres_{\rg}\\[2mm]   \R_{12}^{\rg} \\[2mm] \R_{22}^{\rg}  \end{array} $};
   \end{tikzpicture}
   \caption{Shear Shallow Water (SSW) model: Wave structure of the
     1-D Riemann problem in presence of vacuum, when $u_{\rg} -
     u_{\lf} \geq \Afun(h_{\lf},c_{\lf}) +
     \Afun(h_{\rg},c_{\rg})$.
 Formally, without giving it a physical meaning because the depth and momentum are zero, we can define the velocities of intermediate states as :
      $u_{*\lf}= u_{\lf} + \Afun(h_\lf,c_\lf)$,
     $v_{*\lf} =v_\lf + \frac{2\p_{12}^\lf}{g h_\lf + 2 \p_{11}^\lf} \Afun(h_\lf,c_\lf)$,
     $u_{*\rg}= u_{\rg} - \Afun(h_\rg,c_\rg)$ and 
     $v_{*\rg}  =v_\rg - \frac{2\p_{12}^\rg}{g h_\rg + 2 \p_{11}^\rg} \Afun(h_\rg,c_\rg)$
   }
\label{SSW_Vacuum_Waves}
\end{center}
\end{figure}
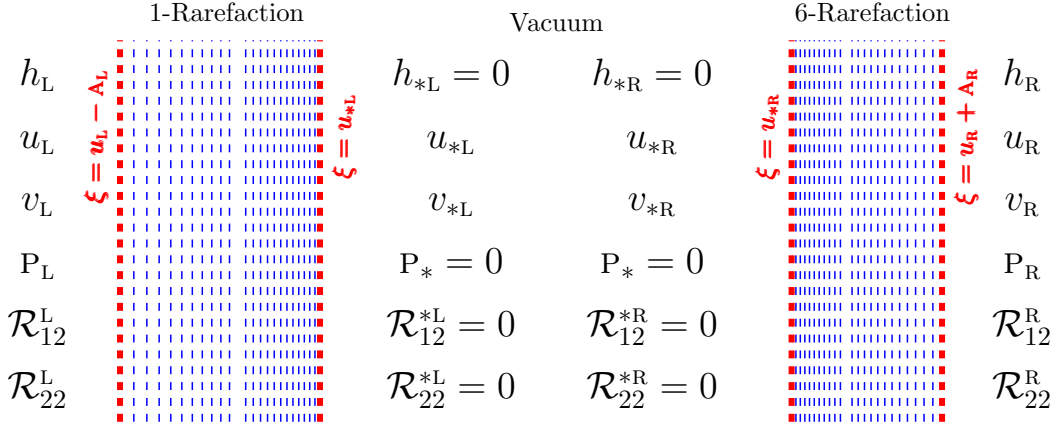

\begin{theorem}
(1)~If the left state is non-vacuum state and the right state is a vacuum state such that $u_\lf + \Afun(h_\lf, c_\lf) = u_\rg$ and $v_\rg = v_\lf + \frac{2\p_{12}^\lf}{g h_\lf + 2 \p_{11}^\lf} \Afun(h_\lf,c_\lf)$, then they can be connected by a 1-rarefaction wave.
(2)~If the left state is a vacuum state and the right state is a non-vacuum state such that $u_\lf = u_\rg - \Afun(h_\rg, c_\rg)$ and $v_{\lf} = v_\rg - \frac{2\p_{12}^\rg}{g h_\rg + 2 \p_{11}^\rg} \Afun(h_\rg,c_\rg)$, then they can be connected by a 6-rarefaction wave. 
(3)~If $u_{\rg} - u_{\lf} = \Afun(h_{\lf},c_{\lf}) + \Afun(h_{\rg},c_{\rg})$, then they can be connected with a 1-rarefaction and 6-rarefaction wave with an intermediate vacuum state and velocity $u_* = u_\lf + \Afun(h_\lf, c_\lf) = u_\rg - \Afun(h_\rg, c_\rg)$.  (4)~If $u_{\rg} - u_{\lf} > \Afun(h_{\lf},c_{\lf}) + \Afun(h_{\rg},c_{\rg})$, then they can be connected with a 1-rarefaction and 6-rarefaction wave with an intermediate vacuum state, see Figure~\ref{SSW_Vacuum_Waves}.
\end{theorem}
\section{Brief description of path conservative schemes}
We refer the reader to \cite{Pares2006} for a good general introduction to the concept of path conservative numerical schemes for non-conservative systems, and to~\cite{Chandrashekar2020} for a discussion specific to the present model.  The Riemann problem is the building block of a finite volume method and this approach can be used for non-conservative systems also~\cite{Gosse2001,Pares2006}. The main idea is to split the {\em fluctuation} into two parts corresponding to left moving and right moving waves arising in the Riemann solution, where the fluctuation is defined as
\[
\D(\conl, \conr) = \int_0^1 \A(\Psi(\xi; \conl, \conr))  \dd{\Psi}{\xi}(\xi; \conl, \conr) \ud \xi = \D^-(\conl,\conr) + \D^+(\conl, \conr)
\]
The splitting of the fluctuation can be performed using a Roe-type Riemann solver or HLL-type Riemann solver, the latter being the approach taken in the present work and following~\cite{Chandrashekar2020}. HLL-type methods model the Riemann solution by simple waves and require estimation of the smallest and largest wave speed arising in the Riemann problem. Assume that there are $m$ simple waves in the approximate Riemann solution with $m-1$ intermediate states. Let us denote the wave speeds as $S_j$, $j=1,\ldots,m$ and the intermediate states as $\con_j^*$, $j=1,\ldots,m-1$ with $\con_0^* = \conl$ and $\con_m^* = \conr$. The fluctuation splitting is given by
\[
\D^\pm(\conl,\conr) = \sum_{j=1}^m S_j^\pm (\con^*_{j+1} - \con^*_j)
\]
where
\[
S^- = \min(0, S), \qquad S^+ = \max(0,S)
\]
The intermediate states are obtained by satisfying the Rankine-Hugoniot conditions across all the waves. The approximate Riemann solvers of different complexity based on the number of waves including in the model can be derived. In~\cite{Chandrashekar2020}, two wave HLL solver, three wave HLLC3 solver and five wave HLLC5 approximate Riemann solvers have been constructed by using the generalized jump conditions. The HLL solver contains only the slowest and fastest waves in its model; the HLLC3 solver also includes the contact wave while the HLLC5 solver includes all five waves.

Let us consider a partition of the domain into disjoint cells of size $\dx$.  Let $\con_j^n$ denote the approximation of the cell average value in the $j$'th cell at time $t=t_n$.  The first order scheme is given by
\[
\con_j^{n+1} = \con_j^n - \frac{\dt}{\dx} ( \D^{+,n}_\jmh + \D^{-,n}_\jph) + \dt \s(\con_j^{n+\theta}), \qquad
\D_\jph^{\pm,n} = \D^{\pm}(\con_j^n, \con_{j+1}^n)
\]
For $\theta=0$ we obtain an explicit scheme and for $\theta=1$ we obtain a semi-implicit scheme; however the coupling in the semi-implicit scheme is only local to the cell. An exact solution process for the semi-implicit scheme is explained in the Appendix of~\cite{Chandrashekar2020}. If the system is conservative, i.e., $\A = \f'(\con)$ for some $\f$, then the above scheme can be written in conservation form with some numerical flux function~\cite{Pares2006}. Such a scheme can be made higher order accurate using a MUSCL-Hancock approach as in~\cite{Chandrashekar2020} or using a method of lines approach combined with a high order Runge-Kutta scheme. The numerical computations used in this work are based on a MUSCL-Hancock approach as explained in~\cite{Chandrashekar2020}.
\subsection{Estimation of wave speeds}
The approximate Riemann solver requires an estimate of the slowest and fastest wave speeds which should enclose the exact wave speeds in order for the entropy condition to be satisfied. One commonly used method to estimate the wave speeds in the Riemann problem uses a combination of the left and right states and the Roe average state~\cite{Einfeldt1988}; following this idea we can use the following speed estimates
\[
S_{\lf}^{HLL} = \min\{ \lambda_1(\prim_{\lf}), \lambda_1(\bar{\prim})\}, \qquad S_{\rg}^{HLL} = \max\{ \lambda_6(\prim_{\rg}), \lambda_6(\bar{\prim})\}, \qquad \bar{\prim} = \half (\prim_{\lf} + \prim_{\rg})
\]
where $\prim$ represents the variables $(h,\vel,\R)$ and we use the arithmetic average instead of the Roe average. If $S_{\lf}^{ex}$, $S_{\rg}^{ex}$ denote the exact wave speeds, then we require that $S_{\lf}^{HLL} \le S_{\lf}^{ex}$ and $S_{\rg}^{HLL} \ge S_{\rg}^{ex}$, but this is not guaranteed to hold with the above estimates. As an example, consider the dam break problem from Section~\ref{sec:dambreak} for which the slowest and fastest speeds are
\[
S_{\lf}^{ex} = -0.44328320518603004, \qquad S_{\rg}^{ex} = 0.43554139386439333
\]
whereas the speed estimate obtained from the above formulae are
\[
S_{\lf}^{HLL} = -0.44328320518603004, \qquad S_{\rg}^{HLL} =  0.38399218742052554
\]
We see that fastest speed $S_{\rg}$ is very much under estimated and this may cause numerical problems like loss of positivity and violation of entropy condition. How to obtain better estimates of the slowest and fastest speeds without using the exact Riemann solution is an open question. In the present work we use a simple way to over-estimate the speeds by using both the states to estimate the speeds as follows
\begin{equation} \label{eq:speedapp}
S_{\lf}^{HLL*} = \min\{ \lambda_1(\prim_{\lf}), \lambda_1(\prim_{\rg}), \lambda_1(\bar{\prim})\}, \qquad S_{\rg}^{HLL*} = \max\{ \lambda_6(\prim_{\lf}), \lambda_6(\prim_{\rg}), \lambda_6(\bar{\prim})\}
\end{equation}
For the dam break problem, this yields
\[
S_{\lf}^{HLL*} = -0.44328320518603004, \qquad S_{\rg}^{HLL*} = 0.44328320518603004
\]
Now the fastest speeds is also estimated in such a way that the numerical Riemann fan bounds the exact Riemann fan. We use the above estimate in all the approximate Riemann solvers used in this study.
\section{Exact solutions compared  with approximate Riemann solvers}
In the next few sections, we compare the exact solutions with numerical solutions obtained with approximate Riemann solvers using a second order accurate MUSCL-Hancock scheme~\cite{Chandrashekar2020}. Unless stated otherwise, we use the speed estimates given by~\eqref{eq:speedapp} in all the test cases. We show results obtained from second order numerical scheme in most of the test cases since we do not observe any qualitative difference between first and second order results, but in some test cases, where significant differences are found, we show first order results also. In all the tests, the bottom topography is constant and the source term $\s$ is absent, since we want to study the purely hyperbolic problem.
\subsection{Dam break problem} \label{sec:dambreak}
 We consider here the test case used in \cite{Bhole2019,Gavrilyuk2018,Chandrashekar2020}. It is a Riemann problem where, initially, the velocity is zero every where, the stress tensor is constant and only the initial depth has a jump,
 \[
 h = \left \{ \begin{array}{rl}
 0.02, & x < 0.5 \\[3mm]
 0.01, & x > 0.5
  \end{array} \right ., \quad
  u = 0, \quad v=0,\quad
 \p_{11} =  10^{-4}, \quad \p_{12} =  0, \quad \p_{22} =  10^{-4}.
  \]
For this Riemann data, we can compute the associated analytical solution. Numerical approximations are performed with HLL and HLLC (3-waves and 5-waves) Riemann solvers (see \cite{Chandrashekar2020} for details). Figure~\ref{fig:DamB} shows that the exact and the approximate solutions are almost comparable, except for the shock front. The HLL and HLLC Riemann solvers are converging to the same limit.  However, in accordance with~\cite{Chandrashekar2020}, the numerical limit does not match with the exact solution as seen in Figure~\ref{fig:DamBF}.  This is probably related to the fact that $\p_{11}$ is too small; initially we have $\p_{11} =10^{-4}$ and $\coefb = \sqrt{\p_{11}} =10^{-2}$ . Indeed, as $\coefb$ goes to zero, the shear and the contact waves approach one another and they coincide in the limit of $\p_{11}=0$. The Riemann solvers used here are not designed to get the proper behaviour at this asymptotic case. The approximate Riemann solvers used here are not designed to strictly conserve the total energy~\eqref{eq:totE}. As shown in Theorem~\ref{lem:totE}, the jump condition of total energy equation is automatically satisfied by the jump conditions of the SSW model. The approximate Riemann solver is based on satisfying these jump conditions and we can expect approximate conservation in the numerical scheme also. To examine the conservation of total energy in the domain, we plot it as a function of time in Figure~\ref{fig:totEdb}, where the ratio of total energy at time $t$ to that at initial time is shown. We see that it is not strictly conserved by the numerical scheme but there is a dissipation of this energy, with the error at the final time being about 0.15\% on the coarse mesh. At the PDE level, the total energy is conserved for inviscid problems ($\Diss=0$).  At the discrete level, this property is satisfied if we solve the conservative form of the total energy, which is not the case here. Nevertheless, it is possible to strengthen this conservation law, either by using an augmented system~\cite{Gavrilyuk2018} or by redistributing the energy conservation defect on the pressure tensor as done in~\cite{Busto2021}.

A modified test case has been designed in order to keep $\p_{11}$ away from zero. The Riemann data is given by
\[
h = \left \{ \begin{array}{rl}
 0.02, & x < 0.5 \\[3mm]
 0.01, & x > 0.5
  \end{array} \right ., \quad
  u = 0, \quad v=0,\quad
  \p_{11} =  4 \times 10^{-2}, \quad \p_{12} =  0, \quad \p_{22} =  4 \times 10^{-2}.
\]
In this modified context the numerical solution does not contain any more defect in the shock front propagation with respect to the exact solution, even at low numerical resolution as shown in Figure~\ref{fig:DamBMM}. The convergence to the analytical solution is also observed in Figures~\ref{fig:DamBMMF} and \ref{fig:DamBMM}.  Thus it seems that we are facing here a lack of asymptotic preserving property of the numerical schemes when $\p_{11}$ goes to zero. At this asymptotic, the shock front seems to be not accurately resolved with the current schemes, when compared with the designed exact solution. Note that disagreement only occurs at the shock front and elsewhere the numerical approximations converge to the analytical solution. The convergence is observed even at the shock front when the value of $\p_{11}$ is not too small. Similar convergence is also observed if we use $\p_{11} = 4 \times 10^{-2}$, $\p_{12} = 0$ and $\p_{22} = 10^{-8}$.

We also test another  variant of the modified dam break problem, where  $\p_{12}$ is set to a small non zero value,
\[
h = \left \{ \begin{array}{rl}
 0.02, & x < 0.5 \\[3mm]
 0.01, & x > 0.5
  \end{array} \right ., \quad
  u = 0, \quad v=0,\quad
  \p_{11} =  4 \times 10^{-2}, \quad \p_{12} =  10^{-8}, \quad \p_{22} =  4\times 10^{-2}.
\]
The numerical approximation, even on a coarse mesh as shown in Figure~\ref{fig:DamBM}, fit very well with the designed exact solution. With this modification, the profile of $\p_{12}$ shows all the five waves of the SSW system. As expected, the intermediate waves are better resolved by the HLLC schemes.

As shown in Figure~\ref{fig:CDamBMM}, the mesh convergence is observed for the three Riemann solvers used, both for the dam break and for the modified dam break problems. Indeed, the different numerical approaches converge asymptotically to the same numerical solution, as the mesh becomes more and more refined. Nevertheless, for the initial dam break problem, the numerical solutions converge to a different solution than the one obtained analytically as seen in the left figure; with the HLL solution being slightly different from the HLLC solvers. On the other hand, for the modified dam break problem shown on the right, where the determinant of $\p$ is not as close to zero, the numerical solutions overlap closely with the analytical solution. However, there is still a small difference in the $\p_{11}$ values around the shock as shown in the inset figure. The convergence of the $L^1$ errors with respect to the exact solution are shown in Figure~\ref{fig:convdb} where we see that both test cases converge to a solution different from the exact solution. The modified dam break case converges to smaller errors but eventually the convergence stalls, which is expected since we have already observed this in Figure~\ref{fig:CDamBMM}.
\begin{figure}
\begin{center}
\includegraphics[width=0.49\textwidth]{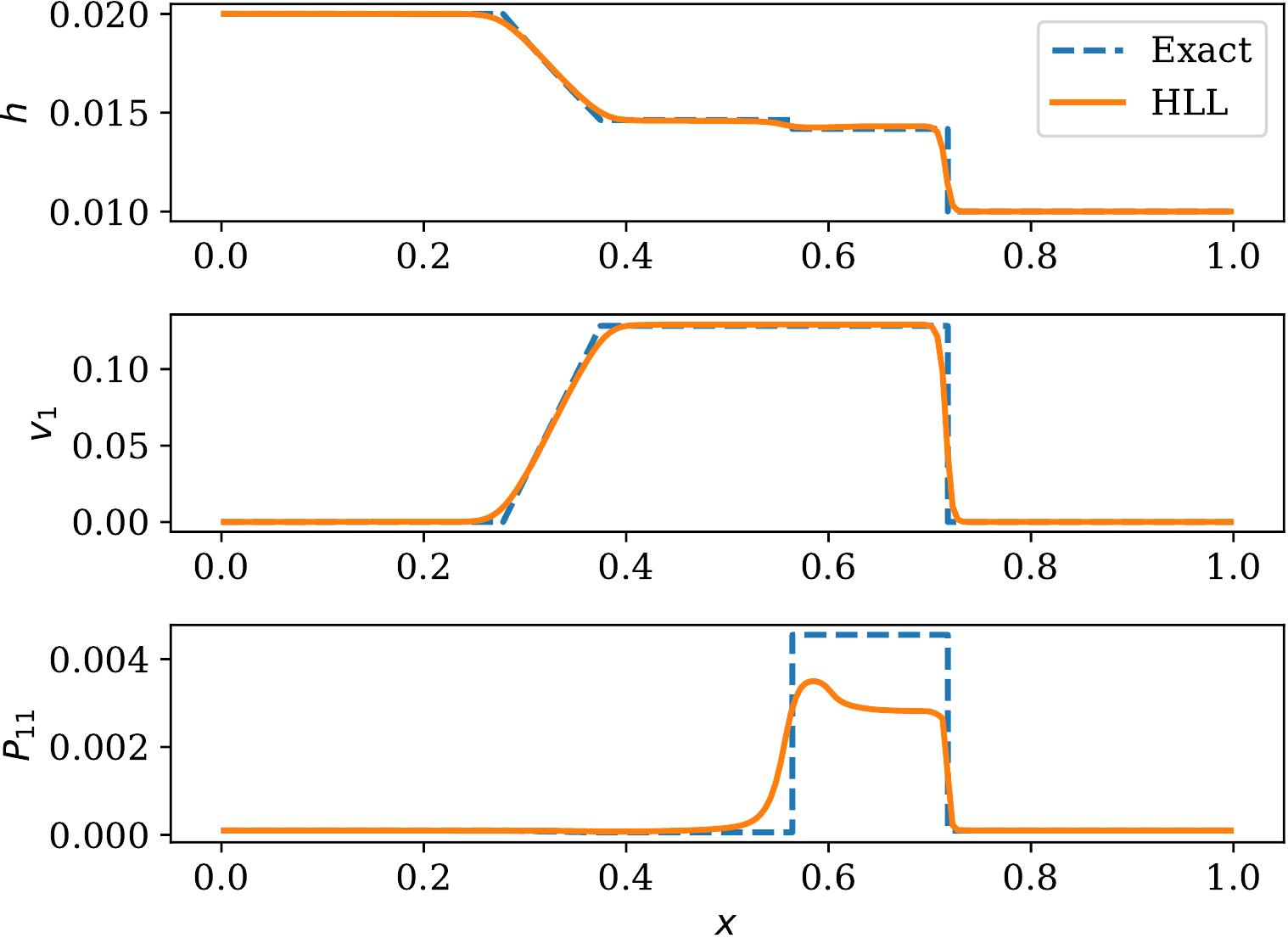}
\includegraphics[width=0.49\textwidth]{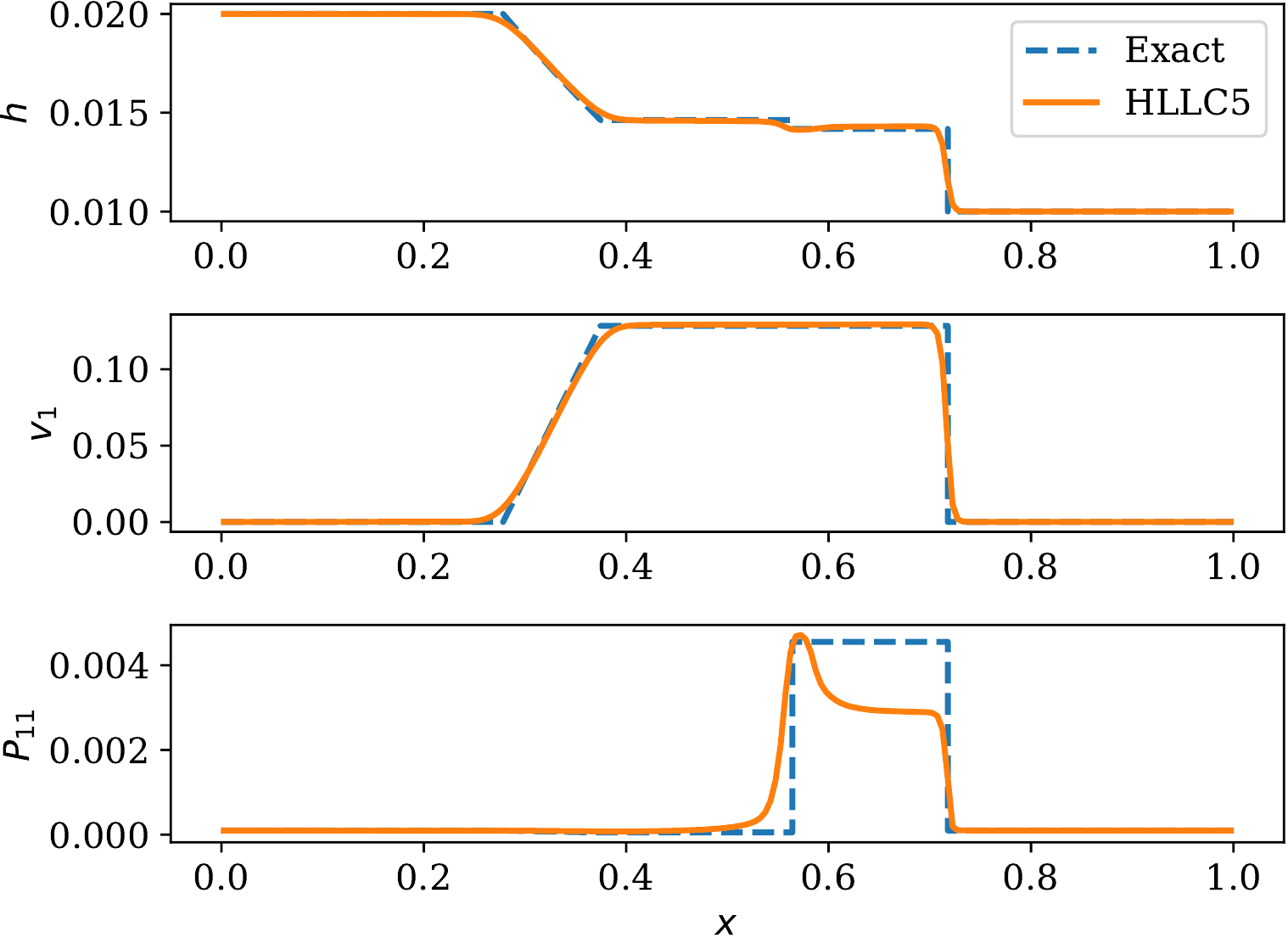}
\caption{Dam break test case with  200 cells and second order approximations. Comparison between exact and numerical solutions obtained with HLL (left) and HLLC5 (right) schemes.}
\label{fig:DamB}
\end{center}
\end{figure}
 \begin{figure}
\begin{center}
\includegraphics[width=0.49\textwidth]{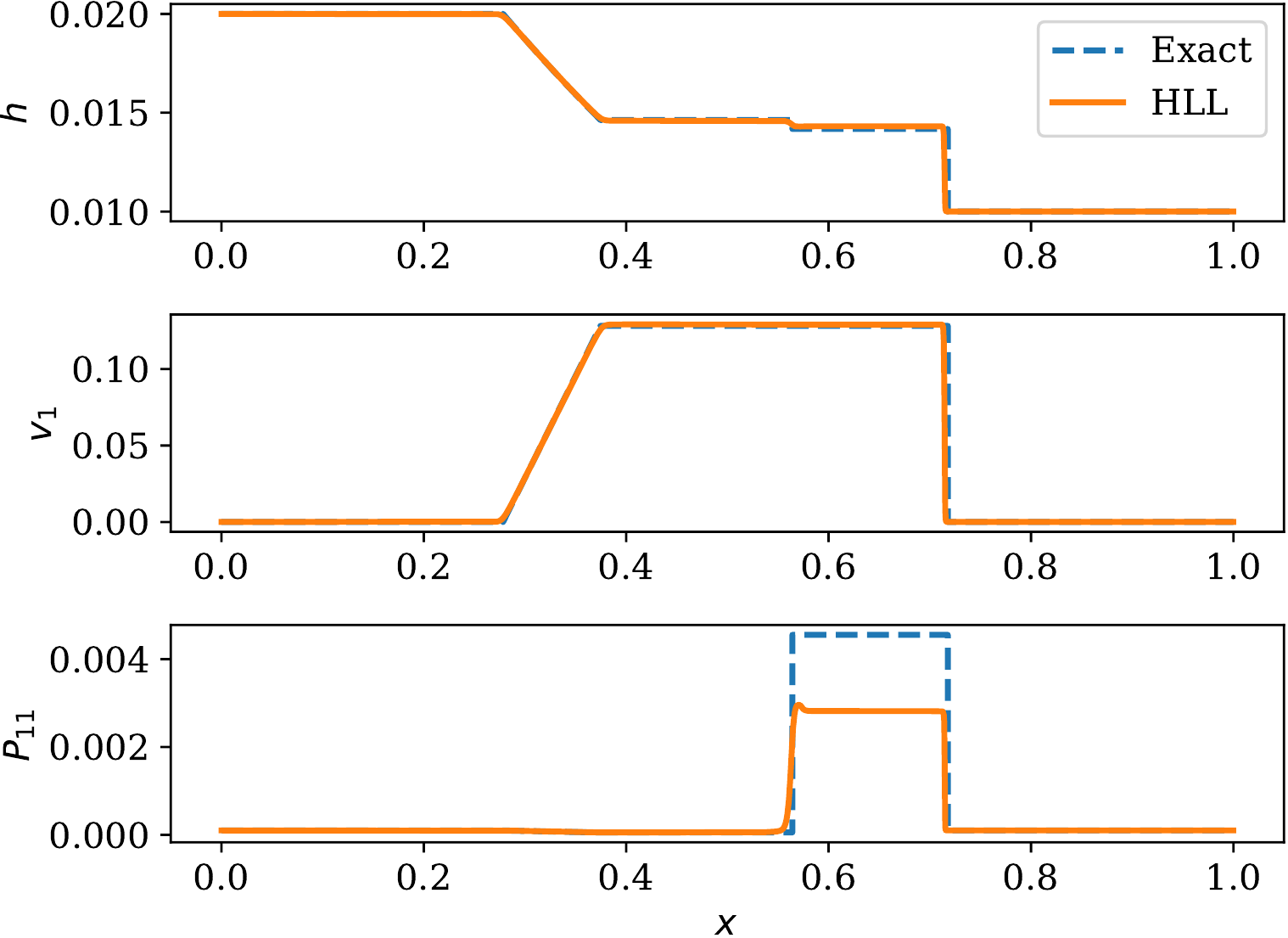}
\includegraphics[width=0.49\textwidth]{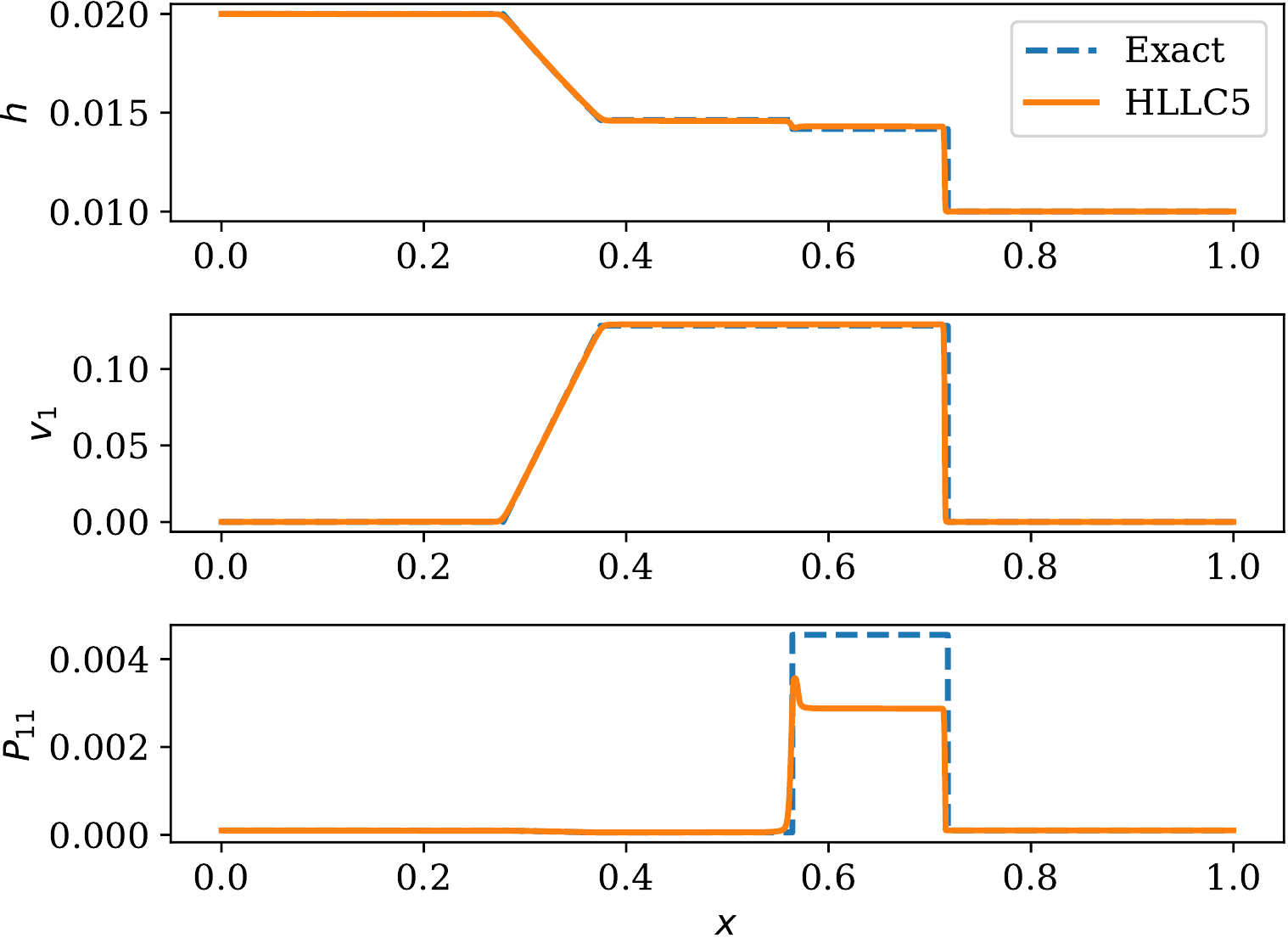}
\caption{Dam break test case with  2000 cells and second order approximations. Comparison between exact and numerical solutions obtained with HLL (left) and HLLC5 (right) schemes.}
\label{fig:DamBF}
\end{center}
\end{figure}
\begin{figure}
\begin{center}
\includegraphics[width=0.49\textwidth]{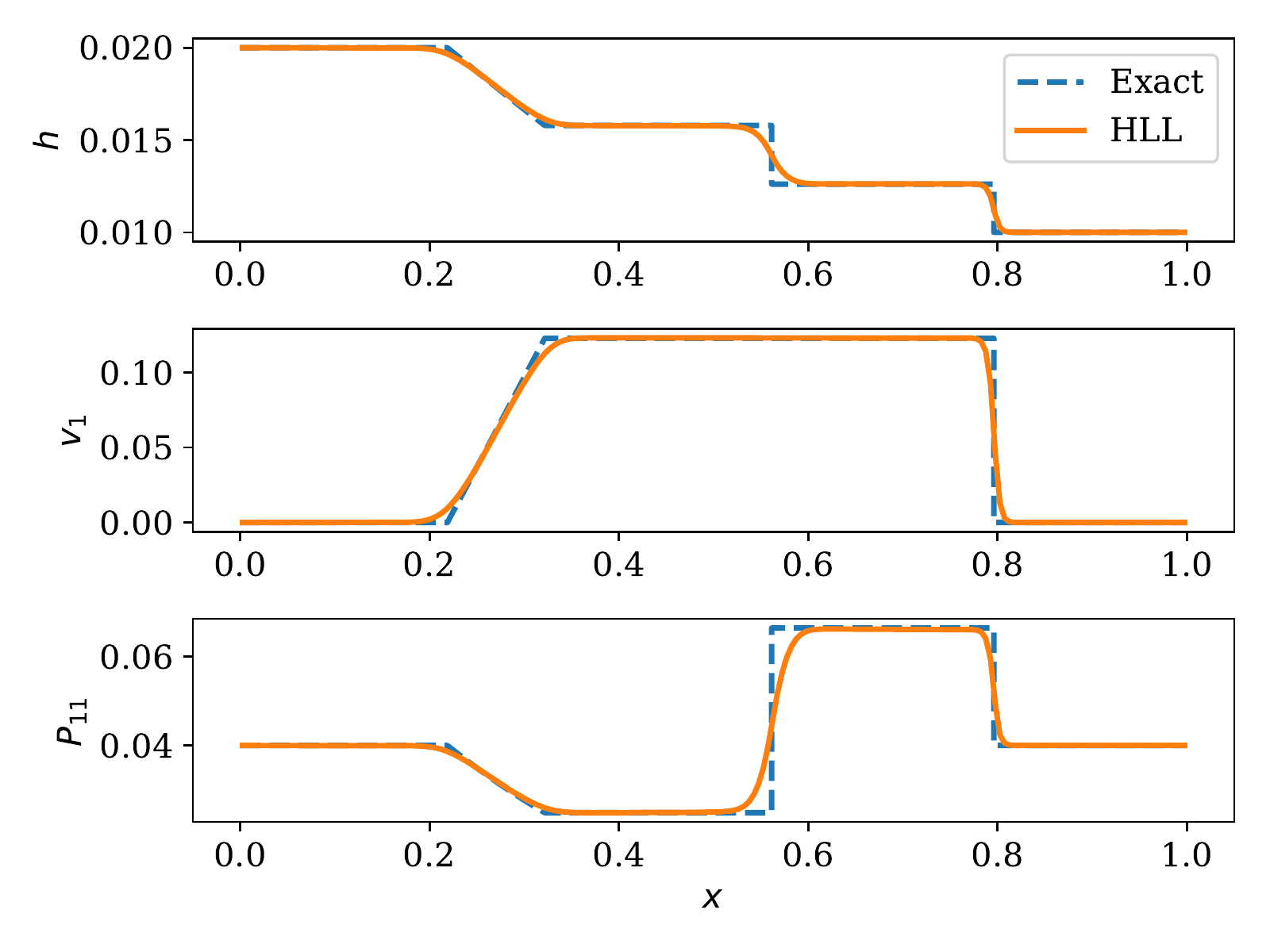}
\includegraphics[width=0.49\textwidth]{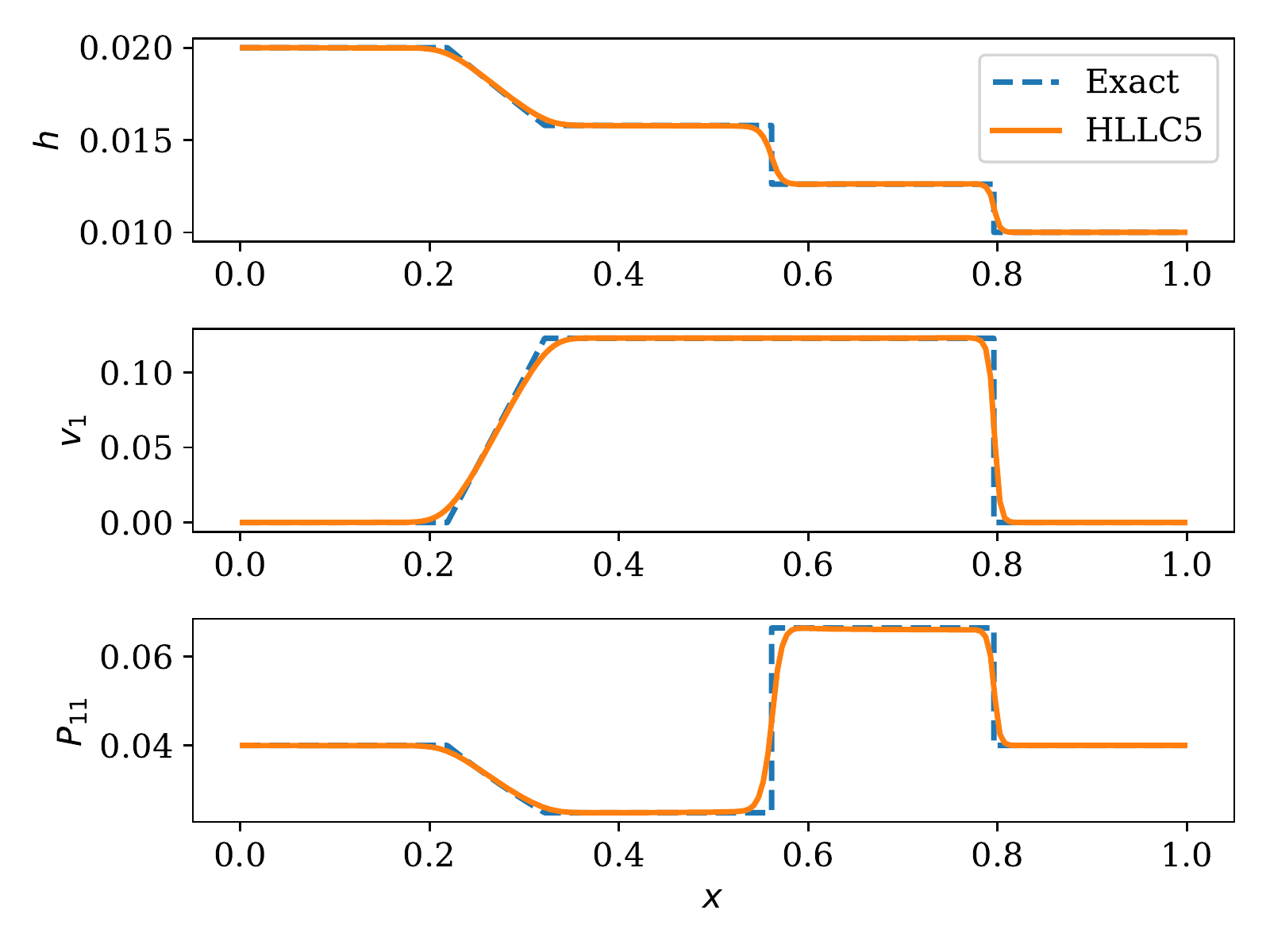}
\caption{Modified dam break  test case with 200 cells and second order approximations. Comparison between exact and numerical solutions obtained with HLL (left) and HLLC5 (right) schemes.}
\label{fig:DamBMM}
\end{center}
\end{figure}
\begin{figure}
\centering
\includegraphics[width=0.5\textwidth]{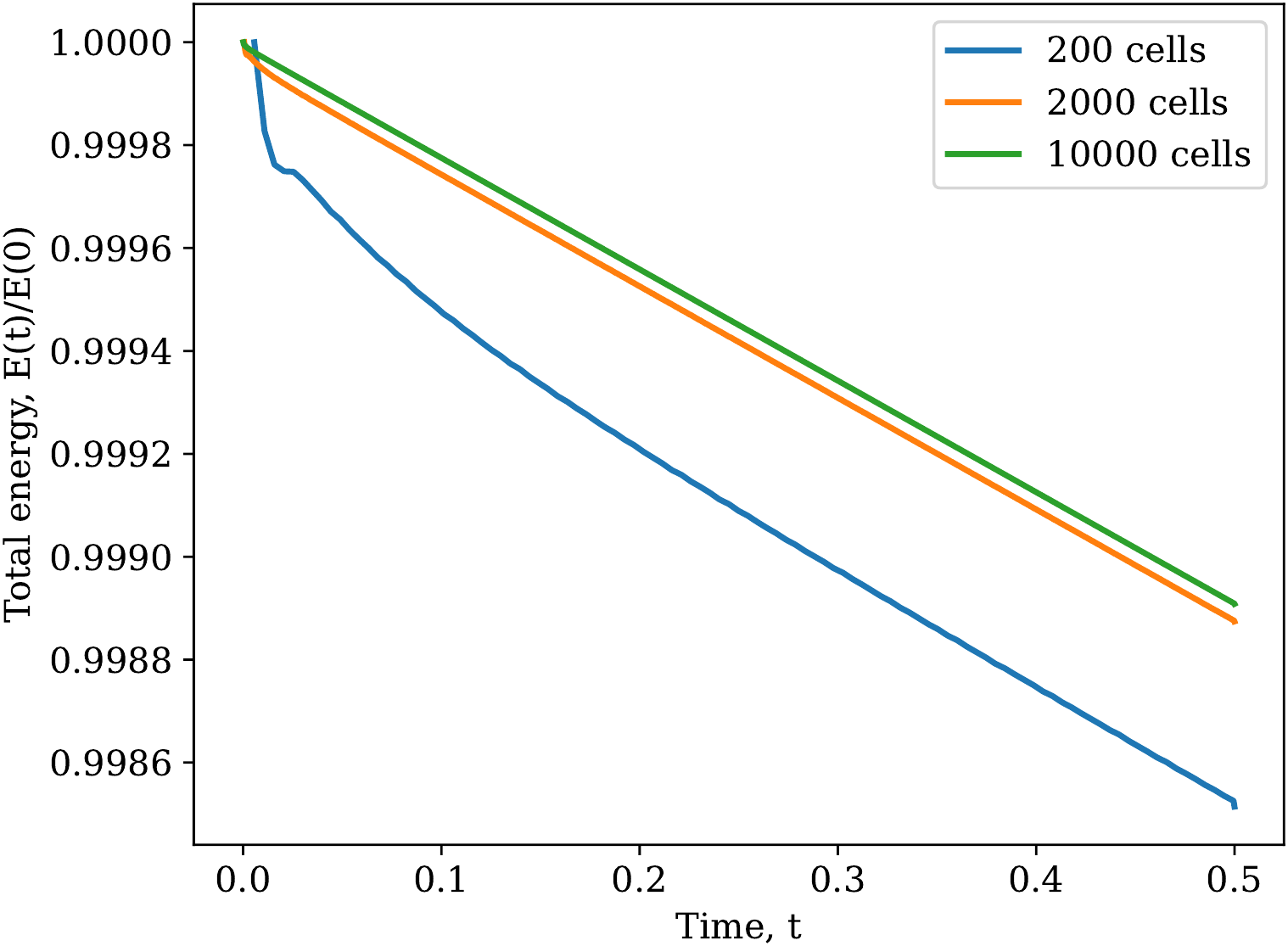}
\caption{Total energy in the domain as a function of time for dam break problem using HLLC5 scheme.}
\label{fig:totEdb}
\end{figure}
 \begin{figure}
\begin{center}
\includegraphics[width=0.49\textwidth]{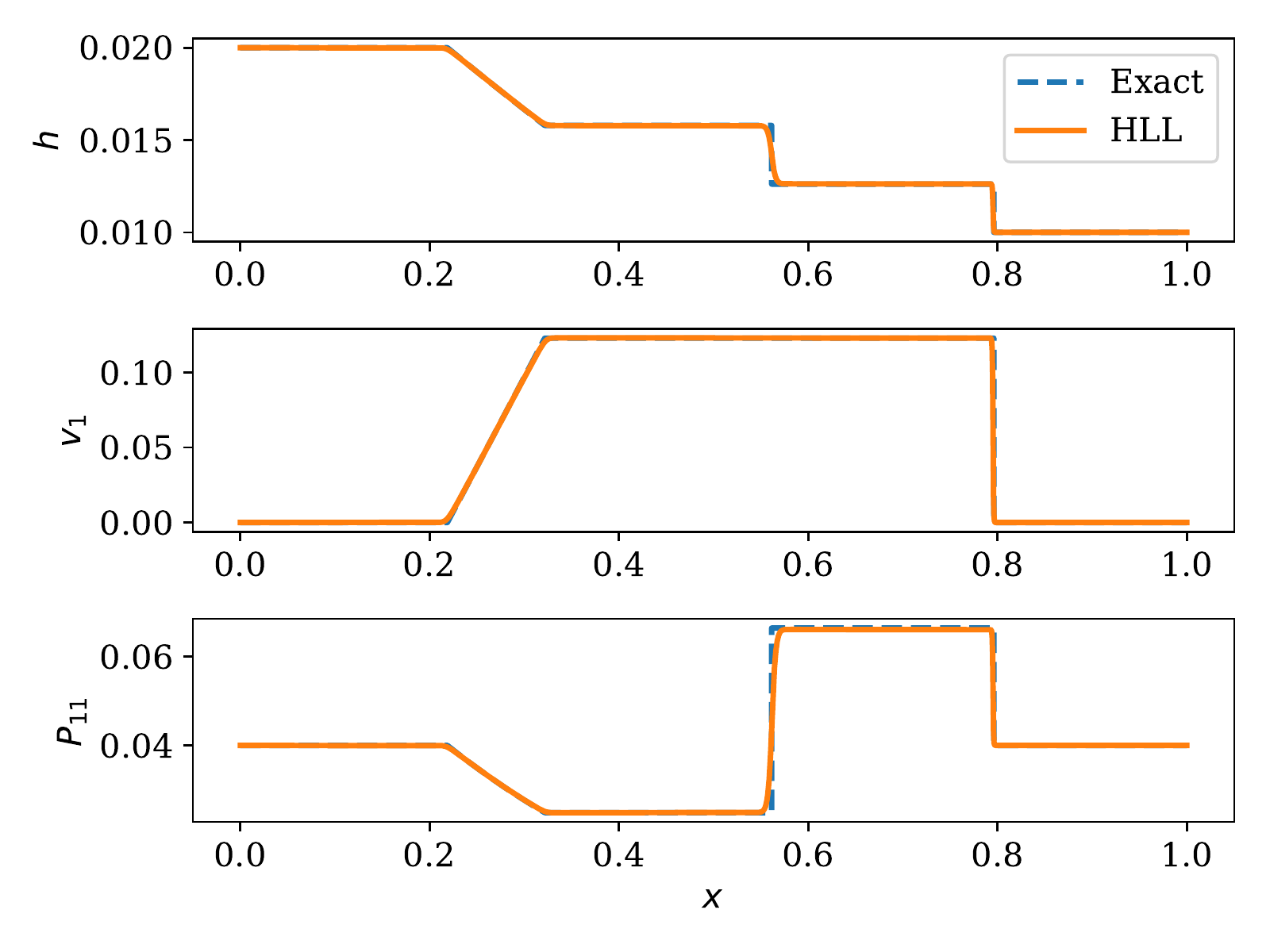}
\includegraphics[width=0.49\textwidth]{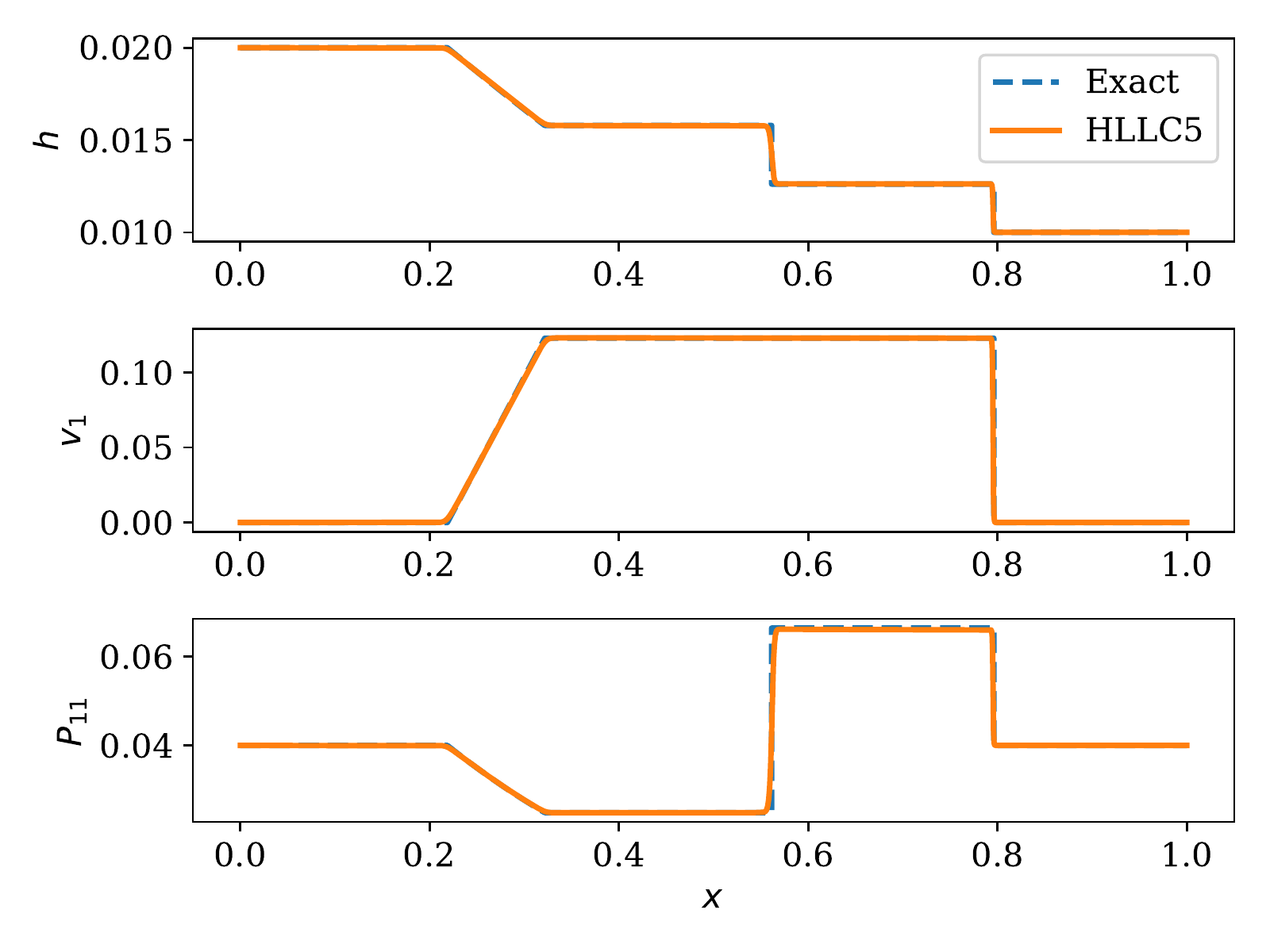}
\caption{Modified dam break test case with  2000 cells and second order approximations. Comparison between exact and numerical solutions obtained with HLL (left) and HLLC5 (right) schemes.}
\label{fig:DamBMMF}
\end{center}
\end{figure}

\begin{figure}
\begin{center}
\includegraphics[width=0.49\textwidth]{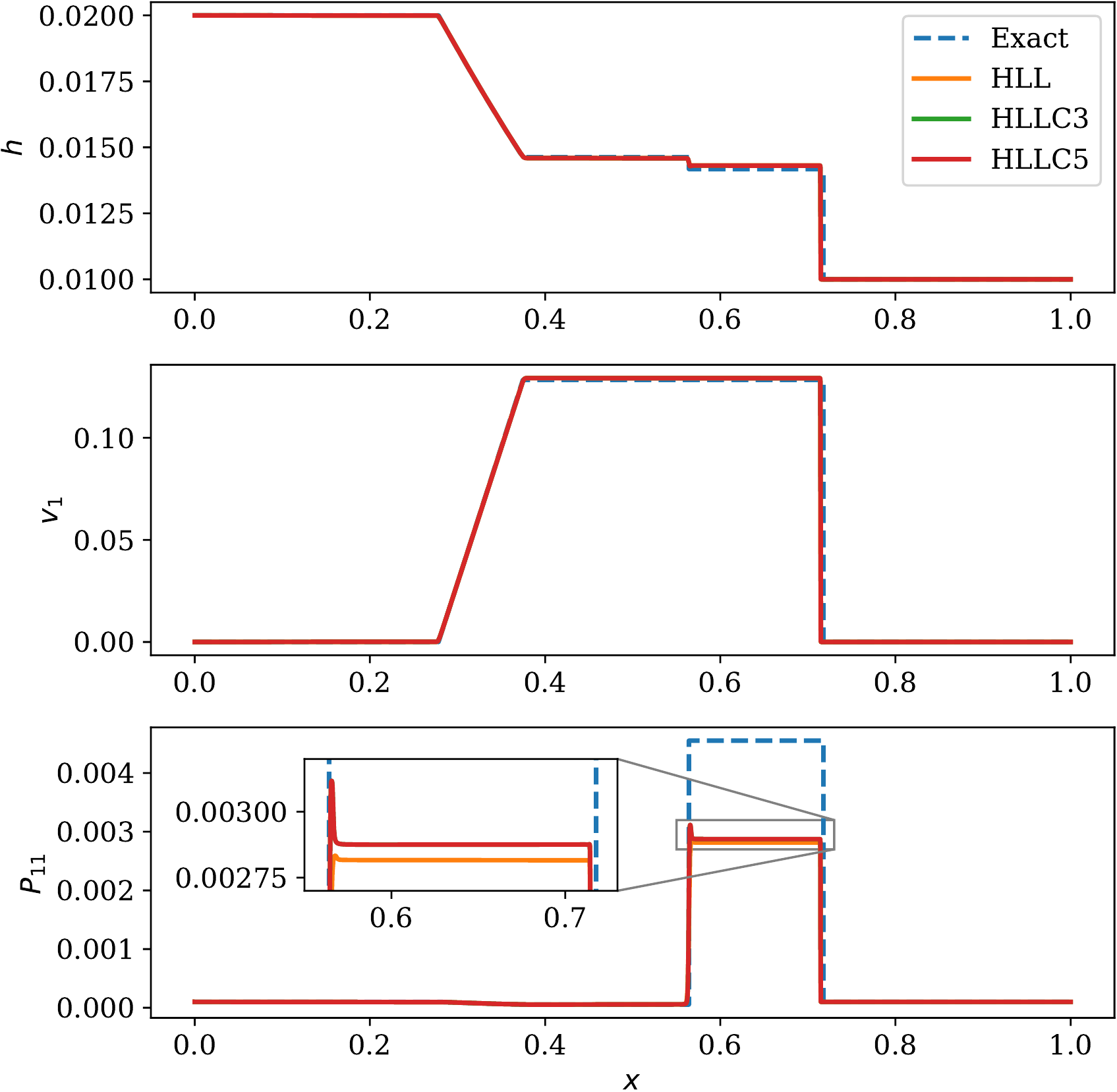}
\includegraphics[width=0.49\textwidth]{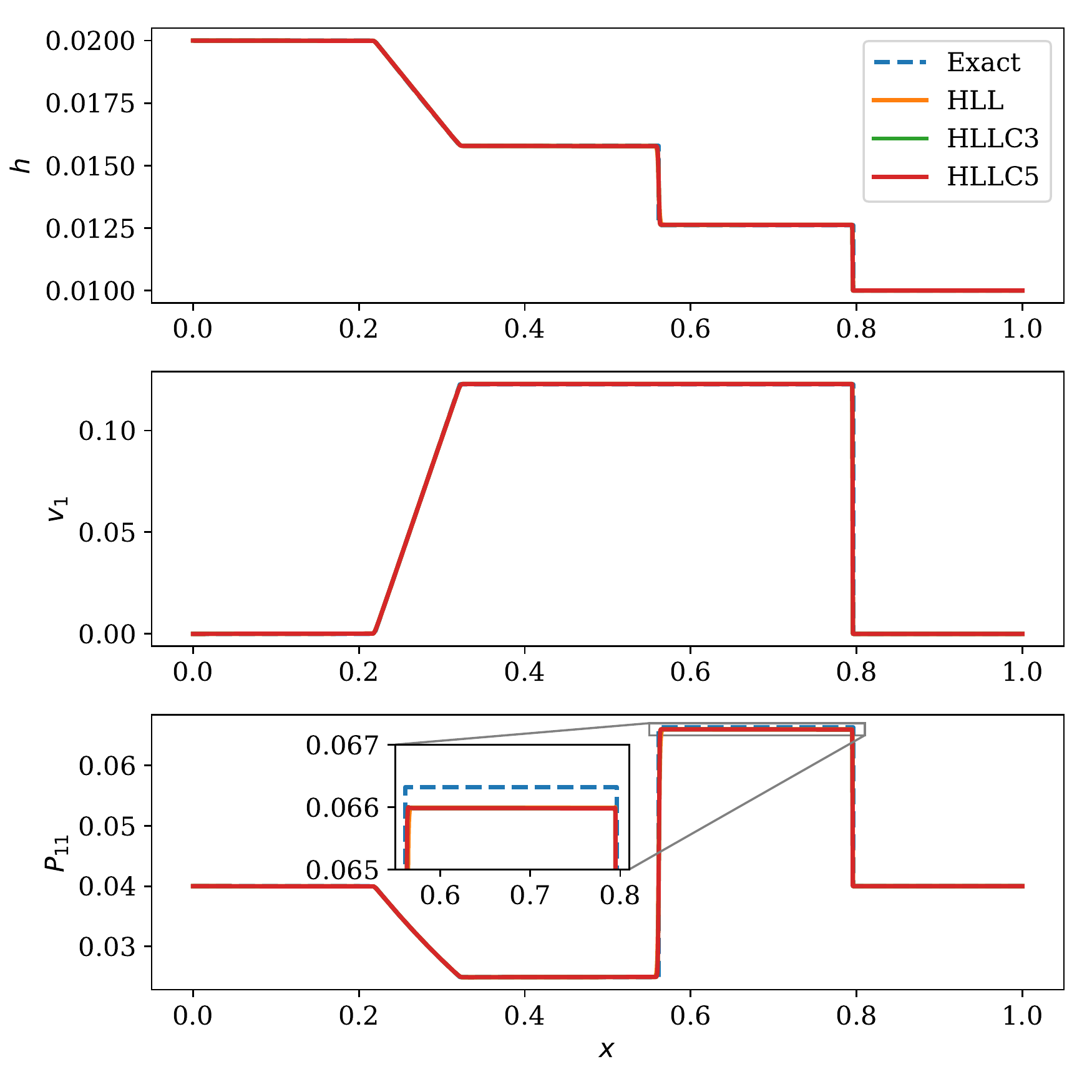}
\caption{Converged solutions for dam break (left) and  modified dam break (right) problems. Numerical solutions are shown with 10000 cells.}
\label{fig:CDamBMM}
\end{center}
\end{figure}

\begin{figure}
\centering
\includegraphics[width=0.9\textwidth]{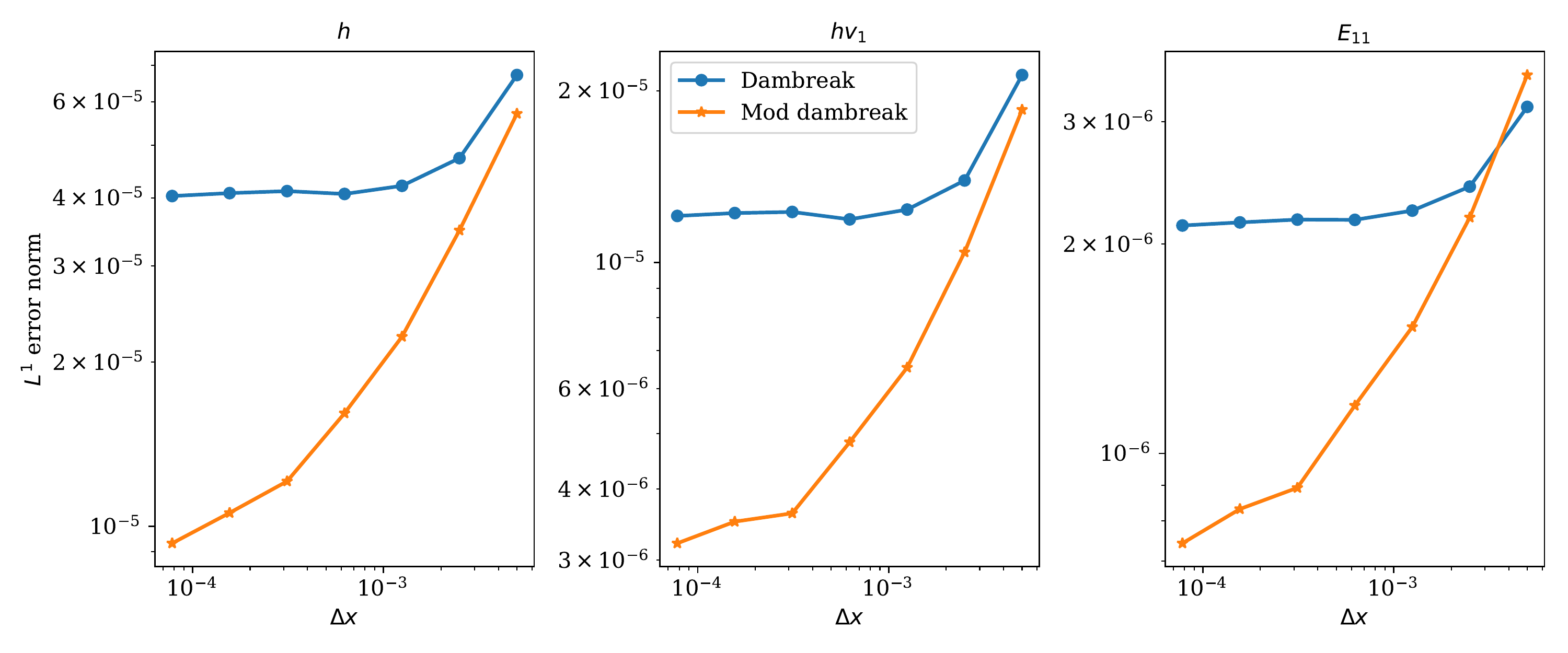}
\caption{Convergence of $L^1$ error norm for dambreak problems using HLLC5 scheme}
\label{fig:convdb}
\end{figure}
\begin{figure}
\begin{center}
\includegraphics[width=0.49\textwidth]{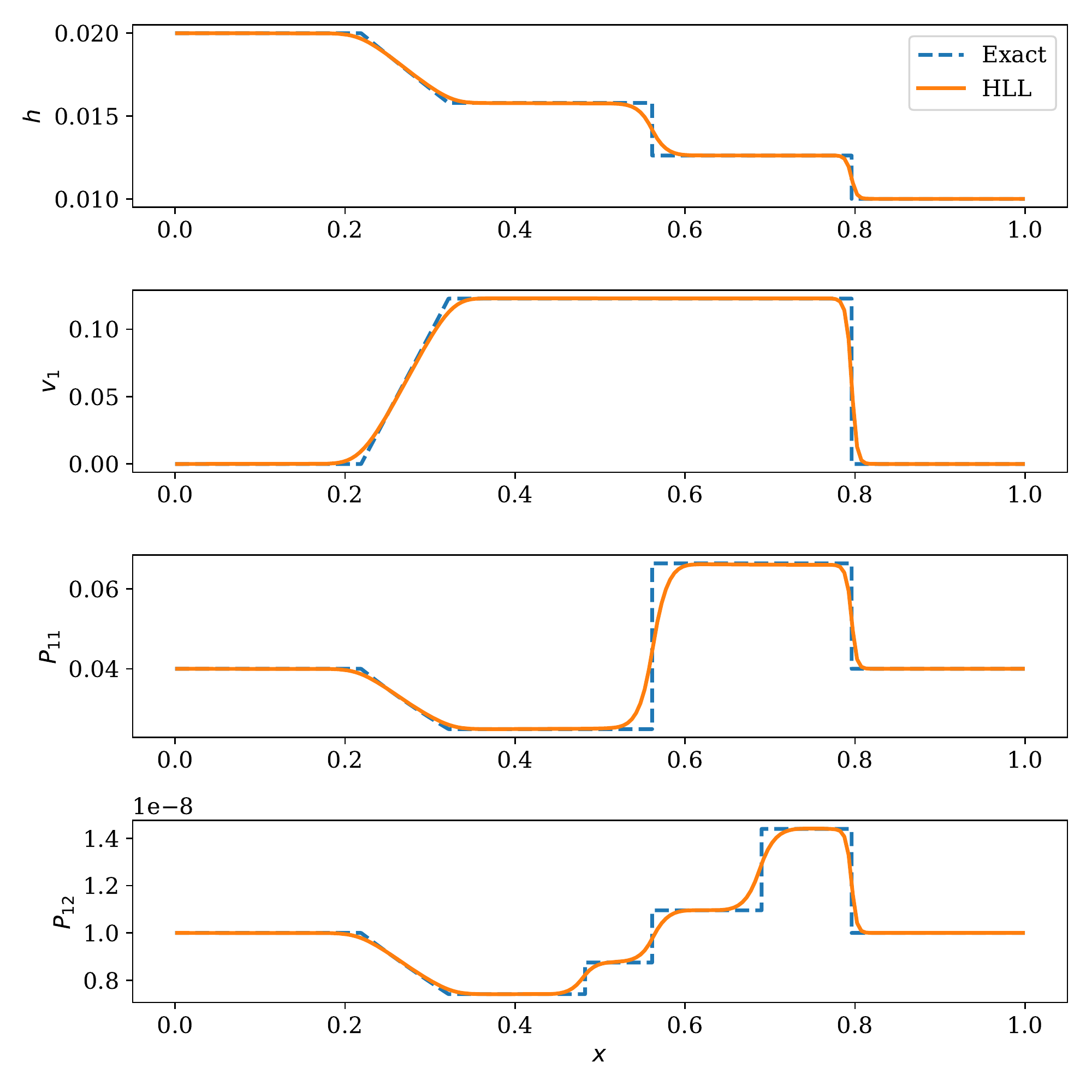}
\includegraphics[width=0.49\textwidth]{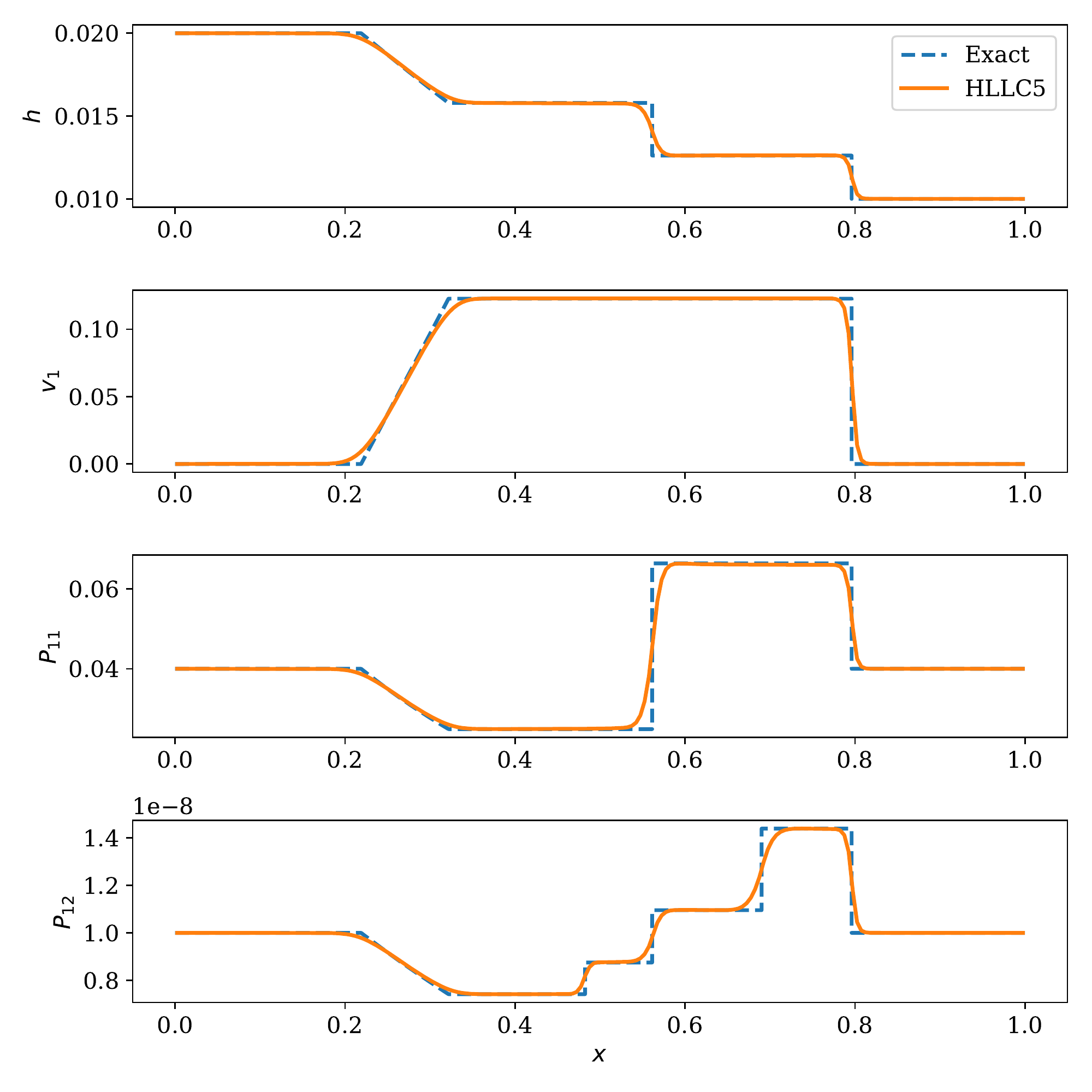}
\caption{Modified dam break with initially $\p_{12} = 10^{-8}$ in the
  entire domain. Mesh of 200 cells and second order
  approximations. Comparison between exact and numerical solutions
  obtained with HLL (left) and HLLC5 (right) schemes. }
\label{fig:DamBM}
\end{center}
\end{figure}
\subsection{Five waves dam break problem}
The initial condition for the Riemann problem is given in the following table,
\begin{center}
\begin{tabular}{|c|c|c|c|c|c|c|}
\hline
 & $h$ & $v_1$ & $v_2$ & $\p_{11}$ & $\p_{12}$ & $\p_{22}$ \\
\hline
$x < 0.5$ & 0.01 & 0.1 & 0.2 & $4 \times 10^{-2}$ & $10^{-8}$ & $4 \times 10^{-2}$ \\
\hline
$x > 0.5$ & 0.02 & 0.1 & -0.2 &  $4 \times 10^{-2}$ & $10^{-8}$ & $4 \times 10^{-2}$ \\
\hline
\end{tabular}
\end{center}
The initial data is like a dam break problem but with some initial shear $v$ and a non-zero normal velocity $u$. The results are shown in Figure~\ref{fig:Five} and (\ref{fig:FiveF}) at time $t=0.5$ units. The solution shows five waves including 1-shock and 6-rarefaction wave. All the waves are captured by both Riemann solvers even on the coarse mesh of 200 cells. The numerical solution and the location of the waves agrees well with the exact solution, and the numerical results approach the exact solution on the finer mesh as seen in Figure~\ref{fig:FiveF}. The values of $\p$ used are larger as in the case of the modified dam break problem and this leads to good agreement between the numerical and exact solutions, which was observed in the previous dam break problem.
\begin{figure}
\begin{center}
\includegraphics[width=0.49\textwidth]{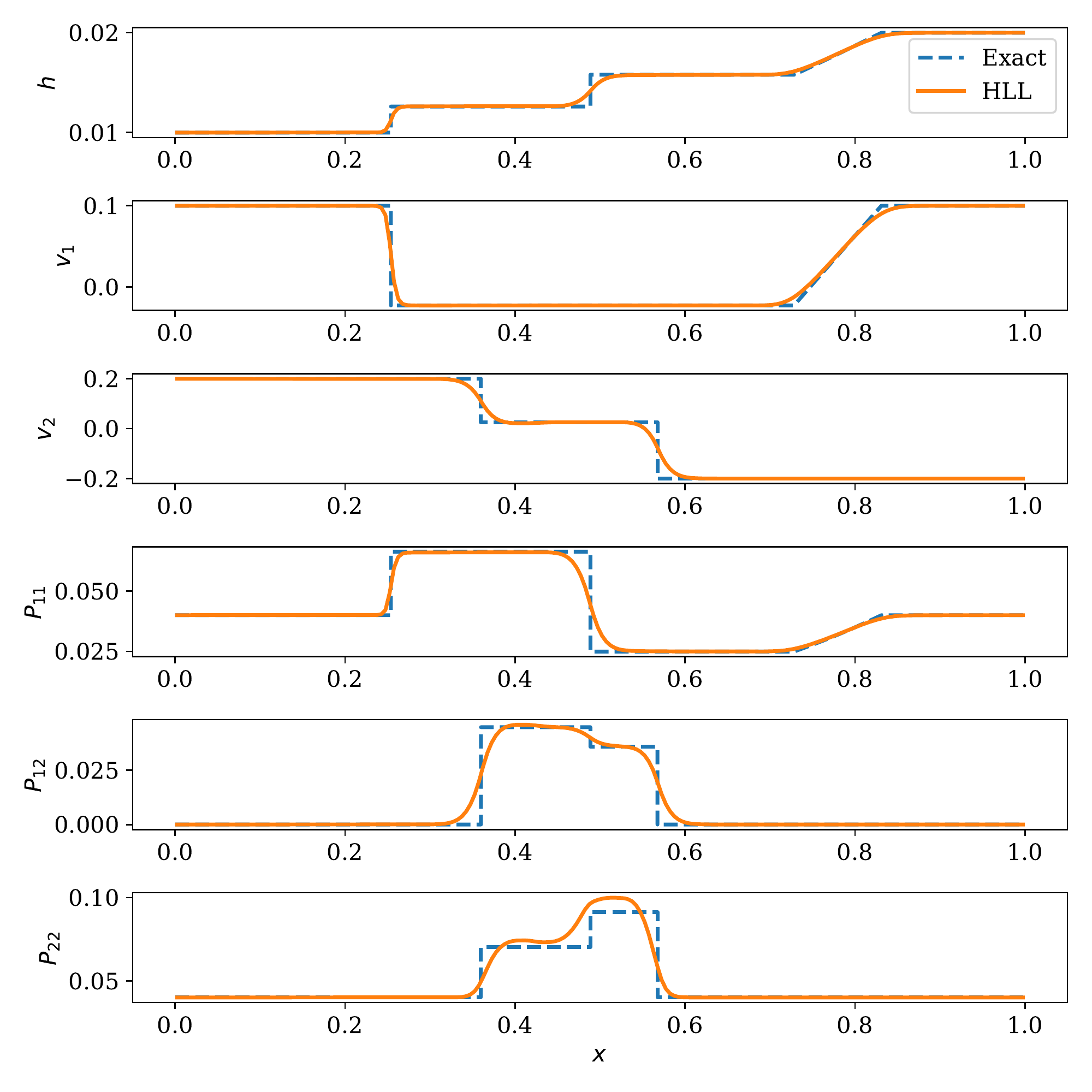}
\includegraphics[width=0.49\textwidth]{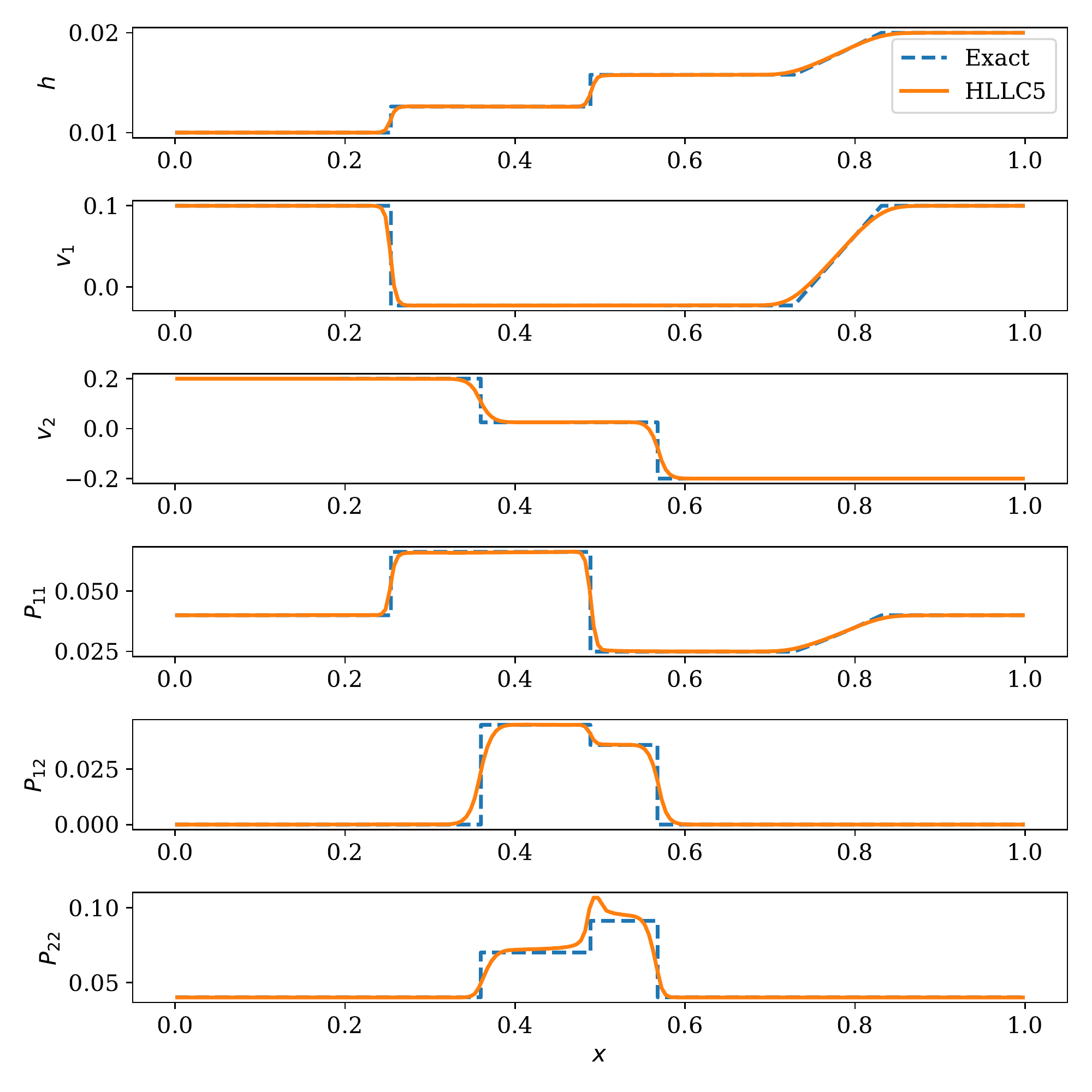}
\caption{Five waves  dam break  test case with 200 cells and second order approximations. Comparison between exact and numerical solutions obtained with HLL (left) and HLLC(right) schemes. }
\label{fig:Five}
\end{center}
\end{figure}
 \begin{figure}
\begin{center}
\includegraphics[width=0.49\textwidth]{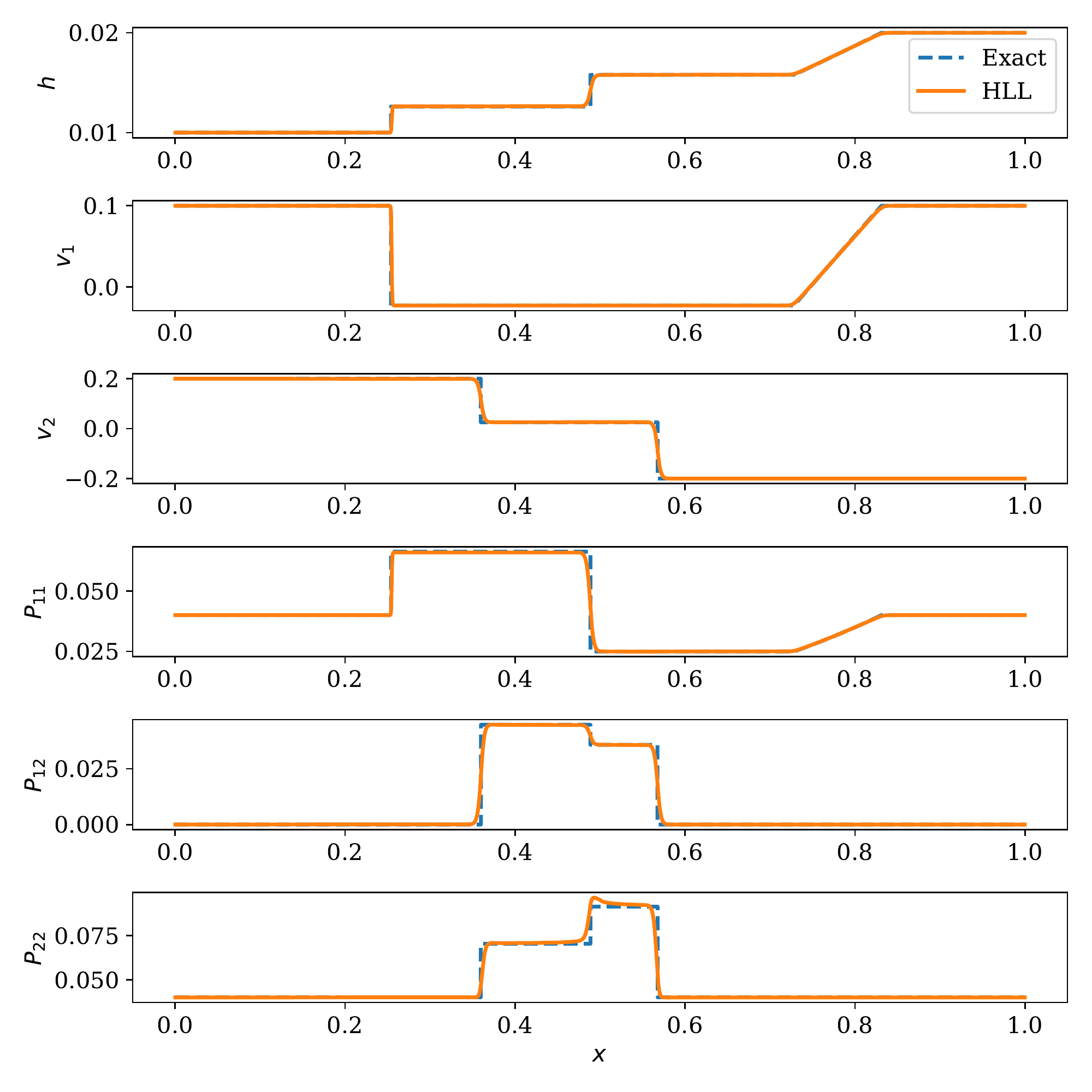}
\includegraphics[width=0.49\textwidth]{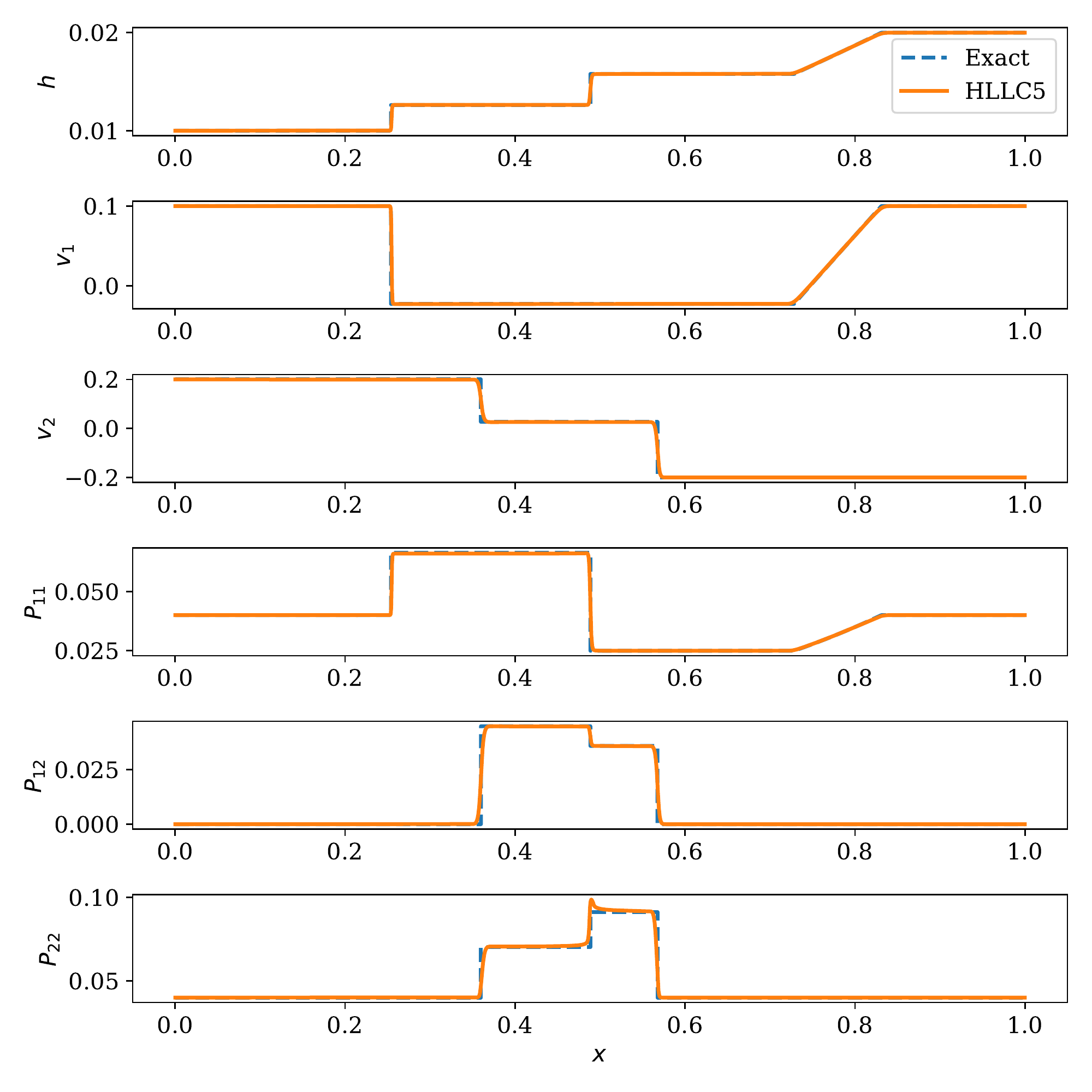}
\caption{Five waves dam break test case with  2000 cells and second order approximations. Comparison between exact and numerical solutions obtained with HLL (left) and HLLC(right) schemes. }
\label{fig:FiveF}
\end{center}
\end{figure}
 \subsection{Shear waves problem}
The initial condition for the Riemann problem is given in the following table.
\begin{center}
\begin{tabular}{|c|c|c|c|c|c|c|}
\hline
 & $h$ & $v_1$ & $v_2$ & $\p_{11}$ & $\p_{12}$ & $\p_{22}$ \\
\hline
$x < 0.5$ & 0.01 & 0.0 & 0.2 & $10^{-4}$ & $0.0$ & $10^{-4}$ \\
\hline
$x > 0.5$ & 0.01 & 0.0 & -0.2 &  $10^{-4}$ & $0.0$ & $10^{-4}$ \\
\hline
\end{tabular}
\end{center}
The result is shown in Figure~\ref{fig:Shear} at time $t=10$ on a mesh of 200 cells, where we see two shear waves in the solution. The numerical solution including the location
of the waves agrees well with the exact solution. The HLLC5 solver
gives a better resolution of the shear waves since they are included
in the approximate wave model. However, there are spurious spikes
found at the center in $\p_{22}$ where there is a stationary contact
discontinuity. This behavior is similar to what is usually observed with
numerical solution of some Riemann problems for the compressible Euler flows.
 \begin{figure}
\begin{center}
\includegraphics[width=0.49\textwidth]{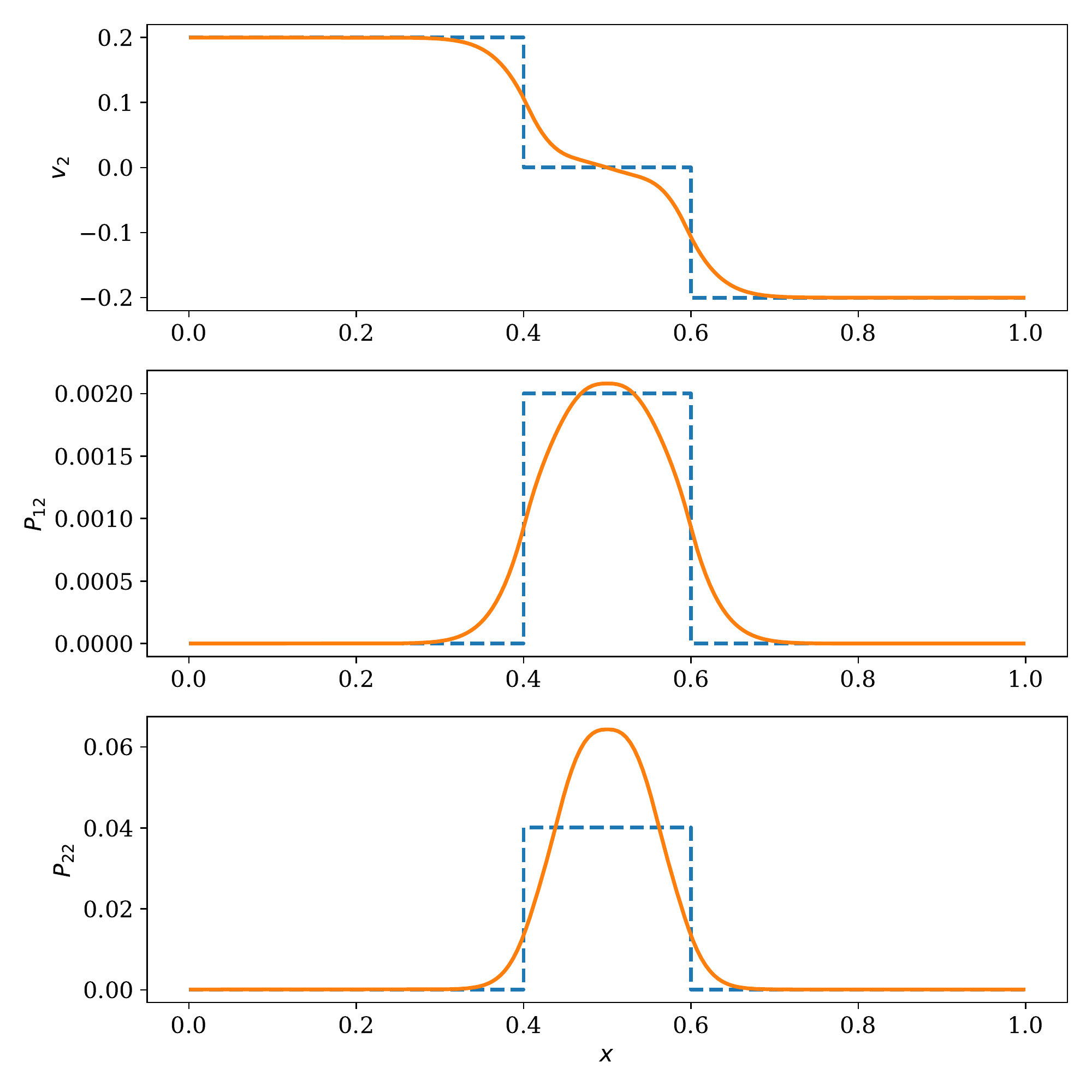}
\includegraphics[width=0.49\textwidth]{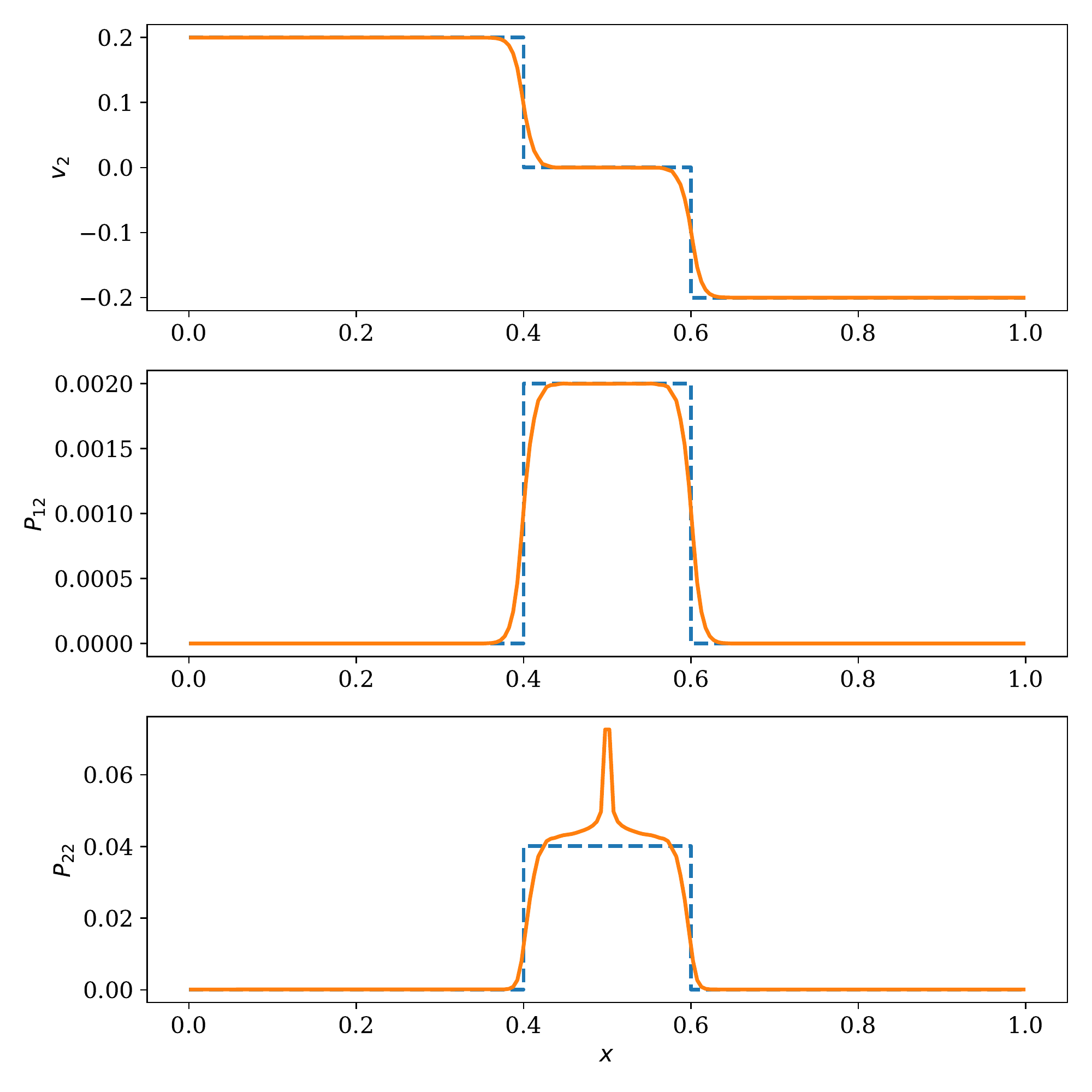}
\caption{Shear test case  with 200 cells and second order
  approximations. Comparison between exact and numerical solutions
  obtained with HLL (left) and HLLC5 (right) schemes. }
\label{fig:Shear}
\end{center}
\end{figure}
 \subsection{Single shock wave problem}
In this test case, we use a Riemann data for which the exact solution
consists of a single shock wave, as described in
Section~\ref{sec:singshock}. The initial condition is given by
\begin{center}
\begin{tabular}{|c|c|c|c|c|c|c|}
\hline
 & $h$ & $v_1$ & $v_2$ & $\p_{11}$ & $\p_{12}$ & $\p_{22}$ \\
\hline
$x < 0.5$ & 0.02 & 0 & 0 & $10^{-4}$ & 0 & $10^{-4}$ \\
\hline
$x > 0.5$ & 0.03 & -0.22169799277395363 & 0 & 0.016616666666666658 & 0 & $10^{-4}$ \\
\hline
\end{tabular}
\end{center}
Figure~\ref{fig:singshock} shows the numerical solution obtained with the
HLLC5 solver on a mesh of 2000 cells. While the shock location matches
closely, we see that the numerical solutions exhibit an extra contact
wave which is not present in the exact solution. All solvers exhibit
this behavior and this is seen even under grid refinement. This
situation is similar to the dambreak problem where the solution of
$\p_{11}$ does not agree with the exact solution.

We next consider the same problem but solve it in a frame where the exact shock is stationary. The corresponding Riemann data is given by
\begin{center}
\begin{tabular}{|c|c|c|c|c|c|c|}
\hline
 & $h$ & $v_1$ & $v_2$ & $\p_{11}$ & $\p_{12}$ & $\p_{22}$ \\
\hline
$x < 0.5$ & 0.02 & 0.6650939783218609 & 0 & $10^{-4}$ & 0 & $10^{-4}$ \\
\hline
$x > 0.5$ & 0.03 & 0.44339598554790727  & 0 & 0.016616666666666658 & 0 & $10^{-4}$ \\
\hline
\end{tabular}
\end{center}
We solve this problem using the speed estimates given in~\eqref{eq:speedapp} and also using the exact speeds obtained from the exact Riemann solver. Figure~\ref{fig:singshock0} shows the two sets of results on a mesh of 2000 cells; with the approximate speeds, we see a similar wave pattern as in the moving shock case, but there are many dispersive waves seen between the shock and the contact, as seen in the bottom figure which shows a zoomed view of $\p_{11}$. When the exact speeds are used, as shown in the right of Figure~\ref{fig:singshock0}, we see a better agreement with the exact solution but there are still some extra waves present in the numerical solution. Figure~\ref{fig:singshock0b} shows the results obtained with a refined mesh of 10000 cells. The numerical solver based on approximated wave speeds behaves almost as a dispersive shock that is usually associated to modulated wave-train.  It seems that, as soon as the shock cannot be numerically resolved without any dissipation, the numerical solution can be different from the analytical one. In other words, the shocks obtained with dissipative numerical schemes and those obtained analytically with the same generalized jump conditions, do not perfectly coincide.  This problem of convergence failure has been analyzed in \cite{Castro2008a} using the modified equation.  It was shown that this non-intuitive behavior is due to numerical viscosity and/or numerical dispersion.  Therefore, as far as the numerical scheme involves some dissipation, they will converge to a solution that depend, not only on the chosen path family, but also and especially on the specific form of its dissipation terms, whereas the analytical solution will be determined only by the choice of the path family. This discrepancy between the numerical and analytical solutions is one of the peculiarities of non-conservative systems. The results we obtain here, plotted for example on Figures~\ref{fig:CDamBMM}, \ref{fig:singshock0}, \ref{fig:singshock0b} and \ref{fig:singshock1}, support the overall trend described in~\cite{Castro2008a}. The two first order results obtained with approximate and exact wave speeds almost coincide in Figure~\ref{fig:singshock1} and we cannot visually distinguish them. However, contrary to the first order accurate scheme in Figure~\ref{fig:singshock1}, when the second order method is used in Figures~\ref{fig:singshock0}, \eqref{fig:singshock0b}  and \eqref{fig:singshock1}, we can observe on the variable $\p_{11}$ a wave train, going to the right, generated at the location of the stationary shock. The structure of this wave train is different depending on whether the wave velocities used in the Riemann solver are exact or approximate. This suggests that, in this context, the numerical diffusion becomes residual and we probably observe here a behavior specific to numerical schemes whose modified equations are dominated by dispersion \cite{Gavrilyuk2020}.  This trend will be analyzed and quantified in future work.

\begin{figure}
\begin{center}
\includegraphics[width=0.49\textwidth]{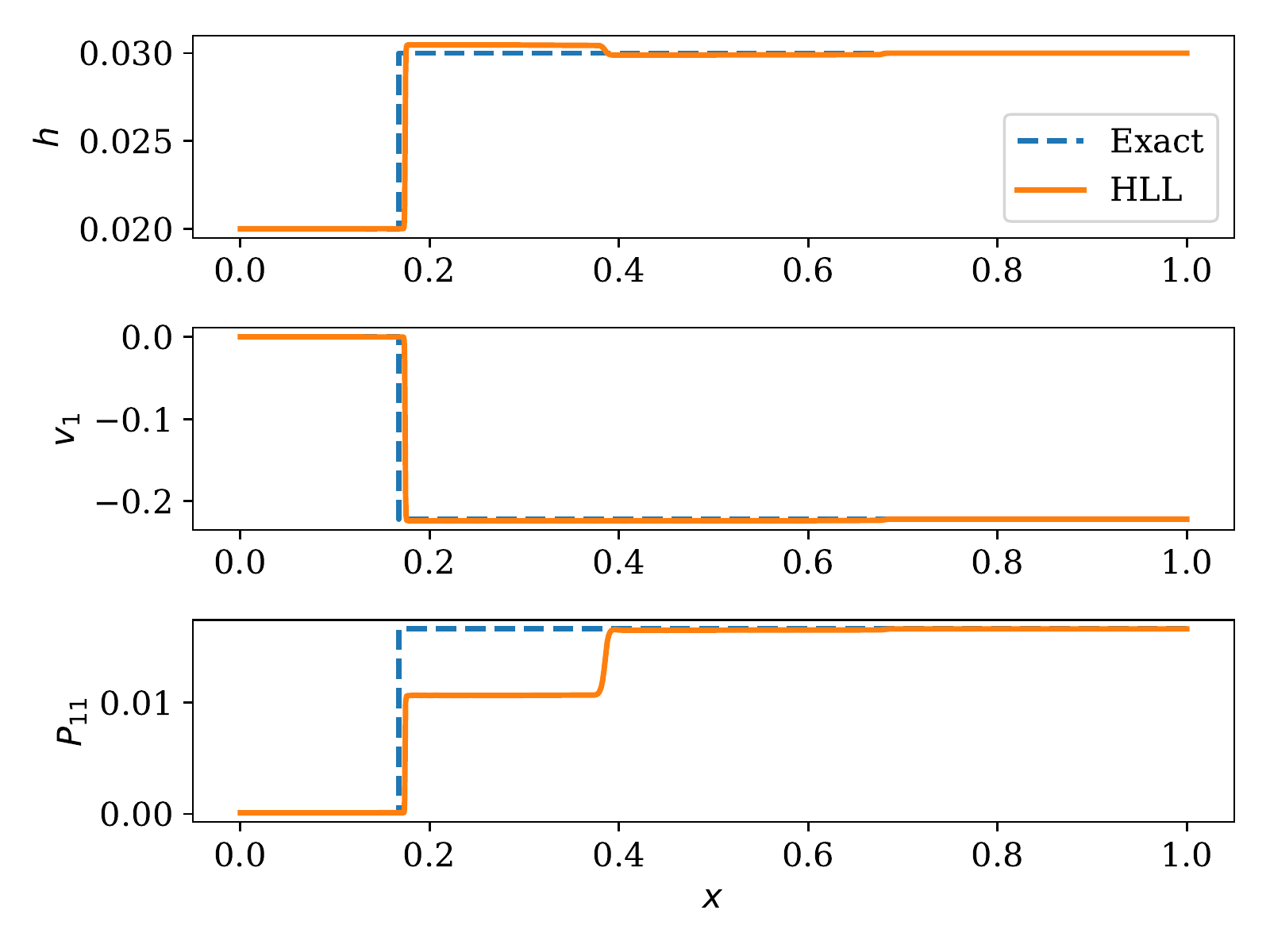}
\includegraphics[width=0.49\textwidth]{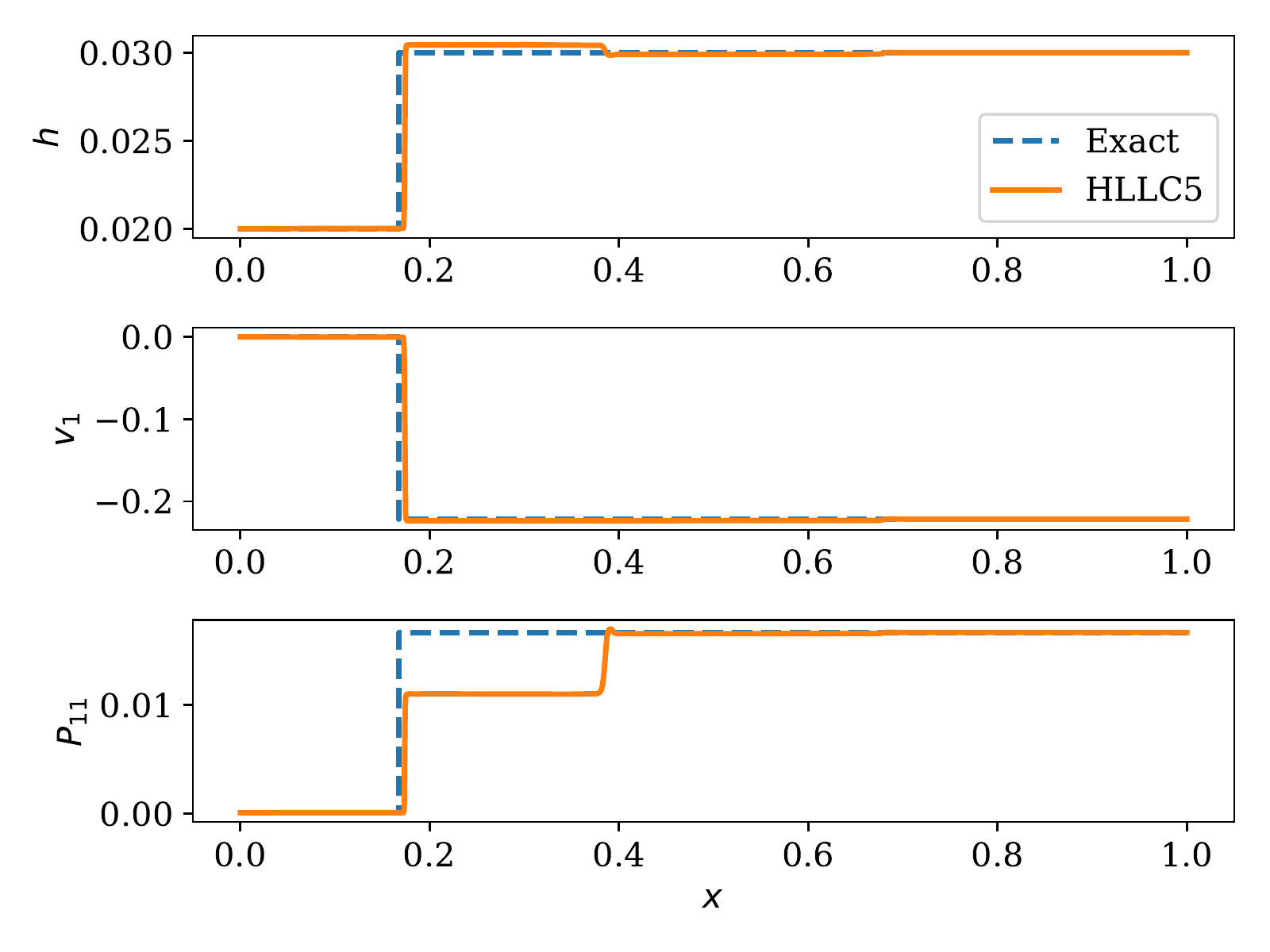}
\end{center}
\caption{Single moving shock test case on 2000 cells}
\label{fig:singshock}
\end{figure}

\begin{figure}
\begin{center}
\includegraphics[width=0.49\textwidth]{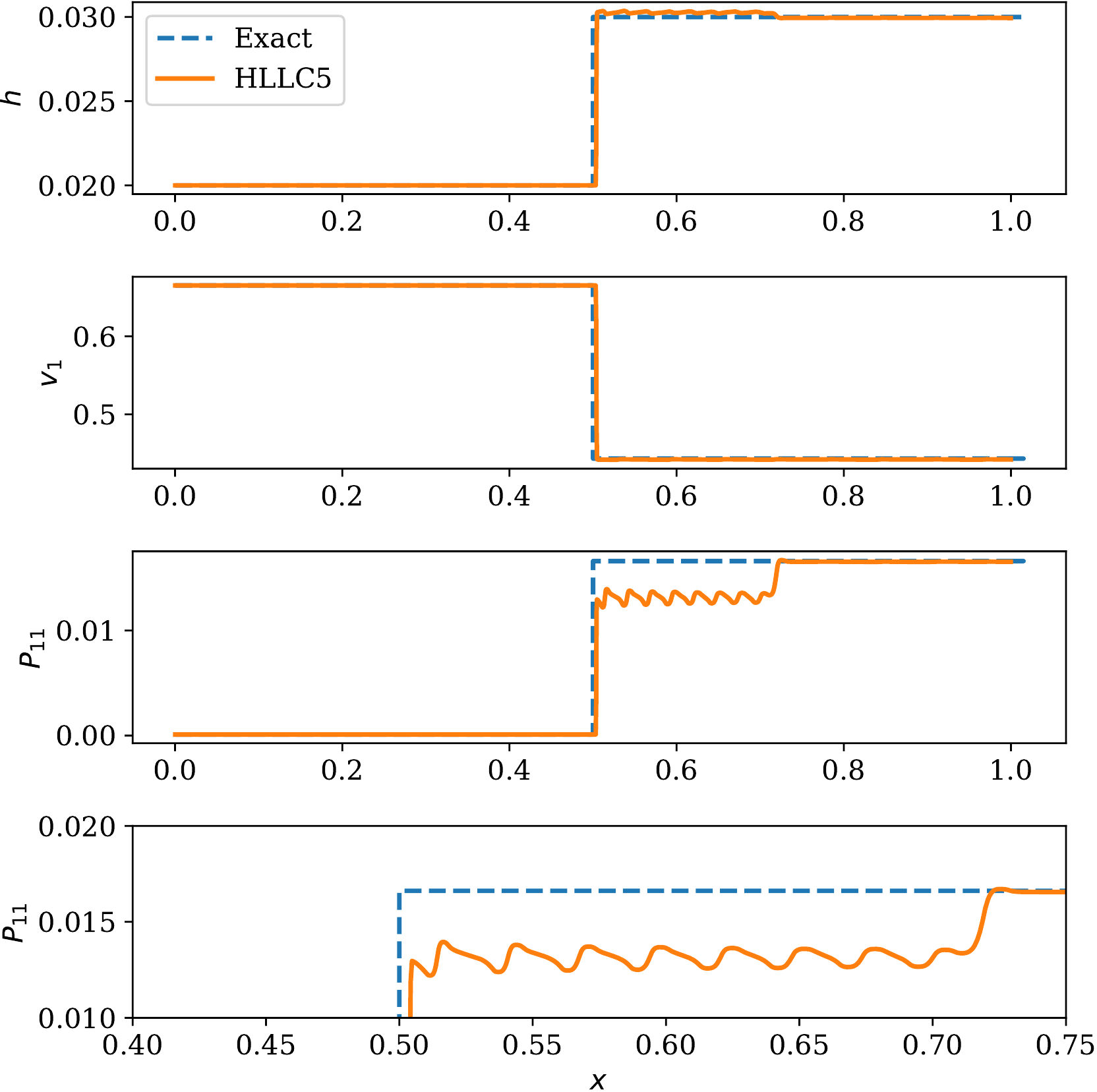}
\includegraphics[width=0.49\textwidth]{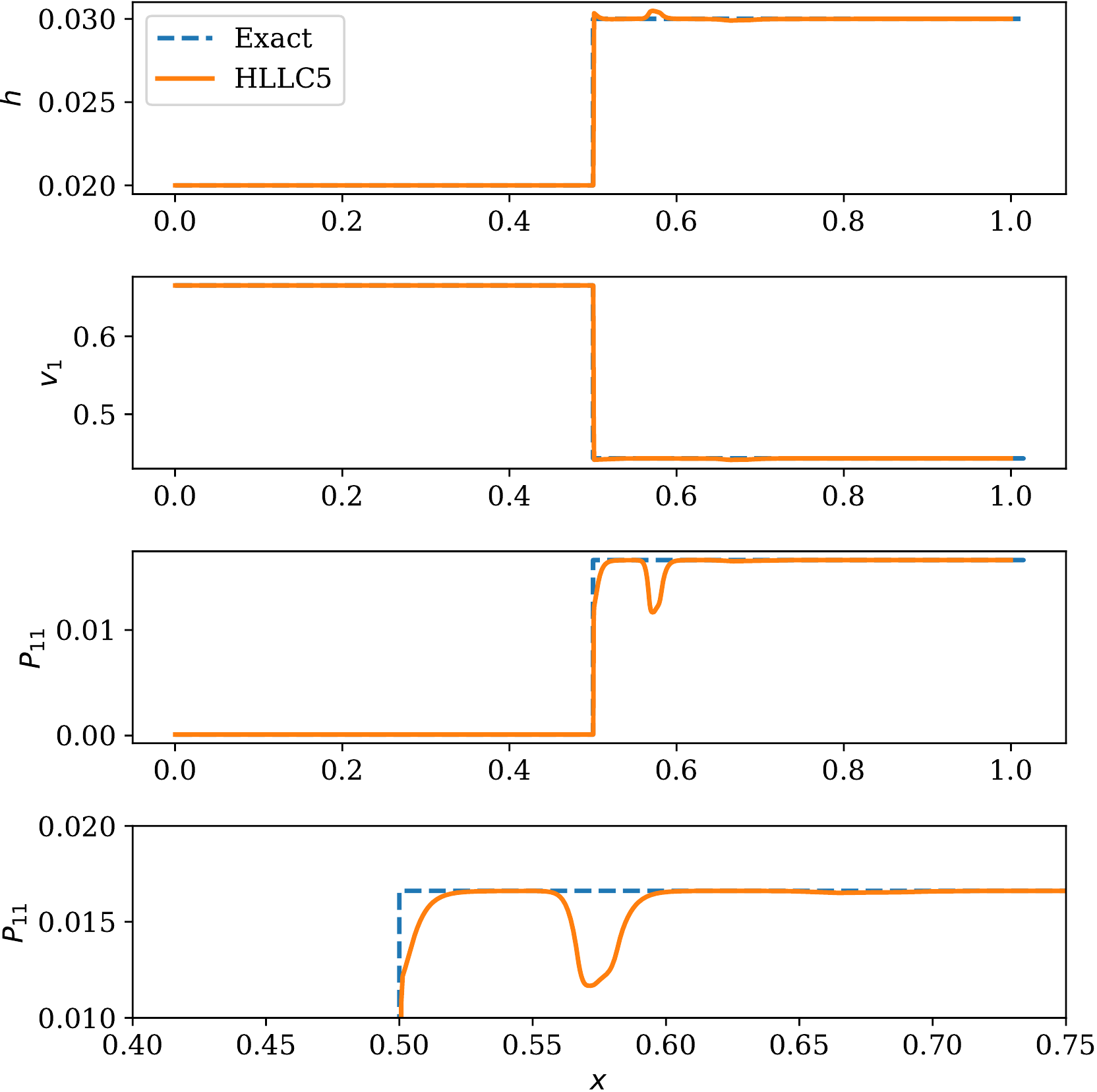}
\end{center}
\caption{Single stationary shock test case on 2000 cells using approximate speed (left) and exact speeds (right).}
\label{fig:singshock0}
\end{figure}

\begin{figure}
\begin{center}
\includegraphics[width=0.49\textwidth]{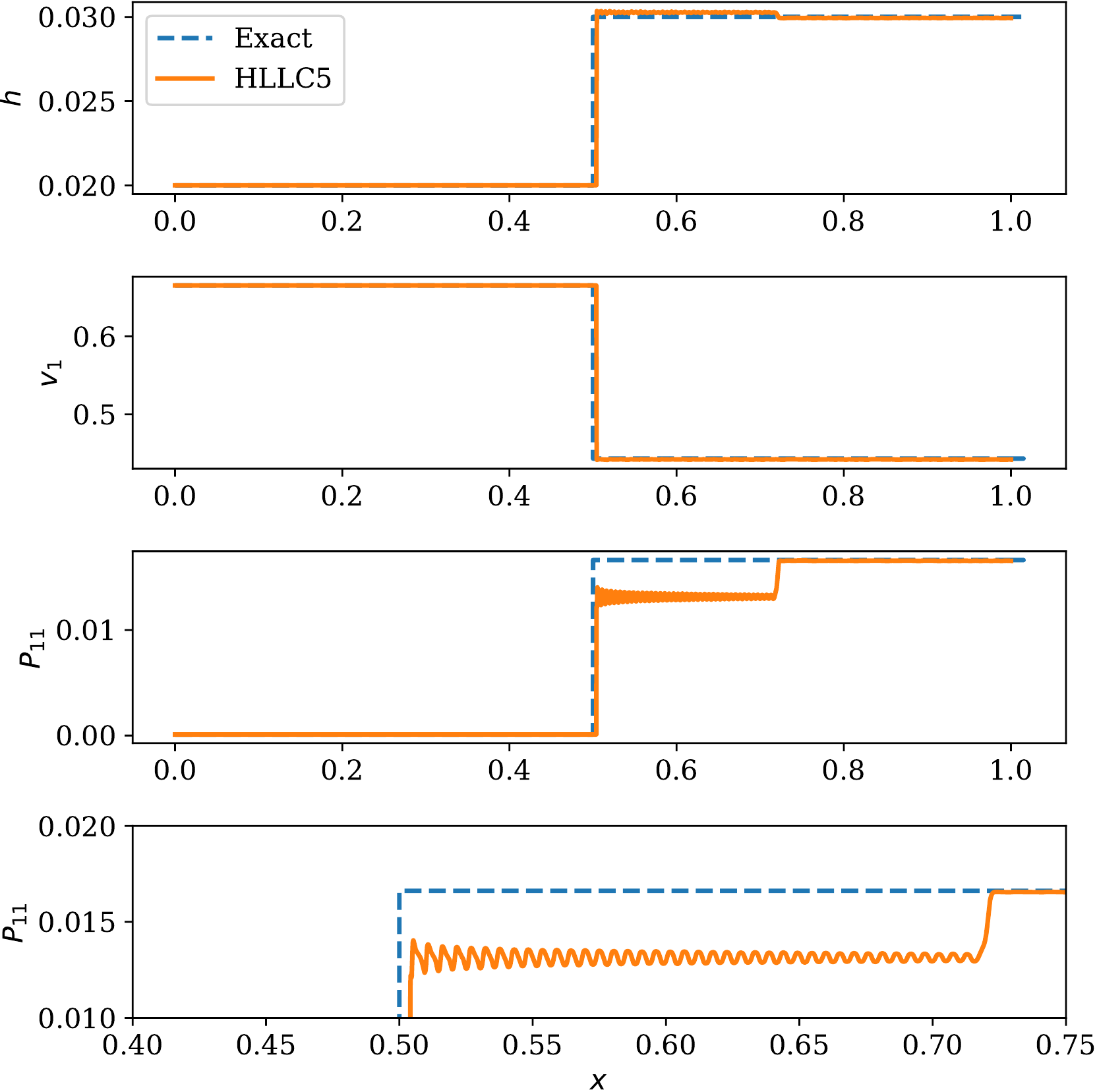}
\includegraphics[width=0.49\textwidth]{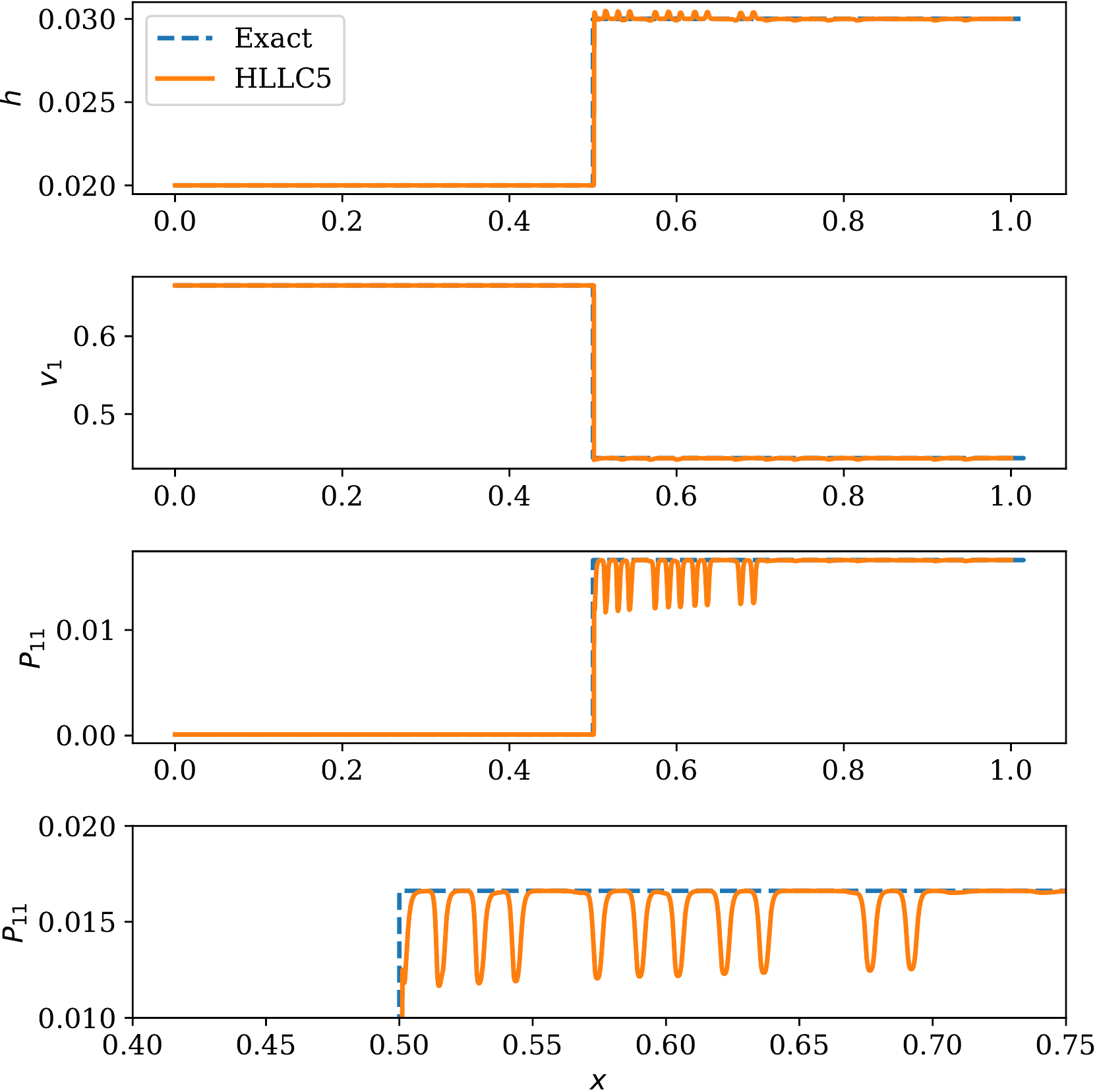}
\end{center}
\caption{Single stationary shock test case on 10000 cells using approximate speed (left) and exact speeds (right).}
\label{fig:singshock0b}
\end{figure}

\begin{figure}
\centering
\includegraphics[width=0.70\textwidth]{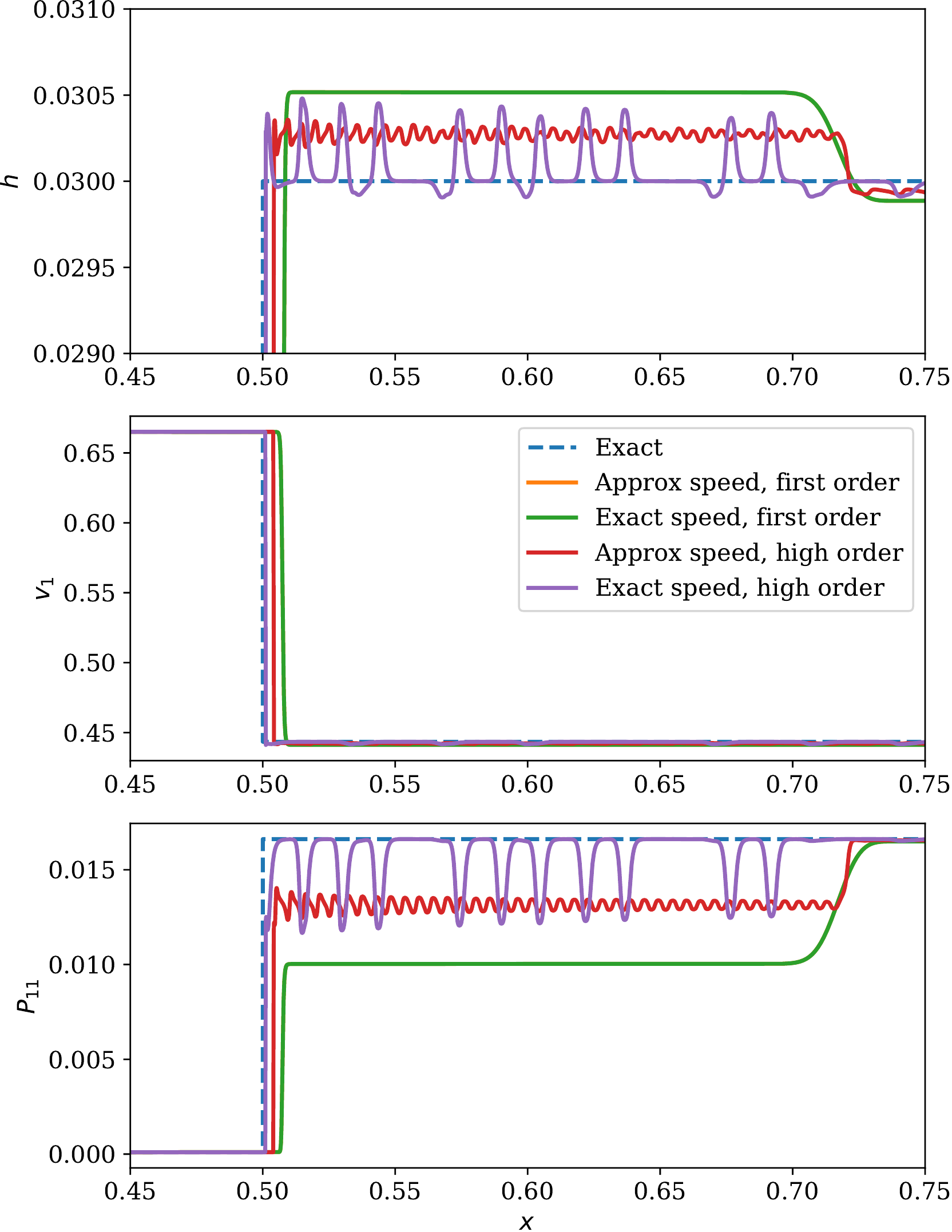}
\caption{Single stationary shock test case using 10000 cells; first and second order schemes using approximate and exact speeds.}
\label{fig:singshock1}
\end{figure}
\subsection{Single contact wave problem}
The Riemann data for this problem is given by
\begin{center}
\begin{tabular}{|c|c|c|c|c|c|c|}
\hline
 & $h$ & $v_1$ & $v_2$ & $\p_{11}$ & $\p_{12}$ & $\p_{22}$ \\
\hline
$x < 0.5$ & 0.02 & 0.1 & 0 & $10^{-4}$ & 0 & $10^{-4}$ \\
\hline
$x > 0.5$ & 0.01 & 0.1 & 0 & 0.014735 & 0 & $2 \times 10^{-4}$ \\
\hline
\end{tabular}
\end{center}
which gives rise to a single contact wave in the exact solution. Since
the water depth $h$ has a jump, the non-conservative terms are
non-zero in this case.  Figure~\ref{fig:singcont} shows the solution
at time $t=2.5$ obtained using the three approximate Riemann
solvers on a mesh of 2000 cells. The solution and the location of the
contact wave is captured well by all the numerical schemes. The HLL
solver introduces more numerical dissipation since it does not
explicitly model the contact wave, while both HLLC3 and HLLC5 solvers
include this wave in their model and give very similar
results. Contrary to the simple shock case, the numerical and
analytical contact discontinuity coincide perfectly, despite the
numerical approximation of the path (several points in the
numerical discontinuity). Indeed, the contact discontinuity is a
linearly degenerate wave and its associated states are defined
by the Riemann invariants that we have obtained explicitly and
independently of the path.
\begin{figure}
\begin{center}
\includegraphics[width=0.5\textwidth]{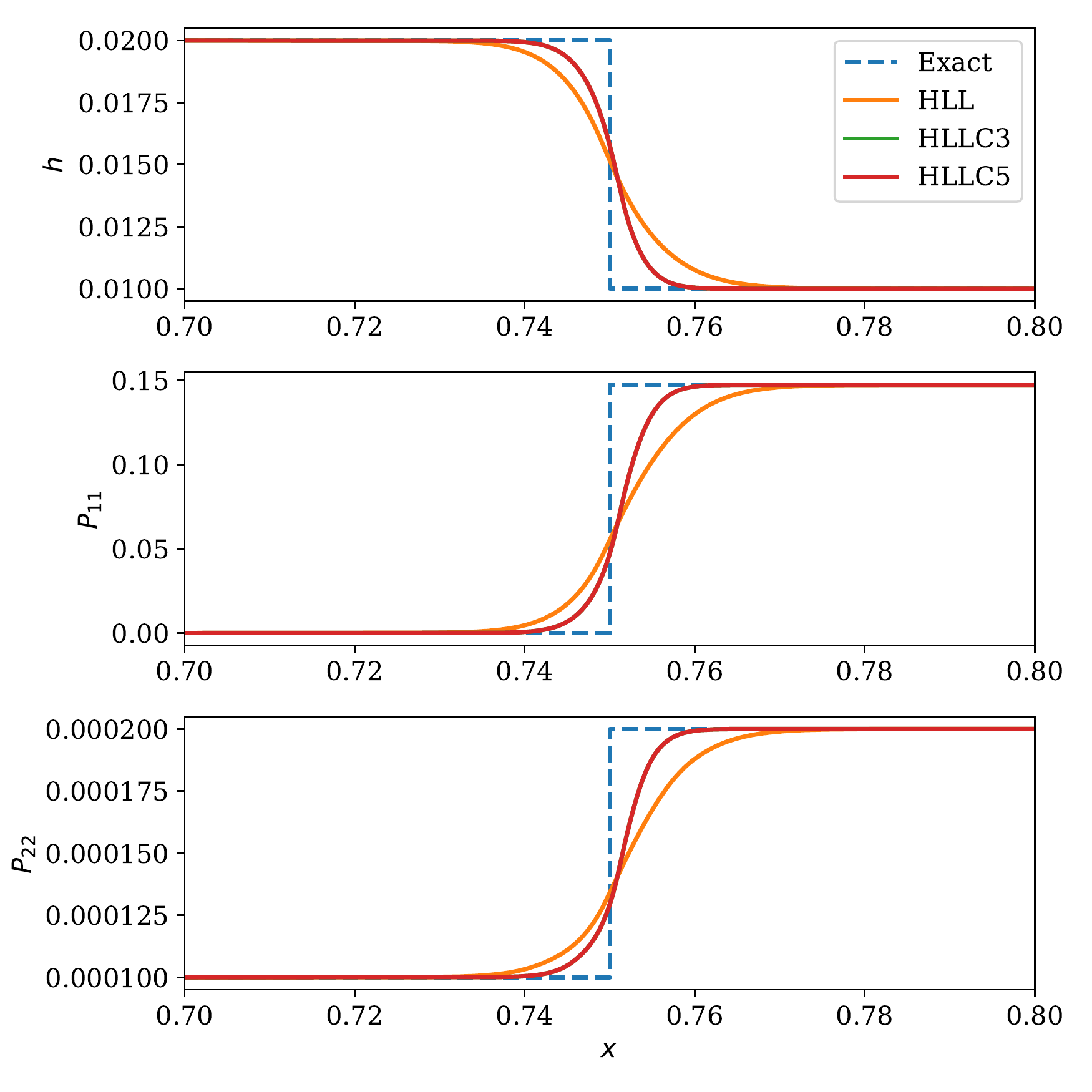}
\end{center}
\caption{Single contact test case on 2000 cells. Comparison of three different Riemann solvers with the exact solution.}
\label{fig:singcont}
\end{figure}
\section{Summary and conclusions}
We have derived the exact solution of the Riemann problem for the
non-conservative model of shear shallow water flows. The PDE is written in an almost conservative form that is very close to the 10-moment model for gas dynamics and admits a convex entropy function. The notion of
solution is based on path conservative approach for which a path has
to be assumed. In the numerical approaches, a linear path in the space
of conserved variables is usually assumed and we derive the exact
solution for this linear path. Several test cases are given and the
numerical results are compared with the exact solution. In some
problems as in dam break case, we see that the numerical solution of
$\p$ does not agree with the exact solution, though wave locations are
predicted correctly. When the stress levels $\p$ are not too small,
the agreement is much better as seen in the modified dam break
problem. In the case of single shock problem, the numerical solutions
produce an extra contact wave; the solutions depend sensitively on the
choice of the speed estimates used in the HLL solvers. When the exact
speeds are used for the stationary shock problem, there is better
agreement but we observe several other waves which may indicate that
the numerical strategies in this context are, locally around the shocks, dominated by
 dispersion rather than dissipation.

We must also remember that the exact solution depends on the choice of
the path and even when the path is fixed, the numerical shocks my be
different to the exact one. Nevertheless, apart from the shocks
profile for which significant differences are observed, for all other
waves (contact discontinuity, shear waves and rarefaction waves) the
numerical solution converges well to the exact solution. The
difficulties observed for the shock waves raise the problem of the
stability of the shocks in the framework of non-conservative hyperbolic
equations, for a given path (approximated Rankine Hugoniot
conditions). To give a solid explanation, it will probably be necessary
to carry out a thorough study of the stability of {\em non-conservative
shocks}, which is beyond the scope of this paper.
\section*{Acknowledgments}
The authors were supported by the French government, through the {\em {\sc uca-jedi} Investments in the Future} project managed by the National Research Agency (ANR) with the reference number ANR-15-IDEX-01. Praveen Chandrashekar's work is supported by the Department of Atomic Energy,  Government of India, under project no.~12-R\&D-TFR-5.01-0520. Boniface Nkonga's work is also supported by the INRIA associated Team AMFoDUC.  The authors thank the anonymous reviewers whose comments helped to improve the paper.
\bibliographystyle{siam}
\bibliography{bbibtex}

\begin{thebibliography}{10}

\bibitem{Abgrall2010}
{\sc R.~Abgrall and S.~Karni}, {\em A comment on the computation of
  non-conservative products}, Journal of Computational Physics, 229 (2010),
  pp.~2759--2763.

\bibitem{Berthon2002}
{\sc C.~Berthon, F.~Coquel, J.~H{\'e}rard, and M.~Uhlmann}, {\em An approximate
  solution of the {{Riemann}} problem for a realisable second-moment turbulent
  closure}, Shock Waves, 11 (2002), pp.~245--269.

\bibitem{Bhole2019}
{\sc A.~Bhole, B.~Nkonga, S.~Gavrilyuk, and K.~Ivanova}, {\em Fluctuation
  splitting {{Riemann}} solver for a non-conservative modeling of shear shallow
  water flow}, Journal of Computational Physics, 392 (2019), pp.~205--226.

\bibitem{Busto2021}
{\sc S.~Busto, M.~Dumbser, S.~Gavrilyuk, and K.~Ivanova}, {\em On
  {{Thermodynamically Compatible Finite Volume Methods}} and
  {{Path-Conservative ADER Discontinuous Galerkin Schemes}} for {{Turbulent
  Shallow Water Flows}}}, Journal of Scientific Computing, 88 (2021), p.~28.

\bibitem{Castro2008a}
{\sc M.~J. Castro, P.~G. LeFloch, M.~L. {Mu{\~n}oz-Ruiz}, and C.~Par{\'e}s},
  {\em Why many theories of shock waves are necessary: {{Convergence}} error in
  formally path-consistent schemes}, Journal of Computational Physics, 227
  (2008), pp.~8107--8129.

\bibitem{Castro2012}
{\sc M.~J. Castro, C.~Par{\'e}s, G.~Puppo, and G.~Russo}, {\em Central
  {{Schemes}} for {{Nonconservative Hyperbolic Systems}}}, SIAM Journal on
  Scientific Computing, 34 (2012), pp.~B523--B558.

\bibitem{CastroDiaz2019}
{\sc M.~J. Castro~D{\'i}az, A.~Kurganov, and T.~{Morales de Luna}}, {\em
  Path-conservative central-upwind schemes for nonconservative hyperbolic
  systems}, ESAIM: Mathematical Modelling and Numerical Analysis, 53 (2019),
  pp.~959--985.

\bibitem{Cauret1989}
{\sc J.~Cauret, J.~Colombeau, and A.~Le~Roux}, {\em Discontinuous generalized
  solutions of nonlinear nonconservative hyperbolic equations}, Journal of
  Mathematical Analysis and Applications, 139 (1989), pp.~552--573.

\bibitem{Chandrashekar2020}
{\sc P.~Chandrashekar, B.~Nkonga, A.~K. Meena, and A.~Bhole}, {\em A path
  conservative finite volume method for a shear shallow water model}, Journal
  of Computational Physics, 413 (2020), p.~109457.

\bibitem{Colombeau1988}
{\sc J.~F. Colombeau and A.~Y. Le~Roux}, {\em Multiplications of distributions
  in elasticity and hydrodynamics}, Journal of Mathematical Physics, 29 (1988),
  pp.~315--319.

\bibitem{DalMaso1995}
{\sc G.~Dal~Maso, P.~G. Lefloch, and F.~Murat}, {\em Definition and weak
  stability of nonconservative products}, J. Math. Pures Appl., 74 (1995),
  pp.~483--548.

\bibitem{Dumbser2016}
{\sc M.~Dumbser and D.~S. Balsara}, {\em A new efficient formulation of the
  {{HLLEM Riemann}} solver for general conservative and non-conservative
  hyperbolic systems}, Journal of Computational Physics, 304 (2016),
  pp.~275--319.

\bibitem{Dumbser2009}
{\sc M.~Dumbser, M.~Castro, C.~Par{\'e}s, and E.~F. Toro}, {\em {{ADER}}
  schemes on unstructured meshes for nonconservative hyperbolic systems:
  {{Applications}} to geophysical flows}, Computers \& Fluids, 38 (2009),
  pp.~1731--1748.

\bibitem{Einfeldt1988}
{\sc B.~Einfeldt}, {\em On {{Godunov-Type Methods}} for {{Gas Dynamics}}}, SIAM
  Journal on Numerical Analysis, 25 (1988), pp.~294--318.

\bibitem{Gavrilyuk2018}
{\sc S.~Gavrilyuk, K.~Ivanova, and N.~Favrie}, {\em Multi-dimensional shear
  shallow water flows: {{Problems}} and solutions}, Journal of Computational
  Physics, 366 (2018), pp.~252--280.

\bibitem{Gavrilyuk2020}
{\sc S.~Gavrilyuk, B.~Nkonga, K.-M. Shyue, and L.~Truskinovsky}, {\em
  Stationary shock-like transition fronts in dispersive systems}, Nonlinearity,
  33 (2020), pp.~5477--5509.

\bibitem{Godlewski1996}
{\sc E.~Godlewski and P.-A. Raviart}, {\em Numerical {{Approximation}} of
  {{Hyperbolic Systems}} of {{Conservation Laws}}}, vol.~118 of Applied
  {{Mathematical Sciences}}, {Springer New York}, {New York, NY}, 1996.

\bibitem{Gosse2001}
{\sc L.~Gosse}, {\em A well-balanced scheme using non-conservative products
  designed for hyperbolic systems of conservation laws with source terms},
  Mathematical Models and Methods in Applied Sciences, 11 (2001), pp.~339--365.

\bibitem{Joseph2003}
{\sc K.~Joseph and P.~Sachdev}, {\em Exact solutions for some non-conservative
  hyperbolic systems}, International Journal of Non-Linear Mechanics, 38
  (2003), pp.~1377--1386.

\bibitem{Lax1960}
{\sc P.~Lax and B.~Wendroff}, {\em Systems of conservation laws},
  Communications on Pure and Applied Mathematics, 13 (1960), pp.~217--237.

\bibitem{Levermore1996}
{\sc C.~D. Levermore}, {\em Moment closure hierarchies for kinetic theories},
  Journal of Statistical Physics, 83 (1996), pp.~1021--1065.

\bibitem{Levermore1998}
{\sc C.~D. Levermore and W.~J. Morokoff}, {\em The {{Gaussian Moment Closure}}
  for {{Gas Dynamics}}}, SIAM Journal on Applied Mathematics, 59 (1998),
  pp.~72--96.

\bibitem{Pares2006}
{\sc C.~Par{\'e}s}, {\em Numerical methods for nonconservative hyperbolic
  systems: A theoretical framework.}, SIAM Journal on Numerical Analysis, 44
  (2006), pp.~300--321.

\bibitem{Pares2019}
{\sc C.~Par{\'e}s and E.~Pimentel}, {\em The {{Riemann}} problem for the
  shallow water equations with discontinuous topography: {{The}} wet\textendash
  dry case}, Journal of Computational Physics, 378 (2019), pp.~344--365.

\bibitem{Schneider2021}
{\sc K.~A. Schneider, J.~M. Gallardo, D.~S. Balsara, B.~Nkonga, and
  C.~Par{\'e}s}, {\em Multidimensional approximate {{Riemann}} solvers for
  hyperbolic nonconservative systems. {{Applications}} to shallow water
  systems}, Journal of Computational Physics, 444 (2021), p.~110547.

\bibitem{Teshukov2007}
{\sc V.~M. Teshukov}, {\em Gas-dynamic analogy for vortex free-boundary flows},
  Journal of Applied Mechanics and Technical Physics, 48 (2007), pp.~303--309.

\bibitem{Toro2001b}
{\sc E.~F. Toro}, {\em Shock-Capturing Methods for Free-Surface Shallow Flows},
  {Wiley-Blackwell}, 2001.

\bibitem{Toumi1992}
{\sc I.~Toumi}, {\em A weak formulation of roe's approximate riemann solver},
  Journal of Computational Physics, 102 (1992), pp.~360--373.

\bibitem{Volpert1967}
{\sc A.~I. Volpert}, {\em The spaces {{BV}} and quasilinear equations},
  Mathematics of the USSR-Sbornik, 2 (1967), pp.~225--267.

\end{thebibliography}
\appendix
\section{Numerical solution of root finding problem}
\label{sec:root}
The solution of \eqref{eq:HugoniotRP} is obtained numerically by applying a Newton method. The algorithm for the Newton method is as follows. Define $z=(z_{\lf},z_{\rg})$ and $H(z) = [F(z), G(z)]^\top$.  Set the tolerance $\epsilon = 10^{-6}$.  We start at the point $z=(1,1)$.
\begin{enumerate}
\item If $|F(z)| < \epsilon$ and $|G(z)| < \epsilon$, then stop.
\item Solve $H'(z) \Delta z = - H(z)$
\item Set $z_{\lf} = z_{\lf} + \frac{1}{2^n} \Delta z_{\lf}$ with smallest $n \in \{ 0,1,2,\ldots\}$ such that $z_{\lf} \in (0,2)$.
\item Set $z_{\rg} = z_{\rg} + \frac{1}{2^n} \Delta z_{\rg}$ with smallest $n \in \{ 0,1,2,\ldots\}$ such that $z_{\rg} \in (0,2)$.
\item Go to Step 1
\end{enumerate}
\end{document}